\documentclass[aap,MSNbibl,seceqn,dvips]{arximspdf}
\usepackage{mathbh}

\doi{10.1214/13-AAP968} 
\volume{24}
\issue{5}
\pubyear{2014}
\firstpage{1946}
\lastpage{1993}

\makeatletter
\newcommand{\rrVert}{\Vert}
\newcommand{\rrvert}{\vert}
\newcommand{\llVert}{\Vert}
\newcommand{\llvert}{\vert}
\newcommand{\eqref}[1]{(\ref{#1})}
\newtheorem{theorem}{Theorem}[section]
\newtheorem{lemma}[theorem]{Lemma}
\newtheorem{proposition}[theorem]{Proposition}

\newproclaim{remark}[theorem]{Remark}
\newproclaim{assmp}[theorem]{Assumption}
\newproclaim{definition}[theorem]{Definition}

\newcommand{\tta}{\theta}
\newcommand{\gL}{\Lambda}
\def\sig{\sigma}
\def\Om{\Omega}
\def\vphi{\varphi}

\renewcommand{\div}{\mathrm{div}}

\newcommand{\bE}{\mathbf{E}}
\newcommand{\bP}{\mathbf{P}}
\newcommand{\bR}{\mathbf{R}}
\newcommand{\bS}{\mathbf{S}}
\newcommand{\bZ}{\mathbf{Z}}

\newcommand{\bbT}{\mathbb{T}}

\newcommand{\cC}{\mathcal{C}}
\newcommand{\cD}{\mathcal{D}}
\newcommand{\cE}{\mathcal{E}}
\newcommand{\cI}{\mathcal{I}}
\newcommand{\cJ}{\mathcal{J}}
\newcommand{\cM}{\mathcal{M}}
\newcommand{\cX}{\mathcal{X}}

\newcommand{\tc}{\tilde{c}}
\newcommand{\tp}{\tilde{p}}
\newcommand{\tv}{\tilde{v}}

\newcommand{\tA}{\tilde{A}}

\makeatother

\begin{document}
\begin{frontmatter}

\title{Mean field limit for disordered diffusions with singular
interactions\thanksref{T2}}
\runtitle{Diffusions with singular interactions}

\begin{aug}
\author[a]{\fnms{Eric} \snm{Lu\c con}\corref{}\ead[label=e1]{eric.lucon@parisdescartes.fr}}
\and
\author[b]{\fnms{Wilhelm} \snm{Stannat}\ead[label=e2]{stannat@math.tu-berlin.de}}

\thankstext{T2}{Supported by the BMBF, FKZ 01GQ 1001B.}
\runauthor{E. Lu\c con and W. Stannat}

\affiliation{Technische Universit\"at, Berlin and  Bernstein Center\break for Computational Neuroscience}

\address[a]{Institut f\"ur Mathematik\\
Technische Universit\"at Berlin\\
Stra{\ss}e des 17. Juni 136\\
D-10623 Berlin\\
Germany\\
and\\
Bernstein Center for\\
\quad Computational Neuroscience\\
Philippstr. 13\\
D-10115 Berlin\\
Germany\\
Current address:\\
Laboratoire MAP5\\
Universit\'e Paris Descartes\\
45 rue des Saints P\`eres\\
75270 Paris Cedex 06\\
France\\
\printead{e1}}

\address[b]{Institut f\"ur Mathematik\\
Technische Universit\"at Berlin\\
Stra{\ss}e des 17. Juni 136\\
D-10623 Berlin\\
Germany\\
and\\
Bernstein Center for\\
\quad Computational Neuroscience\\
Philippstr. 13\\
D-10115 Berlin\\
Germany\\
\printead{e2}}

\end{aug}

\received{\smonth{1} \syear{2013}}
\revised{\smonth{6} \syear{2013}}

%
\begin{abstract}
Motivated by considerations from neuroscience (macroscopic behavior of
large ensembles of interacting neurons), we consider a population of
mean field interacting diffusions in $\bR^{m}$ in the presence of a
random environment and with spatial extension: each diffusion is
attached to one site of the lattice $\bZ^{d}$, and the interaction
between two diffusions is attenuated by a spatial weight that depends
on their positions. For a general class of singular weights (including
the case already considered in the physical literature when
interactions obey to a power-law of parameter $0<\alpha<d$), we address
the convergence as $N\to\infty$ of the empirical measure of the
diffusions to the solution of a deterministic McKean--Vlasov equation
and prove well-posedness of this equation, even in the degenerate case
without noise. We provide also precise estimates of the speed of this
convergence, in terms of an appropriate weighted Wasserstein distance,
exhibiting in particular nontrivial fluctuations in the power-law case
when $\frac{d}{2}\leq\alpha<d$. Our framework covers the case of
polynomially bounded monotone dynamics that are especially encountered
in the main models of neural oscillators.
\end{abstract}

%


\end{frontmatter}

\section{Introduction}\label{sec1}
The purpose of this paper is to provide a general convergence result
for the empirical distribution of spatially extended networks of mean
field coupled diffusions in a random environment. The main novelty of
the paper is to consider a family of interacting diffusions indexed by
the box $\gL_{N}:= [\![-N, \ldots, N]\!]^{d}$ of volume
$|\gL
_{N}|:=(2N+1)^{d}$ in the $d$-dimensional lattice $\bZ^{d}$ ($d\geq1$)
where the interaction between two diffusions in $\gL_{N}$ depends on
their relative positions. We are in particular interested in diffusions
modeling the spiking activity of neurons in a noisy environment. To
motivate the mathematical model we want to work with, let us consider,
as a particular example, a family of stochastic FitzHugh--Nagumo
neurons (see \cite{22657695,MR2674516} and references therein for
further neurophysiological insights on the model)
%
%
\begin{equation}
\label{eq:FHNintro} %
\cases{\displaystyle \mathrm{d}V_{i}(t) =
\biggl(V_{i}(t)-\frac{V_{i}(t)^{3}}{3} -w_{i}(t)+I \biggr)\,
\mathrm{d}t + \sig_{V} \,\mathrm{d}B_{i}^{V}(t),
\vspace*{2pt}
\cr
\displaystyle\mathrm{d}w_{i}(t) = \bigl(a_{i}
\bigl(b_{i} V_{i}(t)-w_{i}(t)\bigr) \bigr)\,
\mathrm{d}t + \sig _{w} \,\mathrm{d}B_{i}^{w}(t) }
\end{equation}
for $i\in\gL_{N}$, with exterior input current $I$. The variable
$V_{i}(t)$ denotes the voltage activity of the neuron, and $w_{i}(t)$
plays the role of a recovery variable. $  (B_{i}^{V}(t),
B_{i}^{w}(t) )$ are independent Brownian motions modeling exterior
stochastic forces. Depending on the parameters $(a_{i}, b_{i})\in\bR
^{2}$, the neurons exhibit an oscillatory, excitable or inhibitory
behavior. Suppose that the precise values of $\omega_{i}=(a_{i}, b_{i})$
are unknown, which will always be the case in real-world applications,
but rather are given as independent and identically distributed random
variables. From a point of view from statistical physics, this
additional randomness in \eqref{eq:FHNintro} may be considered as a
\emph{disorder}. For simplicity we suppose that the $\omega_{i}$ are
independent of the time $t$. Equation \eqref{eq:FHNintro} can be
written as
%
%
\begin{equation}
\mathrm{d}\tta_{i}(t)= c(\tta_{i}, \omega_{i})
\,\mathrm{d}t + \sig \cdot\mathrm{d} B_{i}(t),\qquad t\geq0, i\in
\gL_{N},
\end{equation}
using the shorthand notation $\tta=  (V, w )$, $\omega=(a,b)$,
$c(\tta, \omega) =  (V-\frac{V^{3}}{3} -w+I, a (b V-w) )$,
$B=
(B^{V}, B^{w} )$ and $\sig=
\bigl({
{\sig_{V} \atop 0}\enskip {0\atop\sig_{w}
}}\bigr)
$.
We suppose that the individual neurons are coupled with the help of a
possibly nonlinear and random coupling term $\Gamma (\tta_{i},
\omega
_{i}, \tta_{j}, \omega_{j} )$ ($i,j\in\gL_{N}$) modeling electrical
synapses between the neurons. The coupling intensity between neurons
$i$ and $j$ will depend additionally on some weight $\Psi_{N}(i, j)$
($\Psi_{N}$ may be thought as a function of the distance, but not
necessarily), so that the resulting system gets the following type:
%
%
\begin{eqnarray}
\label{eq:systintro} \mathrm{d}\tta_{i}(t)&=& c\bigl(\tta_{i}(t),
\omega_{i}\bigr)\,\mathrm{d}t
\nonumber\\
&&{}+ \frac{1}{|\gL
_{N}|}\sum_{j\in\gL_{N}} \Gamma \bigl(
\tta_{i}(t), \omega_{i}, \tta _{j}(t),
\omega_{j} \bigr) \Psi_{N}(i, j) \,\mathrm{d}t + \sig\cdot
\mathrm{d}B_{i}(t),\\
 \eqntext{t\geq0, i\in\gL_{N}.}
\end{eqnarray}
The purpose of the paper is to address the behavior of system \eqref
{eq:systintro} in large populations ($N\to\infty$), under general
assumptions on the dynamics $c$, the coupling $\Gamma$ and the spatial
constraint $\Psi_{N}$.

\subsection{Empirical measure and mean-field limit}
All the statistical information of the neural ensemble is contained in
its empirical distribution\break  of the diffusions $\tta_{j}$ (with disorder
$\omega_{j}$ and with renormalized position\vadjust{\goodbreak}\break  $x_{j}:= \frac{1}{2N}\in
[- \frac{1}{2}, \frac{1}{2} ]^{d}$)
%
%
\begin{equation}
\label{eq:nuNintro} \nu^{(N)}_{t}(\mathrm{d}\tta, \mathrm{d}
\omega, \mathrm{d}x):= \frac{1}{ \llvert
\gL_{N}
\rrvert } \sum_{j\in\gL_{N}}
\delta_{(\tta_{i}(t), \omega_{i},
x_{j})}(\mathrm{d}\tta, \mathrm{d}\omega, \mathrm{d}x),\qquad t\geq0
\end{equation}
that can be seen as a random probability measure.
%
%
\begin{remark}
\label{rem:renormalization}
The renormalization of the positions by $ \frac{1}{2N}$ maps $\gL
_{N}=[\![-N, \ldots, N]\!]^{d}$ to a discrete subset of
$
[- \frac{1}{2}, \frac{1}{2} ]^{d}$. The necessity of this
renormalization will become clear in the discussion on the spatial
constraints below in this \hyperref[sec1]{Introduction}.
\end{remark}
Since we are interested in the collective behavior of a large numbers
of neurons, as it is the case for neural ensembles in the brain,
understanding the asymptotic behavior of $\nu_{t}^{(N)}$ as $N\to
\infty
$ is important.

Under the assumption that
%
%
\begin{equation}
\label{eq:defPsi} \Psi_{N}(i, j)= \Psi \biggl(\frac{i}{2N},
\frac{j}{2N} \biggr)
\end{equation}
for a general class of functions $\Psi$ defined on $ [- \frac{1}{2},
\frac{1}{2} ]^{d} \times [- \frac{1}{2}, \frac{1}{2} ]^{d}$,
we prove, as part of our main results in this paper (see Theorems \ref
{theo:LLNPnn} and \ref{theo:LLNpowerinfd}), that $\nu_{t}^{(N)}$
converges to a deterministic measure $\nu_{t}(\mathrm{d}\tta,
\mathrm{d}\omega,
\mathrm{d}x)=
q_{t}(\tta, \omega, x)\,\mathrm{d}\tta\mu(\mathrm{d}\omega)\,\mathrm
{d}x$ where
$q_{t}$ is a weak
solution of the McKean--Vlasov equation
%
%
\begin{eqnarray}
\label{eq:MKVintro} \partial_{t}q_{t} &=& \frac{1}{2}
\div_{\tta} \bigl(\sig\sig ^{T}\nabla _{\tta}
q_{t} \bigr)\nonumber\\
&&{}- \div_{\tta} \biggl(q_{t} \biggl\{c(
\tta, \omega)\\
&&\hspace*{53pt}{} + \int\Gamma(\tta, \omega, \bar\tta,\bar\omega) \Psi(x, \bar
x)q_{t}(\bar\tta, \bar\omega, \bar x)\,\mathrm{d}\bar\tta\,\mathrm{d}
\mu(\bar\omega )\,\mathrm{d}\bar x \biggr\} \biggr).\nonumber
\end{eqnarray}
For a formal derivation of this equation, we refer to the end of
Section~\ref{subsec:MKV} below. The measure $\nu_{t}$ is called the
mean field limit of the system \eqref{eq:systintro}. Through
Theorems \ref{theo:LLNPnn} and \ref{theo:LLNpowerinfd}, we not only
prove the convergence $\nu^{(N)}_{t}$ toward $\nu_{t}$, but we also
provide some explicit estimates on the speed of convergence in terms of
an appropriate weighted Wasserstein distance.

\subsection{Existing literature and motivations}
\subsubsection{The nonspatial case: \texorpdfstring{$\Psi_{N}\equiv1$}{PsiNequiv1}}Of course,
since there is no spatial interaction in this case, indexing the
diffusions by a subset of $\bZ^{d}$ is not relevant.
Systems of type \eqref{eq:systintro} are called \emph{mean field
models} (or \emph{weakly interacting diffusions}) in statistical
physics and have attracted much attention in the past years (see, e.g.,
\cite{McKean1967,Gartner,Oelsch1984,Sznit84,daiPra96}), since they
are capable of modeling complex dynamical behavior of various types
of\vadjust{\goodbreak}
real-world models from physics to biology, like, for example,
synchronization of large populations of individuals, collective
behavior of social insects, emergence of synchrony in neural networks
\cite{22657695,1108.2414,MR2998591,1211.0299} and providing
particle approximations for various nonlinear PDEs appearing in physics
\cite{MR1410117,MR2834721,MR2731396,Malrieu2003,MR2280433}.

The most prominent example of such models is the Kuramoto model, which
has been widely considered in the literature as the main prototype for
synchronization phenomena (see, e.g., \cite
{Acebron2005,Lucon2011,1209.4537,GPP2012,Strogatz1991}),
%
%
\begin{eqnarray}
\label{eq:Kurintro} \mathrm{d}\tta_{i}(t)= \omega_{i} \,
\mathrm{d}t + \frac{K}{N}\sum_{j=1}^{N}
\sin (\tta _{j}-\tta_{i} )\,\mathrm{d}t + \sig\,
\mathrm{d}B_{i}(t),
\nonumber
\\[-8pt]
\\[-8pt]
\eqntext{t\geq0, i=1,\ldots, N,}
\end{eqnarray}
where $K\geq0$ is the intensity of interaction and $\tta_{i}\in\bS
:=\bR
/2\pi$.

In the context of weighted interactions, a notable attempt to go beyond
pure mean field interactions has been to consider \emph{moderately
interacting diffusions}; see \cite{MR779460,Meleard1987,Jourdain1998}.

\subsubsection{The spatial case}
\label{subsubsec:introspatialcase}
The motivation of going beyond pure mean-field interaction comes from
the biological observation that neurons do not interact in a mean-field
way (see, e.g., \cite{PhysRevLett.110.118101} and references therein),
and a vast literature exists in physics about synchronization on
general networks. In particular, several papers have already considered
model \eqref{eq:systintro} (in dimension $d=1$) for different choices
of spatial weight $\Psi$ defined in \eqref{eq:defPsi}. In this paper,
we will be more particularly interested in two classes of spatial weights:
\begin{longlist}[(1)]
\item[(1)] \textit{The $P$-nearest-neighbor model}: this model (see
\cite{PhysRevLett.106.234102,PhysRevE.85.026212}) concerns the case
where each diffusion $\tta_{i}\in\gL_{N}$ only interacts with its
neighbors within a box $\gL_{P}\subseteq\gL_{N}$, where $P$ is smaller
than $N$,
%
%
\begin{eqnarray}
\label{eq:defPnearestintro} \mathrm{d}\tta_{i}(t)= c(\tta_{i},
\omega_{i}) \,\mathrm{d}t + \frac
{1}{|\gL_{P}|} \mathop{\sum
_{j\in\gL_{P}}}_{ j\neq i} \Gamma (\tta_{i},
\omega_{i}, \tta _{j}, \omega_{j} ) \,
\mathrm{d}t+ \sig\cdot\mathrm {d}B_{i}(t),
\nonumber
\\[-8pt]
\\[-8pt]
 \eqntext{i\in\gL_{N}.}
\end{eqnarray}
We are concerned in this work with the case where $P$ is proportional
to $N$, that is,
%
%
\begin{equation}
\label{eq:PRNintro} P=RN
\end{equation}
for a fixed proportion $R\in(0, 1]$.
%
%
\begin{remark}
The case of $R=1$ corresponds to the mean field case.
Understanding the behavior of system \eqref{eq:defPnearestintro} in the
case of a pure local interaction (i.e., when $P\ll N$) does not enter
into the scope of this work. In particular, we will\vadjust{\goodbreak} not address the
question of $P$ of order smaller than $N$ (e.g., $P=RN^{\alpha}$ for
some $\alpha<1$), whose behavior as $N\to\infty$ seems to be quite different.
\end{remark}

Under assumption \eqref{eq:PRNintro}, the $P$-nearest-neighbor model
\eqref{eq:defPnearestintro} enters into the framework of \eqref
{eq:systintro} for the following choice of $\Psi$ in \eqref{eq:defPsi}:
%
%
\begin{equation}
\label{eq:weightnearest} \qquad\quad\forall x, y \in \biggl[- \frac{1}{2}, \frac{1}{2}
\biggr]^{d} \qquad\Psi (x, y):= \chi_{R}(x - y):=
\frac{1}{(2R)^{d}} \mathbh{1}_{[-R,
R]^{d}} (x-y ).
\end{equation}
\item[(2)] \textit{The power-law model}: this model also considered in
the physical literature (see \cite
{PhysRevE.82.016205,PhysRevE.85.066201,PhysRevE.66.011109,PhysRevE.54.R2193})
corresponds to the case where
$\Psi$ in \eqref{eq:defPsi} is given by
%
%
\begin{equation}
\label{eq:weightpower} \forall x, y \in \biggl[- \frac{1}{2}, \frac{1}{2}
\biggr]^{d}\qquad \Psi (x, y):= \frac{1}{\| x-y \|^{\alpha}}
\end{equation}
for some parameter $\alpha\geq0$, that is,
%
%
\begin{eqnarray}
\label{eq:defpowerintro} \mathrm{d}\tta_{i}(t)&=& c(\tta_{i},
\omega_{i}) \,\mathrm{d}t\nonumber \\
&&{}+ \frac
{1}{|\gL_{N}|} \mathop{\sum
_{j\in\gL_{N}}}_{ j\neq i} \Gamma (\tta_{i},
\omega_{i}, \tta _{j}, \omega_{j} ) \biggl\|
\frac{i-j}{2N} \biggr\|^{-\alpha} \,\mathrm{d}t+ \sig\cdot\mathrm{d}
B_{i}(t),\\
 \eqntext{i\in\gL_{N}.}
\end{eqnarray}

Note that the pure mean field case corresponds again to $\alpha=0$. As
observed in the articles mentioned above on the basis of numerical
simulations, it appears that the behavior of the system is strongly
dependent on the value of the parameter $\alpha$. The situation which
is considered in this paper corresponds to the subcritical case where
the parameter is smaller than the dimension
%
%
\begin{equation}
\label{hyp:alphaintro} \alpha<d.
\end{equation}
The case of $\alpha\geq d$ is much more delicate and will be the object
of future work. We refer to Remark~\ref{rem:alphageqd} below for
further explanations on this case.

It is easy to see that in the case of \eqref{hyp:alphaintro} the
renormalization of the positions by a factor $ \frac{1}{2N}$ in \eqref
{eq:defpowerintro} is necessary: by standard arguments, the diverging
series $\sum_{j\in\gL_{N}, j\neq i} \| i-j \|^{-\alpha}$ is of order
$N^{d-\alpha}$. Consequently, $ \frac{1}{ \llvert \gL_{N} \rrvert
}\sum_{j\in\gL_{N}, j\neq i}\| \frac {i-j}{2N} \|
^{-\alpha}$
is of order $ \frac{N^{\alpha}}{ \llvert \gL_{N} \rrvert }
N^{d-\alpha}= O(1)$, so that we should expect a nontrivial limit in
\eqref{eq:defpowerintro}, as $N\to\infty$.
\end{longlist}

\subsection{Main lines of proof and organization of the paper}
The strategy usually used in the literature on mean-field models (see
\cite{Gartner,Jourdain1998,Lucon2011,Oelsch1984}) for the
convergence\vspace*{1pt} of the empirical measure \eqref{eq:nuNintro} is the
following: first prove tightness of $(\nu^{(N)})_{N\geq1}$ in the set
of measure-valued continuous processes and second, prove uniqueness of
any possible limit points, that is, uniqueness in the McKean--Vlasov
equation \eqref{eq:MKVintro}.\vadjust{\goodbreak}

In our context, a priori uniqueness in \eqref{eq:MKVintro} appears
unclear, due the fact that our model includes \emph{singular spatial
weights} [discontinuous in \eqref{eq:weightnearest} and singular in
\eqref{eq:weightpower}] and also a class of dynamics with \emph{no
global-Lipschitz continuity} and \emph{polynomial growth}; recall the
FitzHugh--Nagumo case \eqref{eq:FHNintro}. Note that we are also
concerned with the case where $\sig$ is degenerate (even equally zero)
for which uniqueness in \eqref{eq:MKVintro} is also not clear.

To bypass this difficulty, we adopt a converse strategy: we first prove
\emph{existence} of a solution to the mean-field limit \eqref
{eq:MKVintro} (through an {ad-hoc} fixed point argument, using
ideas from Sznitman \cite{SznitSflour}). Second, via a propagator
method (see \cite{MR1741805} for related ideas), we prove the
convergence (with respect to a Wasserstein-like distance adapted to the
singularities of the interaction) of the empirical measure to \emph
{any} solution to \eqref{eq:MKVintro}. In particular, easy byproducts
of this method are \emph{uniqueness} of any solution to \eqref
{eq:MKVintro} as well as \emph{explicit rates of convergence} to the
McKean--Vlasov limit. In that sense, one of the main conclusions of the
paper is to exhibit a \emph{phase transition} in the size of the
fluctuations in the power-law case; see Theorem~\ref
{theo:LLNpowerinfd}. An actual central limit theorem in this case is of
course a natural perspective and is currently under investigation.

The paper is organized as follows: we give in Section~\ref{sec:results}
the main assumptions on the model and we state the main results
(Theorems \ref{theo:LLNPnn} and \ref{theo:LLNpowerinfd}).
Section~\ref{sec:nonlin} contains the proof of Proposition~\ref
{prop:existencenut}
concerning the existence of a solution to the McKean--Vlasov
equation \eqref{eq:MKVintro}. Section~\ref{sec:propagator} summarizes
the main ideas and results concerning the propagator method. The proofs
of the laws of large numbers are provided in Section~\ref{sec:Pnn} for
the $P$-nearest case and in Section~\ref{sec:PLinfd} for the power-law
case. An additional assumption of regularity is made from Section~\ref
{sec:propagator} to \ref{sec:PLinfd}, with is discarded in
Section~\ref{sec:loclip}.

\section{Mathematical set-up and main results}
\label{sec:results}
\subsection{The model}
Fix $N\geq1$, $T>0$, and let $\gL_{N}$ be the hypercube $[\![-N,
\ldots,\break   N]\!]^{d}\subset\bZ^{d}$ and $|\gL_{N}|=(2N+1)^{d}$
be its
volume. We consider $|\gL_{N}|$ diffusions on $[0, T]$ with values in
the state space\setcounter{footnote}{1}\footnote{Note that it is also possible to choose $\cX$
as the circle $\bS:= \bR/2\pi\bZ$ in the case of the Kuramoto model,
but we will stick to $\cX:=\bR^{m}$ for simplicity.} $\mathcal
{X}:=\bR
^{m}$ for a certain $m\geq1$.

Each diffusion $\tta_{i}$ is attached to the site $i$ of $\gL_{N}$. The
local dynamics of $\tta_{i}$ is governed by the following stochastic
differential equation which is perturbed by a random environment
represented by a vector $\omega_{i}\in\cE:=\bR^{n}$ ($n\geq1$):
%
%
\begin{equation}
\label{eq:singlepart} \mathrm{d}\tta_{i}(t)= c(\tta_{i},
\omega_{i}) \,\mathrm{d}t + \sig \cdot\mathrm{d} B_{i}(t),\qquad
0\leq t\leq T, i\in\gL_{N},
\end{equation}
where $\sig\in\bR^{m\times m}$ is the covariance matrix, $c(\cdot,
\cdot)$ is a function from $\mathcal{X}\times\cE$ to $\mathcal{X}$,
and $(B_{i})$ is a given sequence of independent Brownian motions in
$\cX$.

The vectors $(\omega_{i})_{i\in\gL_{N}}$ are supposed to be i.i.d.
realizations of a law $\mu$ and are hence seen as a random environment
for the diffusions.\vadjust{\goodbreak}

When connected to the others, the diffusions interact in a mean field
way with spatial extension,
%
%
\begin{eqnarray}
\label{eq:odegene} \mathrm{d}\tta_{i}(t)&=& c(\tta_{i},
\omega_{i})\,\mathrm{d}t\nonumber\\
&&{} + \frac
{1}{|\gL
_{N}|}\mathop{\sum
_{j\in\gL_{N}}}_{ j\neq i} \Gamma (\tta _{i}, \omega
_{i}, \tta_{j}, \omega_{j} ) \Psi \biggl(
\frac{i}{2N}, \frac
{j}{2N} \biggr) \,\mathrm{d}t + \sig\cdot
\mathrm{d}B_{i}(t),
\\
\eqntext{0\leq t\leq T, i\in \gL_{N},}
\end{eqnarray}
where $\Gamma$ is a function from $(\mathcal{X}\times\cE)^{2}$ to
$\mathcal{X}$,
and $(x,y)\mapsto\Psi(x,y)$ is a function from $ [- \frac{1}{2},
\frac{1}{2} ]^{d}\times [- \frac{1}{2}, \frac{1}{2}
]^{d}$ to
$[0, \infty)$. The required assumptions for the function $\Psi$ will be
made precise in Assumption~\ref{ass:psi} below. One should notice at
this point that $\Psi(x, y)$ does not need to depend on the difference $x-y$.

We suppose that, at time $t=0$, the variables $(\tta_{i}(0))_{1\leq
i\leq N}$ are independent and identically distributed according to a
probability distribution $\zeta(\mathrm{d}\tta)$ on~$\cX$.

%
\begin{remark}
\label{rem:boundcond}
Instead of considering diffusions on $\gL_{N}$, we can also suppose
periodic boundary conditions, that is, when $\gL_{N}$ is replaced by
$\gL_{N, \mathrm{per}}:=\bbT_{N}^{d}$, where $\bbT_{N}$ is the discrete
$N$-torus, that is,
$[\![-N, \ldots, N]\!]$ with $-N$ and $N$ identified. The
only thing that changes in what follows in the continuous model is that
one should replace $ [- \frac{1}{2}, \frac{1}{2} ]^{d}$ by
$\bbT
^{d}$ where $\bbT:= [- \frac{1}{2}, \frac{1}{2} ]/_{(-{1}/{2})\sim{1}/{2}}$. Since the corresponding changes in the proofs
of this paper remain marginal, we will restrict to the non periodic
case and let the interested reader make the appropriate modifications
in the periodic case.
\end{remark}

\subsection{Notation and assumptions}
\label{subsec:assumptions}
From now on, we will suppose that the following assumptions
(Assumptions \ref{ass:Gammac}, \ref{ass:muzeta} and \ref{ass:psi}) are
satisfied throughout the paper. In particular, saying that
Assumption~\ref{ass:psi} is true means that we are either in the
$P$-nearest-neighbor case or in the power-law case; see hypotheses (H1) and (H2) below.
%
%
\begin{assmp}[(Hypothesis on $\Gamma$ and $c$)]
\label{ass:Gammac}
We make the following assumptions:
\begin{itemize}
\item
The function $(\tta, \omega)\mapsto c(\tta, \omega)$ is supposed
to be
locally Lipschitz-continuous in~$\tta$ (for fixed $\omega$) and
satisfy a
one-sided Lipschitz condition w.r.t. the two variables $(\tta, \omega)$,
%
%
\begin{equation}
\label{eq:cgrowthcond}\qquad \forall(\tta, \omega), (\bar\tta, \bar\omega)\qquad \bigl\langle\tta -
\bar\tta, c(\tta, \omega)-c(\bar\tta, \bar\omega)\bigr\rangle \leq L \bigl(\|
\tta-\bar \tta \|^{2} + \| \omega -\bar\omega \| ^{2} \bigr)
\end{equation}
for some constant $L$ (not necessarily positive). We suppose also some
polynomial bound about the function $c$,
%
%
\begin{equation}
\label{eq:polgrowthc} \forall(\tta, \omega)\qquad
 \bigl\| C(\tta, \omega) \bigr\|\leq|\!|\!|c |\!|\!|
\bigl(1+\| \tta \|^{\kappa} + \| \omega \|^{\iota} \bigr)
\end{equation}
for some constant $|\!|\!|c |\!|\!|>0$ and where $\kappa\geq2$ and
$\iota\geq1$.\vadjust{\goodbreak}
\item The interaction term $\Gamma$ is supposed to be bounded by
$\| \Gamma \|_{\infty}$ and globally Lipschitz-continuous on $
(\cX\times\cE
)^{2}$, with a Lipschitz constant $\| \Gamma \|_{\mathrm{Lip}}$.
\end{itemize}
We also assume that for fixed $\bar\tta, \omega, \bar\omega$,
the functions
$\tta\mapsto c(\tta, \omega)$ and $\tta\mapsto\Gamma(\tta,
\omega, \bar\tta,
\bar\omega)$ are twice differentiable with continuous derivatives.
\end{assmp}

%
\begin{remark}
\label{rem:constantc}
Assumption~\ref{ass:Gammac} is in particular satisfied for the
FitzHugh--Nagumo case. One technical difficulty is the dynamics is not
\emph{globally} Lispchitz continuous. This will entail some technical
complications in the following. Note also that the constant $|\!|\!|c
|\!|\!|$
mentioned in \eqref{eq:cgrowthcond} does not take part in the estimates
of Sections~\ref{sec:propagator} to \ref{sec:PLinfd}. It only enters
into account in Section~\ref{sec:nonlin}.
\end{remark}

%
\begin{assmp}[(Assumptions on $\mu$ and $\zeta$)]
\label{ass:muzeta}
We suppose that the initial distribution $\zeta$ of $\tta$ satisfies
the following moment condition:
%
%
\begin{equation}
\label{eq:asszeta} \int_{\cX} \| \tta \|^{\kappa}\zeta(
\mathrm{d}\tta)<\infty,
\end{equation}
and that the law of the disorder $\mu$ satisfies the moment condition
%
%
\begin{equation}
\label{eq:assmu} \int_{\cE} \| \omega \|^{\iota}\mu(
\mathrm{d}\omega)<\infty,
\end{equation}
where the constants $\kappa$ and $\iota$ are given by \eqref
{eq:polgrowthc} in Assumption~\ref{ass:Gammac}.
\end{assmp}
%
%
\begin{assmp}[(Assumptions on the weight $\Psi$)]
\label{ass:psi}
In order to cover the case of both the $P$-nearest model and the
power-law interaction introduced in Section~\ref
{subsubsec:introspatialcase}, we suppose that either
hypotheses (H1) or~(H2) is true:
\begin{longlist}[(H1)]
\item[(H1)]
$P$-nearest-neighbor:
%
%
\begin{equation}
\label{eq:weightPnn} \forall x, y \in \bigl[- \tfrac{1}{2}, \tfrac{1}{2}
\bigr]^{d}\qquad \Psi (x, y):= \chi_{R}(x,y),
\end{equation}
where $\chi_{R}$ is defined in \eqref{eq:weightnearest}.
\item[(H2)]
Power-law: the function $\Psi$ is supposed to be a nonnegative function
on $ [- \frac{1}{2}, \frac{1}{2} ]^{d}\times [- \frac{1}{2},
\frac{1}{2} ]^{d}$ such that the following properties are satisfied:
%
%
\begin{eqnarray}
\cI_{1}(\Psi)&:=&\sup_{a, x\in  [- {1}/{2}, {1}/{2}
]^{d}} \| x-a
\|^{\alpha} \Psi(x, a) <\infty,\label{eq:discont}
\\
\cI_{2}(\Psi)&:= &\sup_{x, y \in [- {1}/{2},{1}/{2}
]^{d}}\frac{\int\llvert \Psi(x, \bar x) - \Psi(y, \bar x)
\rrvert \,\mathrm{d}\bar x}{\| x-y \|^{(d-\alpha)\wedge1}}<
\infty, \label{eq:intpsi}
\\
\label{eq:holderpsi} \qquad\cI_{3}(\Psi)&:=& \sup_{a, x, y \in [- {1}/{2}, {1}/{2} ]^{d}} \frac{\llvert \| x-a \|^{2\gamma} \Psi(x, a) -
\| y-a \|^{2\gamma} \Psi(y, a) \rrvert }{\| x-y \|^{(2\gamma
-\alpha
)\wedge
1}}
\nonumber
\\[-8pt]
\\[-8pt]
\nonumber
&<&
\infty
\end{eqnarray}
for some parameters $\alpha\in[0, d)$ and $\gamma$ chosen to be
%
%
\begin{equation}
\label{eq:assgamma} %
\cases{\displaystyle \gamma\in \biggl[\alpha, \frac{d}{2} \biggr), &\quad $
\mbox{if } \displaystyle\alpha\in \biggl[0, \frac{d}{2} \biggr),$\vspace*{2pt}
\cr
\displaystyle\gamma=
\frac{d}{2}, & \quad$\mbox{otherwise.}$} %
\end{equation}
\end{longlist}
\end{assmp}
%
%
\begin{remark}
\label{rem:gamma}
Note that we could have chosen simply $\gamma= \frac{d}{2}$ in any
case. But this would have led to worse convergence rates than the ones
that we obtain below in Theorem~\ref{theo:LLNpowerinfd}.
\end{remark}
Of course, the main prototype for hypothesis (H2) is
when $\Psi (x, y ) = \| x-y \|^{-\alpha}$, for $\alpha<d$
[recall \eqref{eq:weightpower}]. But, the assumptions made in (H2) cover a larger class of examples: the reader may
think of the general case of $\Psi(x, y):= \psi(x, y)\| x-y \|
^{-\alpha
}$, for a bounded Lipschitz-continuous function $\psi$. Note also that
the case of bounded Lispchitz interactions is also captured (take
$\alpha=0$).

%
\begin{remark}[(About the supercritical case)]
\label{rem:alphageqd}
The case of a power-law interaction with $\alpha\geq d$ is more
delicate and requires more attention. Note that, to our knowledge, no
proposition for any continuous limit has been made in the literature in
this case. We are only aware of \cite{PhysRevE.82.016205}, where system
\eqref{eq:supercrit} below is considered for finite $N$.

One trivial observation is that the series $\sum_{j\in\gL
_{N}, j\neq i}\| i- j \|^{-\alpha}$ is in this case already
convergent. Consequently, an interaction term of the form $\frac
{1}{|\gL
_{N}|}\times \sum_{j\in\gL_{N}, j\neq i} \Gamma (\tta
_{i}, \omega
_{i}, \tta_{j}, \omega_{j}  )\| i-j \|^{-\alpha}$ simply
vanishes to
$0$ as $N\to\infty$. Hence, the correct model in this case is where the
factor $ \frac{1}{ \llvert \gL_{N} \rrvert }$ is absent,
%
%
\begin{eqnarray}
\label{eq:supercrit} \mathrm{d}\tta_{i}(t)&=& c(\tta_{i},
\omega_{i}) \,\mathrm{d}t\nonumber\\
&&{} + \mathop{\sum_{j\in\gL
_{N}}}_{ j\neq i}
\Gamma (\tta_{i}, \omega_{i}, \tta_{j},
\omega_{j} ) \| i-j \|^{-\alpha} \,\mathrm{d}t + \sig\cdot
\mathrm{d}B_{i}(t),\\
 \eqntext{i\in \gL_{N}.}
\end{eqnarray}
The main difficulty for the derivation of the correct continuous limit
in the case of \eqref{eq:supercrit} lies in the fact that the
interaction term $\sum_{j\in\gL_{N}, j\neq i} \Gamma
(\tta_{i}, \omega_{i},\break  \tta_{j}, \omega_{j}  ) \| i-j \|
^{-\alpha}$ is not
sufficiently mixing: if it exists, the McKean--Vlasov limit in this
case should be random. We believe that the correct continuous limit
should be governed by a stochastic partial differential equation
instead of a deterministic PDE. This case is currently under
investigation and will be the object of a future work.
\end{remark}

\subsection{The empirical measure} Let us consider for fixed horizon
$T$ and time $t\in[0, T]$, the empirical measure $\nu^{(N)}_{t}$
[introduced in \eqref{eq:nuNintro}],
%
%
\begin{equation}
\nu_{t}^{(N)}(\mathrm{d}\tta, \mathrm{d}\omega,
\mathrm{d}x):= \frac{1}{ \llvert
\gL_{N}
\rrvert } \sum_{j}
\delta_{(\tta_{j}(t), \omega_{j},
x_{j})}(\mathrm{d}\tta, \mathrm{d}\omega, \mathrm{d}x)
\end{equation}
as a probability measure on $\cX\times\cE\times [- \frac{1}{2},
\frac{1}{2} ]^{d}$. Here
%
%
\begin{equation}
x_{j}:= \frac{j}{2N}\in \biggl[- \frac{1}{2},
\frac{1}{2} \biggr]^{d},\qquad j\in \gL_{N}.
\end{equation}
%
\subsection{The McKean--Vlasov equation}
\label{subsec:MKV}
The convergence of the empirical measure at $t=0$ is clear: since
$(\tta
_{i}(0), \omega_{i})_{1\leq i\leq N}$ are i.i.d. random variables sampled
according to $\zeta\otimes\mu$, the initial empirical measure $\nu
_{0}^{(N)}$ converges, as $N\to\infty$, to
%
%
\begin{equation}
\label{eq:nu0} \nu_{0}(\mathrm{d}\tta, \mathrm{d}\omega,
\mathrm{d}x):=\zeta (\mathrm{d}\tta)\mu(\mathrm{d} \omega)\,\mathrm{d}x.
\end{equation}
An application of It\^o's formula to \eqref{eq:odegene} [for any $(\tta,
\omega, x)\mapsto f(\tta, \omega, x)$ bounded function of class
$\cC^{2}$
w.r.t. $\tta$ with bounded derivatives] leads to the following martingale
representation for $\nu^{(N)}$:
%
%
\begin{eqnarray}
\label{eq:nuNt} \bigl\langle\nu_{t}^{(N)}, f\bigr\rangle& =&
\bigl\langle\nu_{0}^{(N)}, f\bigr\rangle + \int
_{0}^{t}\biggl\langle\nu _{s}^{(N)}
, \frac{1}{2}\div_{\tta} \bigl(\sig\sig^{T}\nabla
_{\tta }f \bigr) + \nabla_{\tta}f \cdot c(\cdot, \cdot)\biggr
\rangle\,\mathrm{d}s\nonumber
\\
&&{}+ \int_{0}^{t}\biggl\langle
\nu_{s}^{(N)}, \nabla_{\tta}f \cdot\int\Gamma (
\cdot, \cdot, \bar\tta,\bar\omega) \Psi(\cdot, \bar x)\nu _{s}^{(N)}(
\mathrm{d}\bar\tta, \mathrm{d}\bar\omega, \mathrm {d}\bar x)\biggr\rangle\,
\mathrm{d}s\\
&&{} + M^{(N)}_{t}(f),\nonumber
\end{eqnarray}
where $M^{(N)}_{t}(f):= \frac{1}{ \llvert \gL_{N} \rrvert }
\sum_{j} \int_{0}^{t} \nabla_{\tta} f(\tta_{j}(s), \omega_{j}, x_{j})
\cdot
\sig\,\mathrm{d}B_{j}(s)$ is a martingale. Note that we use here the usual
duality notation $\langle\nu, f\rangle=\int f \,\mathrm{d}\nu$ for
the integral of a
test function $f$ against a measure $\nu$.

Taking formally $N\to\infty$ in \eqref{eq:nuNt} shows that any limit
point of $\nu^{(N)}$ should satisfy the following nonlinear
McKean--Vlasov equation:
%
%
\begin{eqnarray}
\label{eq:nut} \partial_{t} \langle\nu_{t}, f\rangle &=&
\biggl\langle\nu_{t}, \frac
{1}{2}\div _{\tta} \bigl(
\sig\sig^{T}\nabla_{\tta}f \bigr) + \nabla_{\tta}f
\cdot c(\cdot, \cdot)\biggr\rangle
\nonumber
\\[-8pt]
\\[-8pt]
\nonumber
&&{}+ \biggl\langle\nu_{t}, \nabla_{\tta}f \cdot\int\Gamma (
\cdot, \cdot, \bar\tta,\bar \omega) \Psi(\cdot, \bar x)\nu_{t}(
\mathrm{d} \bar\tta, \mathrm {d}\bar \omega, \mathrm{d}\bar x)\biggr\rangle,
\end{eqnarray}
where $\Psi(\cdot, \cdot)$ is the weight function introduced either in
hypotheses (H1) or in~(H2).
%
%
\begin{remark}
An important remark about a priori properties of \eqref{eq:nut} is the
following: taking a test function $f$ in \eqref{eq:nut} that does not
depend on $\tta$ implies
\[
\langle\nu_{0}, f\rangle= \langle\nu_{t}, f\rangle\qquad
\forall t\in[0, T].
\]
In particular, the marginal distribution of $(\omega, x)$ w.r.t. the measure
$\nu_{t}$ is independent of $t$ and equal to $\mathrm{d}\mu\otimes
\mathrm{d}x$. This
implies that, for the class of singular weight we consider here, $\Psi$
is always integrable against $\nu_{t}$, for all $t$, since the function
$y\mapsto\| x-y \|^{-\alpha}$ is integrable w.r.t. to the Lebesgue measure
on $ [- \frac{1}{2}, \frac{1}{2} ]^{d}$.

Moreover, since the function $c$ is supposed to have a polynomial
growth [recall~\eqref{eq:polgrowthc}], one has to justify in particular
the term $ \langle\nu_{t}, \nabla_{\tta}f\cdot c(\cdot, \cdot
)\rangle$ in
\eqref
{eq:nut} (the others are easily integrable). Thus, one should look for
solutions $t\mapsto\nu_{t}$ having finite moment: for all $t\in[0, T]$,
$ \int_{\cX\times\cE} \| \tta \|^{\kappa} \| \omega \|^{\iota
}\nu_{t}(\mathrm{d}\tta,
\mathrm{d}\omega, \mathrm{d}x)<\infty$.

In particular, well-posedness in \eqref{eq:nut} will be addressed
within the class of all measure-valued processes satisfying the
properties mentioned above.
\end{remark}

Formally integrating by parts in equation \eqref{eq:nut} and assuming
the existence of a density $\nu_{t}(\mathrm{d}\tta, \mathrm
{d}\omega, \mathrm{d}x) =
q_{t}(\tta
, \omega, x)\,\mathrm{d}\tta\mu(\mathrm{d}\omega)\,\mathrm{d}x$,
$q_{t}$ satisfies
%
%
\begin{eqnarray}
\label{eq:nutstrong} %
 \partial_{t}
q_{t} &=& \frac{1}{2}\div_{\tta} \bigl(\sig\sig
^{T}\nabla _{\tta}q_{t} \bigr)-
\div_{\tta} \bigl(q_{t}(\tta, \omega, x)c(\tta, \omega )
\bigr)
\nonumber\\
&&{} -\div_{\tta} \biggl(q_{t}(\tta, \omega, x)\int \Gamma(
\tta, \omega, \bar\tta,\bar\omega) \Psi(x, \bar x)q_{t}(\bar\tta, \bar
\omega, \bar x)\,\mathrm{d} \bar\tta\mu(\mathrm{d}\bar\omega)\,\mathrm{d}\bar x
\biggr),\\
 \eqntext{t>0. }
\end{eqnarray}
In the case where $\sig$ is nondegenerate, one can make this
integration by parts rigorous: using the same arguments as in \cite{GLP2011},
Appendix A, one can show that for any measure-valued initial
condition in \eqref{eq:nut}, by the regularizing properties of the heat
kernel, the solution of \eqref{eq:nut} has a regular density $q_{t}$
for all positive time that solves \eqref{eq:nutstrong}. We refer to
\cite{GLP2011}, Proposition A.1, for further details. But of course, if
$\sig$ is degenerate, the strong formulation \eqref{eq:nutstrong} does
not necessarily make sense, and one has to restrict to the weak
formulation \eqref{eq:nut} in that case.

\subsection{Results}
\label{subsec:resultsLLN}
The first result of this paper, whose proof is given in Section~\ref
{sec:nonlin}, concerns the existence of a weak solution to the
McKean--Vlasov equation \eqref{eq:nut}:
%
%
\begin{proposition}
\label{prop:existencenut}
Under Assumptions \ref{ass:Gammac}, \ref{ass:muzeta} and \ref{ass:psi},
for any initial condition $\nu_{0}(\mathrm{d}\tta, \mathrm{d}\omega
, \mathrm{d}
x)=\zeta(\mathrm{d}\tta
)\mu(\mathrm{d}\omega)\,\mathrm{d}x$, there exists a solution
$t\mapsto\nu_{t}$
to \eqref{eq:nut}.
\end{proposition}
Having proven the existence of at least one such solution in the
general case, we turn to the issue of the convergence of the empirical
measure to \emph{any} of such solution. From now on, we specify the
problem to the case of hypothesis (H1) (Section~\ref
{subsubsec:Pnncase}) and of hypothesis (H2)
(Section~\ref{subsubsec:Plawcase}). For each case, in order to state
the convergence result, one needs to define an appropriate distance
between two random measures that is basically the supremum over
evaluations against a set of test functions. Such a space of test
functions must incorporate the kind of singularities that are present
either in hypotheses (H1) or~(H2).
\subsubsection{The $P$-nearest-neighbor case}
\label{subsubsec:Pnncase}Suppose that the weight function $\Psi$
satisfies hypothesis (H1) of Assumption~\ref{ass:psi}.
%
%
\begin{definition}[(Test functions for $P$-nearest-neighbor)]
\label{def:spacePnearest}
For fixed $R\in(0,1]$ and $a\in [-\frac{1}{2}, \frac{1}{2}
]^{d}$, let $\cC_{R, a}$ be the set of functions $f$ on $\cX\times
\cE
\times [- \frac{1}{2}, \frac{1}{2} ]^{d}$ of the form
\[
f\dvtx(\tta, \omega, x)\mapsto g(\tta, \omega) \cdot\chi _{R} (x-a
),
\]
where $\chi_{R}$ is given in \eqref{eq:weightnearest} and $g$ is
globally Lipschitz-continuous w.r.t. $(\tta, \omega)$
%
%
\begin{eqnarray}
\label{eq:PnnLipschitz} \exists C>0,
\forall(\tta, \omega, \bar\tta, \bar\omega)
\nonumber
\\[-8pt]
\\[-8pt]
\eqntext{\displaystyle\bigl \| g(
\tta, \omega) - g(\bar\tta, \bar\omega) \bigr\| \leq C \bigl(\llVert \tta- \bar\tta
\rrVert + \llVert \omega- \bar\omega\rrVert \bigr).}
\end{eqnarray}
Let
\[
\| f \|_{R,a}:= \sup_{\tta, \bar\tta, \omega, \bar\omega} \frac{ \| g(\tta, \omega ) - g(\bar\tta, \bar\omega) \|}{ \|
\tta-\bar\tta \| + \| \omega -\bar\omega \|}
\]
be the corresponding seminorm.
\end{definition}
%
%
\begin{remark}
\label{rem:nablafR}
Note that for any $f\in\cC_{R, a}$ that is $\cC^{1}$ in the variable
$\tta$, the following estimate holds:
%
%
\begin{equation}
\label{eq:nablafR} \forall\tta, \omega, x\qquad\bigl \| \nabla_{\tta} f(\tta, \omega,
x) \bigr\| \leq\| f \|_{R,
a} \chi_{R} (x-a ).
\end{equation}
\end{remark}
We now turn to the appropriate distance between two random measures:
%
%
\begin{definition}[(Distance for $P$-nearest-neighbor)]
\label{def:distancePnn}
For random probability measures $\lambda$ and $\nu$ on $\cX\times
\cE\times
[- \frac{1}{2}, \frac{1}{2} ]^{d}$, let
\[
d_{R}(\lambda, \nu):= \sup_{f} \bigl(\bE\bigl\|
\langle f, \lambda \rangle - \langle f, \nu\rangle \bigr\|^{2}
\bigr)^{1/2},
\]
where the supremum is taken over all functions $f\in\bigcup_{a\in
[-1,1]^{d}}\cC_{R, a}$, such that $\| f \|_{R, a}\leq1$, $\| f \|
_{\infty}\leq1$.
\end{definition}

Our convergence result is given in the following:
%
%
\begin{theorem}[(Law of large numbers)]
\label{theo:LLNPnn}
Under Assumptions \ref{ass:Gammac}, \ref{ass:muzeta} and
hypothesis \textup{(H1)} of Assumption~\ref{ass:psi}, for all $R\in(0, 1]$, for
any arbitrary solution $\nu$ to the mean-field equation \eqref{eq:nut},
we have
%
%
\begin{equation}
\sup_{0\leq t\leq T} d_{R}\bigl(\nu^{(N)}_{t},
\nu_{t}\bigr)\leq \frac{C}{N^{ 1\wedge{d}/{2}}},
\end{equation}
where the constant $C>0$ only depends on $T$, $\Gamma$, $R$ and $c$.
\end{theorem}
%
\subsubsection{The case of the power-law interaction}
\label{subsubsec:Plawcase}
Assume that the weight function $\Psi$ satisfies hypothesis (H2). In view of the form of $\Psi$ in this case (recall
Assumption~\ref{ass:psi}), the main idea is to consider test functions
$(\tta, \omega, x)\mapsto f(\tta, \omega, x)$ that become regular when
renormalized by $\| x-a \|^{\alpha}$. The seminorm $\| \cdot \|_{a}$
introduced in \eqref{eq:seminormplaw} below should therefore be thought
of as a \emph{weighted H\"older seminorm}.
%
%
\begin{definition}[(Test functions for power-law interaction)]
\label{def:spacepowerinfd}
For fixed $\alpha$ and $\gamma$ as in Assumption~\ref{ass:psi} and for
fixed $a\in [- \frac{1}{2}, \frac{1}{2} ]^{d}$, let $\cC
_{a}$ be
the set of functions $(\tta, \omega, x)\mapsto f(\tta, \omega,
x)$ on $\cX
\times\cE\times [- \frac{1}{2}, \frac{1}{2} ]^{d}$ satisfying:
\begin{itemize}
\item regularity w.r.t. $(\tta, \omega)$: $(\tta, \omega)\mapsto\|
x-a \|^{\alpha
} f(\tta, \omega, x)$ is globally Lipschitz-continuous on $\cX
\times\cE$,
uniformly in $x$, that is,
%
%
\begin{eqnarray}
\label{eq:fLipttaom} \exists C>0,
\forall(\tta, \omega, \bar\tta, \bar\omega)
\nonumber
\\[-8pt]
\\[-8pt]
\eqntext{\displaystyle \| x-a
\|^{\alpha}\bigl \| f(\tta, \omega, x) - f(\bar\tta, \bar\omega, x) \bigr\|
 \leq C
\bigl(\llVert \tta- \bar\tta\rrVert + \llVert \omega- \bar\omega \rrVert \bigr);}
\end{eqnarray}
\item regularity w.r.t. $x$: $x\mapsto\| x-a \|^{\alpha} f(\tta,
\omega, x)$
is uniformly bounded
%
%
\begin{equation}
\label{eq:falphabounded} \exists C>0 \qquad\| x-a \|^{\alpha} \bigl\|
 f(\tta, \omega, x)\bigr \|
\leq C,
\end{equation}
and $x\mapsto\llvert  x-a \rrvert ^{2\gamma} f(\tta, \omega
, x)$ is
globally $(2\gamma-\alpha)\wedge1$-H\"older, uniformly in $(\tta,
\omega)$
%
%
\begin{eqnarray}
\label{eq:falphaHolder} \exists C>0
\nonumber
\\[-8pt]
\\[-8pt]
\eqntext{\displaystyle
\bigl\| \| x-a \|^{2\gamma} f(\tta, \omega, x) - \|
y-a \| ^{2\gamma} f(\tta, \omega, y) \bigr\| \leq C \| x-y \|^{(2\gamma-\alpha
)\wedge1}.}
\end{eqnarray}
\end{itemize}
Denote by
%
%
\begin{eqnarray}
\label{eq:seminormplaw} \| f \|_{a}&:=& \sup_{\tta, \bar\tta, \omega, \bar\omega, x}
\frac{ \| x-a \|^{\alpha}\| f(\tta, \omega, x) - f(\bar\tta, \bar
\omega, x) \|}{ \| \tta -\bar\tta \| + \| \omega-\bar\omega \|} \nonumber\\
&&{}+ \sup_{\tta, \omega,
x} \| x-a \|^{\alpha}\bigl \| f(
\tta, \omega, x) \bigr\|
\\
&&{}+ \sup_{\tta,
\omega, x, y} \frac{ \| \| x-a \|^{2\gamma } f(\tta, \omega, x) -
\| y-a \|^{2\gamma} f(\tta, \omega, y) \|}{\| x-y \|^{(2\gamma
-\alpha)\wedge1}}\nonumber
\end{eqnarray}
the corresponding seminorm.
\end{definition}

%
\begin{remark}
\label{rem:nablafalpha}
Note that for any $f\in\cC_{a}$ that is $\cC^{1}$ in the variable
$\tta
$, the following holds:
%
%
\begin{equation}
\label{eq:nablafalpha} \forall\tta,
\omega, x \qquad\bigl\| \nabla_{\tta} f(\tta, \omega,
x) \bigr\| \leq\frac{\| f \|_{a}}{\| x-a \|^{\alpha}}.
\end{equation}
\end{remark}
The corresponding definition of the distance between two random
measures is similar to Definition~\ref{def:distancePnn} given in the
$P$-nearest-neighbor case. The main difference here is that one needs
to take care of test functions with singularities. Since those
singularities happen at points of the form $ \frac{i}{2N}$ (for some
$i$ and $N$) that are regularly distributed on $ [- \frac{1}{2},
\frac{1}{2} ]^{d}$, we first need to introduce some further
notation: for all integers $K\geq1$, we denote by $\cD_{K}$ the regular
discretization of $ [- \frac{1}{2}, \frac{1}{2} ]^{d}$ with mesh
of length $\frac{1}{2K}$
%
%
\begin{eqnarray}
\label{eq:defIK} \cD_{K}&:=& \biggl\{ \biggl(\frac{j_{1}}{2K}, \ldots,
\frac
{j_{d}}{2K} \biggr); -K\leq j_{1}\leq K, \ldots, -K\leq
j_{d}\leq K\biggr\}
\nonumber
\\[-8pt]
\\[-8pt]
\nonumber
& \subset& \biggl[- \frac
{1}{2},
\frac{1}{2} \biggr]^{d}.
\end{eqnarray}
The appropriate distance between two random measures is then:
%
%
\begin{definition}[(Distance for power-law interaction)]
\label{def:distancePLinfd}
Let $\alpha<d$ and $p\geq2$ be defined by
%
%
\begin{equation}
\label{eq:p} p:= %
\cases{ 2,& \quad$\mbox{if } \displaystyle\alpha\in\biggl[0,
\frac{d}{2}\biggr)$,\vspace*{2pt}
\cr
\displaystyle\biggl\lceil\frac{d}{d-\alpha}\biggr\rceil,&\quad $
\mbox{if } \displaystyle\alpha\in\biggl[\frac
{d}{2}, d\biggr)$, } %
\end{equation}
where $\lceil x\rceil$ stands for the smallest integer strictly larger
than $x$.
On the set of random probability measures on $\cX\times\cE\times
[- \frac{1}{2}, \frac{1}{2} ]^{d}$, let us define a sequence of
distances $  (d^{(p)}_{K}(\cdot, \cdot) )_{K\geq1} $ indexed
by $K\geq1$, between two elements $\lambda$ and $\nu$ by
\[
d^{(p)}_{K}(\lambda, \nu)= \sup_{f}
\bigl(\bE\bigl\| \langle f, \lambda \rangle - \langle f, \nu\rangle \bigr\|^{p}
\bigr)^{1/p},
\]
where the supremum is taken over all the functions $f\in
\bigcup_{a\in\cD_{K'}, 1\leq K'\leq K}\cC_{a}$, such that $\| f \|
_{a}\leq1$. Let us then define the distance $d_{\infty
}^{(p)}(\cdot,
\cdot)$ by
%
%
\begin{equation}
d^{(p)}_{\infty}(\lambda, \nu):= \sum
_{K\geq1} \frac
{1}{2^{K}}\frac
{e^{-CK^{ {dp}/{q}}}}{K^{2d}}
\bigl(d^{(p)}_{K}(\lambda, \nu )\wedge 1 \bigr)
\end{equation}
for a sufficiently large constant $C$\vspace*{1pt} (that depends on the parameters
of our model) and where $q$ is the conjugate of $p$: $ \frac{1}{p} +
\frac{1}{q}=1$. For a precise estimate on $C$, we refer to
Proposition~\ref{prop:distanceKN} below.
\end{definition}
Apart from the weight $ \frac{e^{-CK^{ {dp}/{q}}}}{K^{2d}}$ (which
is precisely here to compensate the estimate that we find in
Proposition~\ref{prop:distanceKN} below), the definition of
$d^{(p)}_{\infty}(\cdot, \cdot)$ exactly follows the usual Fr\'echet
construction; see, for example, \cite{Gelfand1964}.
%
%
\begin{remark}
\label{rem:p}
The choice of the integer $p$ in \eqref{eq:p} is made for integrability
reasons that will become clear in the proof of Theorem~\ref
{theo:LLNpowerinfd}. One only has to notice here that $p$ has been
precisely defined so that its conjugate $q$ always satisfies $q\alpha
<d$.
\end{remark}
The main result of this work is the following:
%
%
\begin{theorem}[(Law of large numbers in the power-law case)]
\label{theo:LLNpowerinfd}
\mbox{}
Under Assumptions \ref{ass:Gammac}, \ref{ass:muzeta} and
hypothesis \textup{(H2)} of Assumption~\ref{ass:psi}, for any arbitrary
solution $\nu$ to the mean-field equation \eqref{eq:nut}, we have
%
%
\begin{equation}
\sup_{0\leq t\leq T} d_{\infty}^{(p)}\bigl(
\nu^{(N)}_{t}, \nu _{t}\bigr)\leq C %
\cases{\displaystyle \frac{1}{N^{\gamma\wedge1}}, & \quad
$\mbox{if } \displaystyle\alpha\in \biggl[0, \frac
{d}{2} \biggr),$
\vspace*{2pt}
\cr
\displaystyle\frac{\ln N}{N^{{d}/{2}\wedge1}}, & \quad $\mbox{if } \displaystyle\alpha=\frac
{d}{2},$
\vspace*{2pt}
\cr
\displaystyle\frac{\ln N}{N^{(d-\alpha)\wedge1}}, & \quad$\mbox{if } \alpha\in \biggl(
\displaystyle\frac
{d}{2}, d \biggr),$} %
\end{equation}
where the constant $C>0$ only depends on $T$, $\Gamma$, $\Psi$,
$\alpha
$ and $c$.
\end{theorem}
Note that the speed of convergence found in Theorem~\ref
{theo:LLNpowerinfd} is never smaller than $ N^{-{d}/{2}}$ which is
the optimal speed for the case without spatial extension; recall the
CLT results in the mean field case in \cite{Lucon2011}. Note also that,
in the case where $0\leq\alpha< \frac{d}{2}$, we have obtained a speed
of convergence which is arbitrarily close to $N^{-({d}/{2}\wedge
1)}$ (since in that case $\gamma$ is arbitrarily close to $ \frac
{d}{2}$). We believe that the optimal speed in this case should be
\emph
{exactly} $N^{- ({d}/{2}\wedge1)}$, but the proof we propose in
this work does not seem to reach this optimal result.

Nevertheless, in the case where we only consider a \emph{bounded
Lispchitz-continuous} weight function $\Psi$ (i.e., with no singularity
at all), the proof of Theorem~\ref{theo:LLNpowerinfd} can be
considerably simplified and one obtains a speed that is $N^{- {d}/{2}}$.

Note also that the fluctuations when $\alpha\in [\frac{d}{2},
d )$ appear to be nontrivial. A natural perspective of this work
would be to prove a precise central limit theorem in this case and to
study the limiting fluctuation process in details.

\subsection{Well-posedness of the McKean--Vlasov equation}
A straightforward corollary of Theorems \ref{theo:LLNPnn} and \ref
{theo:LLNpowerinfd} is that uniqueness holds for the McKean--Vlasov
equation \eqref{eq:nut}:
%
%
\begin{proposition}[(Well-posedness of the McKean--Vlasov equation)]
Under Assumptions \ref{ass:Gammac}, \ref{ass:muzeta} and \ref{ass:psi},
for every initial condition $\nu_{0}(\mathrm{d}\tta,\break  \mathrm
{d}\omega, \mathrm{d}x)=
\zeta(\mathrm{d}
\tta)\mu(\mathrm{d}\omega)\,\mathrm{d}x$, there exists a unique solution
$t\mapsto\nu
_{t}\in\cM_{1} (\cC([0, T],\break  \cX)\times\cE\times [-
\frac
{1}{2}, \frac{1}{2} ]^{d} )$ to the McKean--Vlasov
equation \eqref{eq:nut}.
\end{proposition}

\section{The nonlinear process and the existence of a continuous-limit}
\label{sec:nonlin}
The purpose of this paragraph is to prove Proposition~\ref
{prop:existencenut} concerning the existence of a solution to the
McKean--Vlasov equation \eqref{eq:nut}. This part is reminiscent of the
techniques used by Sznitman \cite{SznitSflour} in order to prove
propagation of chaos for nondisordered models.

\subsection{Distance on probability measures}
Let us first consider the set $\cM_{\cX}$ of probability measures on
$\cC([0,T], \cX)$ with finite moments of order $\kappa$ [where
$\kappa
\geq2$ is given in \eqref{eq:polgrowthc}] and endow this set with the
Wasserstein metric
%
%
\begin{equation}
\label{eq:wasserstein} \delta_{\cX}^{(T)}(p_1,
p_2):=\inf \Bigl\{\bE \Bigl(\sup_{s\leq
T}\bigl\|
\vartheta_{s}^{(1)}-\vartheta_{s}^{(2)}
\bigr\|^{\kappa} \Bigr)^{
{1}/{\kappa}} \Bigr\},
\end{equation}
where the infimum in \eqref{eq:wasserstein} is considered over all
couplings $  (\vartheta^{(1)}, \vartheta^{(2)} )$ with
respective marginals $p_{1}$ and $p_{2}$. Here, the $\vartheta^{(i)}$
are understood as random variables on a certain probability space
$
( \Om, \bP )$. Note, however, that the definition of \eqref
{eq:wasserstein} does not depend on its particular choice. Equation~\eqref
{eq:wasserstein} defines a complete metric on $\cM_{\cX}$ encoding the
topology of convergence in law with convergence of moments up to order
$\kappa$; see \cite{MR2459454}, Theorem~6.9, page 96. We endow $\cM
_{\cX
}$ with the corresponding Borel $\sig$-field.

Fix some probability measure $m$ on $\cC([0, T], \cX)\times\cE
\times
[- \frac{1}{2}, \frac{1}{2} ]^{d}$ (endowed with its Borel $\sig
$-field) such that its marginal on $\cE\times [- \frac{1}{2},
\frac
{1}{2} ]^{d}$ is absolutely continuous w.r.t. $\mu(\mathrm
{d}\omega
)\otimes\mathrm{d}
x$. Thanks to a usual disintegration result (see, e.g., \cite{MR1932358},
Theorem~10.2.2) one can write $m$ as
\[
m(\mathrm{d}\tta, \mathrm{d}\omega, \mathrm{d}x)= m^{\omega,
x}(\mathrm{d}
\tta) \mu(\mathrm{d} \omega)\,\mathrm{d}x,
\]
where $(\omega, x)\mapsto m^{\omega, x}(\mathrm{d}\tta)$ is a measurable
map from
$\cE\times [- \frac{1}{2}, \frac{1}{2} ]^{d}$ (endowed with its
Borel $\sig$-field) into $\cM_{\cX}$. We consider the set $\cM$ of such
measures $m$ such that for all $(\omega, x)$, $m^{\omega, x}$
belongs to $\cM
_{\cX}$, endowed with the following metric:
%
%
\begin{definition}
\label{def:wass}
Fix $p$ to be equal to $2$ in the case of hypothesis (H1)
or as in \eqref{eq:p} in the case of hypothesis (H2).
Then define
%
%
\begin{eqnarray}
\forall m_{1}, m_{2}\in\cM
\nonumber
\\[-8pt]
\\[-8pt]
\eqntext{\displaystyle \delta_{T}(m_{1},
m_{2}):= \biggl[\int_{\cE\times [- {1}/{2}, {1}/{2} ]^{d}} \bigl(\delta
_{\cX
}^{(T)} \bigl(m_{1}^{\omega, x},
m_{2}^{\omega, x} \bigr) \bigr)^{p}\mu(\mathrm{d}\omega
)\,\mathrm{d}x \biggr]^{ {1}/{p}}.}
\end{eqnarray}
The space $\cM$ endowed with $\delta_{T}$ is a complete metric space;
see \cite{SznitSflour}, page~173.
\end{definition}

Note that, by construction [see \eqref{eq:nu0}], the initial condition
$\mathrm{d}\nu_{0}(\tta, \omega, x)= \zeta(\mathrm{d}\tta)\mu
(\mathrm{d}\omega
)\,\mathrm{d}x$ belongs to
$\cM$.

\subsection{The nonlinear process}
The proof of Proposition~\ref{prop:existencenut} is based on a Picard
iteration in the space $\cM$ endowed with the metric introduced in
Definition~\ref{def:wass}.
For fixed $\omega\in\cE$ and Brownian motion $B$ in $\cX$,
independent of
the sequence $(B_{k})_{k\geq1}$, and for a fixed $m\in\cM$, consider
the following stochastic differential equation in~$\cX$:
%
%
\begin{eqnarray}
\label{eq:odem} \mathrm{d}\tta(t) &=& c\bigl(\tta(t), \omega\bigr)\,\mathrm{d}t
\nonumber
\\[-8pt]
\\[-8pt]
\nonumber
&&{}+
\int\Gamma \bigl(\tta(t), \omega, \bar\tta,\bar\omega\bigr) \Psi(x, \bar x)
m_{t}(\mathrm{d}\bar\tta,\mathrm {d}\bar\omega, \mathrm{d}\bar x)\,
\mathrm{d}t + \sig\cdot\mathrm{d}B(t),
\end{eqnarray}
with initial condition $\tta(0)\sim\zeta$. Note here that for all
$t\geq0$, $m_{t}(\mathrm{d}\tta, \mathrm{d}\omega, \mathrm{d}x)$,
probability measure
on $\cX
\times\cE\times [- \frac{1}{2}, \frac{1}{2} ]^{d}$, stands for
the projection of $m$ at time $t$.
The integral term in \eqref{eq:odem} is well defined since
\begin{eqnarray*}
&&\int\bigl\| \Gamma\bigl(\tta(t), \omega, \bar\tta,\bar\omega\bigr) \bigr\| \Psi (x, \bar
x) m_{t}(\mathrm{d}\bar\tta,\mathrm{d}\bar\omega, \mathrm{d}\bar x)
\\
&&\qquad\leq\| \Gamma \|_{\infty} \int_{
[- {1}/{2}, {1}/{2} ]^{d}}\Psi(x, \bar x)
\underbrace {\int_{\cX\times\cE}m_{t}^{\bar\omega, \bar x}(
\mathrm{d}\bar\tta )\mu(\mathrm{d} \bar\omega )}_{=1}\,\mathrm{d}\bar x
\leq\| \Gamma \|_{\infty} S(\Psi),
\end{eqnarray*}
where the quantity
%
%
\begin{equation}
\label{eq:SPsi} S(\Psi):=\sup_{x} \int_{ [- {1}/{2}, {1}/{2}
]^{d}}
\Psi (x, \bar x)\,\mathrm{d}\bar x
\end{equation}
is smaller than $1$ in case of hypothesis (H1) and smaller
that $\cI_{1}(\Psi)$ [using \eqref{eq:discont}] in the case of
hypothesis (H2). Moreover, thanks to the regularity
properties of $\Gamma$ and $c$, equation \eqref{eq:odem} has a unique
(strong) solution.

Let us denote by $\Theta\dvtx\cM\to\cM$ the functional which maps any
measure $m(\mathrm{d}\tta, \mathrm{d}\omega, \mathrm{d}x)\in\cM$
to the law $\Theta
(m)$ of
$(\tta,
\omega, x)$ where $(\tta_t)_{0\leq t\leq T}$ is the unique solution to
\eqref{eq:odem}. Note that the functional $\Theta$ effectively
preserves the set $\cM$. Proposition~\ref{prop:existencenut} is a
direct consequence of the following lemma:
%
%
\begin{lemma}
\label{lem:fixedpoint}
The functional $\Theta$ admits a fixed point $\bar\nu$ in $\cM$.
\end{lemma}
\begin{pf}
As in \cite{SznitSflour}, we prove the following:
%
%
\begin{eqnarray}
\label{eq:Thetacontr} \forall m_{1}, m_{2}\in\cM, \forall t\leq
T
\nonumber
\\[-8pt]
\\[-8pt]
\eqntext{\displaystyle\delta_t\bigl(\Theta(m_1), \Theta(m_2)
\bigr)^{\kappa}\leq C_T \int_0^t
\delta_u(m_1, m_2)^{\kappa
}\,
\mathrm{d}u.}
\end{eqnarray}
If \eqref{eq:Thetacontr} is proved, the proof of Proposition~\ref
{prop:existencenut} will be finished since in that case, one can
iterate this inequality and find
\[
\forall k\geq1 \qquad\delta_T\bigl(\Theta^{k+1}(
\nu_{0}), \Theta^k(\nu _{0})
\bigr)^{\kappa} \leq C_T^k\frac{T^k}{k!}
\delta_T\bigl(\Theta(\nu_{0}), \nu _{0}
\bigr)^{\kappa},
\]
which gives that $ ( \Theta^k(\nu_{0})  )_{k\geq1}$ is a
Cauchy sequence, and thus converges to some fixed-point $\bar\nu$ of
$\Theta$. Let us now prove \eqref{eq:Thetacontr}. The key calculation
is the following: there exists a constant $C>0$ such that for all
$\theta_{1}, \theta_{2}\in\cX$, $\omega\in\cE$, $x\in [-
\frac{1}{2},
\frac{1}{2} ]^{d}$, for all $m_{1}, m_{2}\in\cM$,
%
%
\begin{eqnarray}
\label{aux:GammaTheta} \delta\Gamma&:=&\biggl\| \int\Gamma (\theta_{1}, \omega,
\cdot, \cdot ) \Psi(x, \cdot) \,\mathrm{d}m_{1, t} - \int\Gamma (
\theta_{2}, \omega, \cdot, \cdot ) \Psi(x, \cdot) \,\mathrm
{d}m_{2, t} \biggr\|
\nonumber
\\[-8pt]
\\[-8pt]
\nonumber
&\leq &C \bigl(\| \tta _{2}-\tta_{1} \|\wedge1 +
\delta_{t}(m_{1}, m_{2}) \bigr).
\end{eqnarray}
Indeed,
%
%
\begin{eqnarray}
\label{aux:Theta1} \delta\Gamma&\leq&\biggl\| \int\Gamma (\theta_{1}, \omega,
\cdot, \cdot ) \Psi(x, \cdot) \,\mathrm{d}m_{1, t} - \int\Gamma (\theta
_{2}, \omega, \cdot, \cdot ) \Psi(x, \cdot) \,\mathrm{d}m_{1, t}
\biggr\|
\nonumber\\
&&{} +\biggl\| \int\Gamma (\theta_{2}, \omega, \cdot, \cdot ) \Psi(x, \cdot)
\,\mathrm{d}m_{1, t} - \int\Gamma (\theta_{2}, \omega, \cdot,
\cdot ) \Psi(x, \cdot) \,\mathrm{d}m_{2, t} \biggr\|\\
&:=& \delta
\Gamma_{1} + \delta \Gamma_{2}.
\nonumber
\end{eqnarray}
The first term $\delta\Gamma_{1}$ in \eqref{aux:Theta1} is easily
bounded by $\| \Gamma \|_{\mathrm{Lip}}S(\Psi)\| \tta_{2}- \tta
_{1} \|$, where
$S(\Psi)$ is defined by \eqref{eq:SPsi}. The second term $\delta
\Gamma
_{2}$ in \eqref{aux:Theta1} can be successively bounded by
\begin{eqnarray*}
\delta\Gamma_{2}&=& \biggl\| \int_{ [- {1}/{2}, {1}/{2}
]^{d}\times\cE} \Psi(x, \bar x)
\biggl(\int\Gamma (\theta _{2}, \omega, \bar\tta, \bar\omega )
m_{1, t}^{\bar\omega,
\bar x}(\mathrm{d} \bar\tta) \\
&&\hspace*{108pt}{}- \int \Gamma (
\theta_{2}, \omega, \bar\tta, \bar\omega ) m_{2, t}^{\bar\omega, \bar
x}(
\mathrm{d}\bar\tta) \biggr)\,\mathrm{d}\bar x\mu(\mathrm{d}\bar \omega) \biggr\|
\\
&\leq& \biggl(\int_{ [- {1}/{2}, {1}/{2} ]^{d}} \Psi(x, \bar x)^{q}\,
\mathrm{d}\bar x \biggr)^{ {1}/{q}}
\\
&&{}\times \biggl(\int_{ [- {1}/{2}, {1}/{2} ]^{d}\times\cE} \biggl\| \int \Gamma (\theta_{2},
\omega, \bar\tta, \bar\omega ) m_{1, t}^{\bar\omega, \bar x}(\mathrm{d}\bar
\tta) \\
&&\hspace*{85pt}{}- \int\Gamma (\theta_{2}, \omega, \bar \tta, \bar\omega )
m_{2,
t}^{\bar\omega, \bar x}(\mathrm{d}\bar\tta)\biggr \|^{p}\,
\mathrm{d}\bar x\mu(\mathrm{d}\bar \omega) \biggr)^{ {1}/{p}}.
\end{eqnarray*}
Note that the first term in the last inequality is always bounded: it
is straightforward in the $P$-nearest-neighbor case and comes from
Remark~\ref{rem:p} in the power-law case. Indeed, $q$ has been
precisely chosen so that $q\alpha<d$, so that $\Psi(x, \cdot)^{q}$ is
integrable.

Using the Lipschitz-continuity of $\Gamma$, we see that, for \emph{any}
coupling\break  $m^{\omega, x}(\mathrm{d}\vartheta_{1},  \mathrm
{d}\vartheta_{2})$ of
$m_{1}^{\omega, x}$ and $m_{2}^{\omega, x}$,
\begin{eqnarray*}
\delta\Gamma_{2}&\leq& C \| \Gamma \|_{\mathrm{Lip}} \biggl(\int
_{ [-
{1}/{2}, {1}/{2} ]^{d}\times\cE} \bigl(\bE_{m^{\omega,
x}}\bigl\| \vartheta
_{1}(t)-\vartheta_{2}(t)\bigr \| \bigr)^{p}\,\mathrm
{d}\bar x\mu(\mathrm{d}\bar\omega ) \biggr)^{{1}/{p}}
\\
&\leq& C \| \Gamma \|_{\mathrm{Lip}} \biggl(\int_{ [- {1}/{2}, {1}/{2}
]^{d}\times\cE}
\bigl( \bigl[\bE_{m^{\omega, x}}\bigl\| \vartheta _{1}(t)-\vartheta
_{2}(t) \bigr\|^{\kappa} \bigr]^{ {1}/{\kappa
}} \bigr)^{p}
\,\mathrm{d}\bar x\mu(\mathrm{d} \bar\omega) \biggr)^{ {1}/{p}}.
\end{eqnarray*}
By Definition~\ref{def:wass}, this gives $\delta\Gamma_{2}\leq C \|
\Gamma \|_{\mathrm{Lip}} \delta_{t}(m_{1}, m_{2})$, which proves
\eqref
{aux:GammaTheta}.
We are now in position to prove \eqref{eq:Thetacontr}. Let us consider
$(\tta_{1}, \omega, x)$ and $(\tta_2, \omega, x)$ solutions to
\eqref
{eq:odem} for two different measures $m_{1}$ and $m_{2}$ in $\cM$
driven by the same Brownian motion, with the same initial condition. We
have for all $0\leq t\leq T$,
\begin{eqnarray*}
&&\bigl\| \tta_{1}(t)-\tta_{2}(t) \bigr\|^{2}\\
&&\qquad= 2\int
_0^t \bigl\langle\tta _{1}(s)-\tta
_{2}(s), c\bigl(\tta_{1}(s), \omega\bigr)-c\bigl(
\tta_{2}(s), \omega \bigr)\bigr\rangle\,\mathrm{d}s
\\
&&\quad\qquad{}+ 2\int_{0}^{t} \biggl\langle
\tta_{1}(s)-\tta_{2}(s), \int\Gamma \bigl(
\theta_{1}(s), \omega, \cdot, \cdot \bigr) \Psi(x, \cdot) \,
\mathrm{d}m_{1} \\
&&\quad\qquad\hspace*{88pt}{}- \int\Gamma \bigl(\theta _{2}(s), \omega,
\cdot, \cdot \bigr) \Psi(x, \cdot) \,\mathrm{d} m_{2}\biggr\rangle\,
\mathrm{d}s.
\end{eqnarray*}
Using the one-sided Lipschitz condition \eqref{eq:cgrowthcond} and
\eqref{aux:GammaTheta}, we obtain
\begin{eqnarray*}
&&\bigl\| \tta_{1}(t)-\tta_{2}(t) \bigr\|^{2}\\
&&\qquad\leq C \int
_0^t \bigl\| \tta _{1}(s)-\tta
_{2}(s) \bigr\|^{2} \,\mathrm{d}s + C\int_0^t
\bigl\| \tta _{1}(s)-\tta_{2}(s)\bigr \| \delta _{s}(m_{1},
m_{2})\,\mathrm{d}s
\\
&&\qquad\leq C \int_0^t \bigl\| \tta_{1}(s)-
\tta_{2}(s)\bigr \|^{2}\,\mathrm{d}s + C\int_0^t
\delta _{s}(m_{1}, m_{2})^{2}\,
\mathrm{d}s.
\end{eqnarray*}
Consequently, using Gronwall's lemma,
\[
\sup_{s\leq t}\bigl\| \tta_{1}(s)-\tta_{2}(s)
\bigr\|^{2}\leq C e^{CT}\int_0^t
\delta_{s}(m_{1}, m_{2})^{2}\,
\mathrm{d}s.
\]
Elevating this inequality to the power $ \frac{\kappa}{2}\geq1$ gives
\begin{eqnarray*}
\sup_{s\leq t}\bigl\| \tta_{1}(s)-\tta_{2}(s)
\bigr\|^{\kappa}&\leq& \bigl(C e^{CT} \bigr)^{ {\kappa}/{2}} \biggl(\int
_0^t \delta_{s}(m_{1},
m_{2})^{2}\,\mathrm{d}s \biggr)^{ {\kappa}/{2}}
\\
&\leq& \bigl(C e^{CT} \bigr)^{ {\kappa}/{2}} T^{ {(\kappa-2)}/{2}}\int
_0^t \delta_{s}(m_{1},
m_{2})^{\kappa}\,\mathrm{d}s,
\end{eqnarray*}
which gives
\[
\delta_{\cX}^{(t)}\bigl(\Theta(m_{1})^{\omega, x},
\Theta (m_{2})^{\omega,
x}\bigr)\leq \bigl(C e^{CT}
\bigr)^{ {1}/{2}} T^{ {(\kappa
-2)}/{2\kappa}} \biggl(\int_0^t
\delta_s(m_1, m_2)^{\kappa}\,
\mathrm {d}s \biggr)^{
{1}/{\kappa}}.
\]
Elevating this inequality to the power $p$ and integrating over
$\omega$
and $x$ leads to the desired result \eqref{eq:Thetacontr}. Lemma~\ref
{lem:fixedpoint} is proved.
\end{pf}
We are now in position to prove Proposition~\ref{prop:existencenut}.
\begin{pf*}{Proof of Proposition~\ref{prop:existencenut}}
It remains to prove that if $\bar\nu$ is a fixed point of $\Theta$,
then $\bar\nu$ is a solution to the weak formulation of the continuous
limit \eqref{eq:nut}. Indeed if $\bar\nu=\Theta(\bar\nu)$, one can
write $\bar\nu(\mathrm{d}\tta, \mathrm{d}\omega, \mathrm{d}x)=
\bar\nu^{\omega,
x}(\mathrm{d}\tta)\mu
(\mathrm{d}\omega)\,\mathrm{d}x$ where, for fixed $\omega, x$, $\bar
\nu
^{\omega, x}(\mathrm{d}\tta)$
is the law of the process solution to \eqref{eq:odem}. Applying It\^o's
formula, one obtains for all $f(\tta, \omega, x)$, $\cC^{2}$ w.r.t.
$\tta$
with bounded derivatives,
%
%
\begin{eqnarray}
\label{eq:Itofixedpoint} f\bigl(\tta(t), \omega, x\bigr)&=& f(\tta_{0}, \omega,
x) + \frac{1}{2} \int_{0}^{t} \div
_{\tta} \bigl(\sig\sig^{T}\nabla_{\tta}f \bigr)
\bigl(\tta(s), \omega, x\bigr)\,\mathrm{d}s\nonumber \\
&&{}+\int_{0}^{t}
\nabla_{\tta}f \cdot c\bigl(\tta(s), \omega\bigr)\,\mathrm {d}s
\nonumber
\\[-8pt]
\\[-8pt]
\nonumber
&&{}+ \int_{0}^{t}\nabla_{\tta}f \cdot\int
\Gamma\bigl(\tta(t), \omega, \bar \tta,\bar\omega \bigr) \Psi(x, \bar x) \bar
\nu_{t}^{\bar\omega, \bar x}(\mathrm{d}\bar \tta )\mu(\mathrm{d}\bar \omega)
\,\mathrm{d}\bar x\,\mathrm{d}s \\
&&{}+ \int_{0}^{t}
\nabla_{\tta} f\bigl(\tta(s), \omega, x\bigr) \cdot (\sig\,
\mathrm{d}B_{s}).\nonumber
\end{eqnarray}
Taking the expectation in \eqref{eq:Itofixedpoint} leads to \eqref
{eq:nut}. But in order to do so, we need to know that the term $\nabla
_{\tta}f(\tta, \omega, x) \cdot c(\tta, \omega)$ is integrable
w.r.t. the
measure $\bar\nu^{\omega, x}(\mathrm{d}\tta)\mu(\mathrm{d}\omega
)\,\mathrm{d}x$
(the other terms
are integrable, by assumptions on $f$). This is ensured by \eqref
{eq:asszeta}, the fact that (by construction) $\bar\nu^{\omega,
x}(\mathrm{d}\tta
)$ has finite moments up to order $\kappa$, and the fact that $\mu$ has
finite moment of order $\iota$; recall \eqref{eq:assmu}.
\end{pf*}

The rest of the document is devoted to provide a proof for
Theorems \ref
{theo:LLNPnn} and~\ref{theo:LLNpowerinfd}.

\section{Definition and properties of the propagator}
\label{sec:propagator}
For reasons that will be made clear in Remark~\ref{rem:backKolm} below,
we make in this section, as well as in Sections~\ref{sec:Pnn} and \ref
{sec:PLinfd}, some supplementary assumption on the regularity on the
dynamics $c$:

%
\begin{assmp}[(Additional regularity on $c$)]
\label{ass:cGlobal}
We assume that for all $\omega$, the function $\tta\mapsto c(\tta,
\omega)$
is globally Lispchitz continuous.
\end{assmp}
Of course, the FitzHugh--Nagumo case does not enter into the framework
of Assumption~\ref{ass:cGlobal}. Assumption~\ref{ass:cGlobal} is made
in order to ensure the existence of a backward Kolmogorov equation; see
Remark~\ref{rem:backKolm}. The purpose of Section~\ref{sec:loclip} will
be to discard this assumption.

In this section, the function $\Psi$ is either defined as in
hypotheses (H1) or as in~(H2).
We know from Proposition~\ref{prop:existencenut} that there exists at
least one measure-valued solution $t\mapsto\nu_{t}$ to the continuous
equation \eqref{eq:nut}. We fix once and for all one such solution. We
can then consider the stochastic differential equation
%
%
\begin{eqnarray}
\label{eq:odemeanfield} \mathrm{d}\tta(t) &=& c\bigl(\tta(t),
 \omega\bigr)\,\mathrm{d}t \nonumber\\
 &&{}+
\int\Gamma \bigl(\tta(t), \omega, \bar\tta,\bar\omega\bigr) \Psi(x, \bar x)
\nu_{t}(\mathrm{d}\bar\tta,\mathrm {d}\bar \omega, \mathrm{d}\bar x)\,
\mathrm{d} t + \sig\cdot\mathrm{d}B(t)
\\
&=:& c\bigl(\tta(t), \omega\bigr)\,\mathrm{d}t + v\bigl(t, \tta(t), \omega, x
\bigr)\,\mathrm{d}t+ \sig \cdot\mathrm{d}B(t),
\nonumber
\end{eqnarray}
where $\tta(0)\sim\zeta$. Thanks to the regularity properties of
$\Gamma
$ and $c$ and to the integrability of $\Psi$, \eqref{eq:odemeanfield}
has a unique solution. Define the propagator corresponding to~\eqref
{eq:odemeanfield}
%
%
\begin{equation}
\label{eq:propagator} \forall s,t\in[0, T]\qquad P_{s,t}f(\tta, \omega, x):=
\bE_{B} f\bigl(\Phi _{s}^{t}(\tta; \omega, x),
\omega, x\bigr),
\end{equation}
where $\bE_{B}$ is the expectation w.r.t. the Brownian motion $B$, $f$ is
a bounded measurable function on $\cX\times\cE\times [- \frac{1}{2},
\frac{1}{2} ]^{d}$, $0\leq s\leq t$ and $t\mapsto\Phi
_{s}^{t}(\tta;
\omega, x)$ is the unique solution to \eqref{eq:odemeanfield} such that
$\Phi_{s}^{s}(\tta; \omega, x)=\tta$.
%
%
\begin{remark}
\label{rem:backKolm}
If $f$ is $\cC^{2}$ w.r.t. the variable $\tta$, under Assumptions
\ref
{ass:Gammac} and \ref{ass:cGlobal} made about $c$ and $\Gamma$, it is
standard to see that the function $P_{s, t}f$ is of class $\cC^{2}$ in
$\tta$ and $\cC^{1}$ in $s$ and satisfies the backward Kolmogorov
equation (see, e.g., \cite{doi10.108007362999808809576}, Remark~2.3)
%
%
\begin{eqnarray}
\label{eq:backKolm} &&\forall(\tta, \omega, x, s, t)\qquad \partial_{s}
P_{s, t}f (\tta, \omega, x) + \tfrac{1}{2} \div_{\tta}
\bigl(\sig\sig^{T}\nabla_{\tta
}P_{s,t} \bigr) (\tta,
\omega, x)
\nonumber\\
&&\quad\hspace*{85pt}{}
+ \bigl( \bigl[c(\tta, \omega) + v(t, \tta, \omega, x) \bigr]\cdot
\nabla_{\tta} \bigr)P_{s,t}f(\tta, \omega, x) \\
&&\qquad\hspace*{85pt}{}=0.\nonumber
\end{eqnarray}
The main problem which motivates the work of Section~\ref{sec:loclip}
at the end of this paper is that proving similar Kolmogorov when
Assumption~\ref{ass:cGlobal} is discarded appears to be difficult; see,
in particular, the recent work in this direction \cite{1209.6035}.
Nevertheless, we work in this section under this additional hypothesis,
and we provide in Section~\ref{sec:loclip} a way to bypass this
technical difficulty.
\end{remark}
The key calculation of this work is the object of Lemma~\ref{lem:varf}:
%
%
\begin{lemma}
\label{lem:varf}
Let $f\dvtx\cX\times\cE\times [- \frac{1}{2}, \frac
{1}{2}
]^{d}\to\bR$ be $\cC^{2}$ w.r.t. the variable $\tta$. Then
%
%
\begin{eqnarray}
\label{eq:varf} \bigl\langle f, \nu_{T}^{(N)}-
\nu_{T}\bigr\rangle &=& \bigl\langle P_{0, T}f, \nu
_{0}^{(N)}-\nu _{0}\bigr\rangle \nonumber\\
&&{}+
\frac{1}{ \llvert \gL_{N} \rrvert } \sum_{k}\int
_{0}^{T} \nabla_{\tta}
(P_{t, T}f ) \bigl(\tta_{k}(t), \omega_{k},
x_{k} \bigr) \cdot\sig\,\mathrm{d}B_{k}(t)
\nonumber
\\[-8pt]
\\[-8pt]
\nonumber
&&{}+ \frac{1}{ \llvert
\gL_{N}
\rrvert } \sum_{k}\int
_{0}^{T} \nabla_{\tta}
(P_{t,
T}f ) \bigl(\tta_{k}(t), \omega_{k},
x_{k} \bigr) \\
&&\hspace*{70pt}{}\times  \bigl[\bigl\langle\Gamma (\tta_{k},
\omega_{k}, \cdot, \cdot )\Psi (x_{k}, \cdot ),
\nu_{t}^{(N)} - \nu_{t}\bigr\rangle \bigr]\,
\mathrm{d}t.\nonumber
\end{eqnarray}
\end{lemma}
\begin{pf}
An application of It\^o's formula gives the following: for all $k$ and $0<t<T$,
\begin{eqnarray*}
P_{t, T}f \bigl(\tta_{k}(t), \omega_{k},
x_{k} \bigr) &=& P_{0, T}f \bigl(\tta_{k}(0),
\omega_{k}, x_{k} \bigr) + \int_{0}^{t}
\partial _{s} P_{s,
T}f \bigl(\tta_{k}(s),
\omega_{k}, x_{k} \bigr) \,\mathrm{d}s
\\
&&{}+ \int_{0}^{t} \nabla_{\tta}
P_{s, T}f \bigl(\tta_{k}(s), \omega_{k},
x_{k} \bigr) \cdot \,\mathrm{d}\tta_{k}(s)\\
&&{} +
\frac{1}{2} \int_{0}^{t}
\div_{\tta} \bigl(\sig \sig ^{T} \nabla_{\tta}P_{s, T}f
\bigr) \bigl(\tta_{k}(s), \omega_{k}, x_{k}
\bigr) \,\mathrm{d}s.
\end{eqnarray*}
Using the definition of $\tta_{k}$ [recall \eqref{eq:odegene}] and
\eqref{eq:backKolm} we obtain
\begin{eqnarray*}
&&P_{t, T}f \bigl(\tta_{k}(t), \omega_{k},
x_{k} \bigr)\\
&&\qquad= P_{0, T}f \bigl(\tta_{k}(0),
\omega_{k}, x_{k} \bigr) \\
&&\qquad\quad{}- \int_{0}^{t}
v\bigl(s, \tta _{k}(s), \omega _{k}, x_{k}
\bigr)\cdot\nabla_{\tta} P_{s, T}f \bigl(\tta_{k}(s),
\omega_{k}, x_{k} \bigr) \,\mathrm{d}s
\\
&&\qquad\quad{}+ \int_{0}^{t} \nabla_{\tta}
P_{s, T}f \bigl(\tta _{k}(s), \omega_{k},
x_{k} \bigr) \cdot\bigl\langle\Gamma (\tta _{k}, \omega
_{k}, \cdot, \cdot )\Psi (x_{k}, \cdot ), \nu
_{s}^{(N)}\bigr\rangle\,\mathrm{d}s
\\
&&\qquad\quad{}+ \int_{0}^{t} \nabla_{\tta}
P_{s,t} f\bigl(\tta_{k}(s), \omega_{k},
x_{k}\bigr)\cdot\bigl(\sig\,\mathrm{d}B_{k}(s)\bigr).
\end{eqnarray*}
Then, using the definition of $v(\cdot)$ [recall \eqref
{eq:odemeanfield}] and summing over $k$ lead to
\begin{eqnarray*}
\bigl\langle P_{t, T}f, \nu_{t}^{(N)}\bigr\rangle
&=& \bigl\langle P_{0, T}f, \nu _{0}^{(N)}\bigr
\rangle + \frac
{1}{ \llvert \gL_{N} \rrvert }\sum_{k} \int
_{0}^{t} \nabla _{\tta}
P_{s,t} f\bigl(\tta_{k}(s), \omega_{k},
x_{k}\bigr)\cdot\bigl(\sig\,\mathrm {d}B_{k}(s)\bigr)
\\
&&{}+ \frac
{1}{ \llvert \gL_{N} \rrvert }\sum_{k}\int
_{0}^{t}\nabla _{\tta
}P_{s, T}f
\bigl(\tta_{k}(s), \omega_{k}, x_{k} \bigr)
\\
&&\hspace*{69pt}{}\times \bigl\langle \Gamma (\tta_{k}, \omega_{k}, \cdot,
\cdot )\Psi (x_{k}, \cdot ), \nu_{s}^{(N)} -
\nu_{s}\bigr\rangle \,\mathrm{d}s.
\end{eqnarray*}
A straightforward calculation using \eqref{eq:backKolm} shows that
$\partial_{t}\langle P_{t, T}f, \nu_{t}\rangle=0$. Using this and
the previous
equality, one obtains the desired result (choose $t=T$ and recall that
$P_{T,T}f=f$). Lemma~\ref{lem:varf} is proved.
\end{pf}
The purpose of the following lemma is to establish regularity
properties of the propagator $P_{t, T}$:
%
%
\begin{lemma}[(Estimates on the propagator $P_{t, T}$)]
\label{lem:propagatorLip}
Fix $T>0$, $0<t<T$ and $a\in [- \frac{1}{2}, \frac{1}{2} ]^{d}$.
\begin{longlist}[(1)]
\item[(1)] Assume $\Psi$ satisfies hypothesis \textup{(H1)}. For
any $R\in(0, 1]$ and any $f$ in $\cC_{R, a}$, $P_{t, T}f$ is also in
$\cC_{R, a}$, and one has the following estimate:
%
%
\begin{equation}
\label{eq:propLipPnn} \| P_{t, T}f \|_{R, a}\leq\sqrt{2}e^{|\!|\!|P|\!|\!|(T-t)}
\| f \| _{R, a}
\end{equation}
for some constant $|\!|\!|P|\!|\!|$ [that can be chosen equal to $L +
3/2 \| \Gamma \|_{\mathrm{Lip}}$; recall~\eqref{eq:cgrowthcond}].
\item[(2)] Assume $\Psi$ satisfies hypothesis \textup{(H2)}.
For every $a\in [- \frac{1}{2}, \frac{1}{2} ]^{d}$, for any $f$
in $\cC_{a}$, $P_{t, T}f$ is also in $\cC_{a}$, and one has the
following estimate:
%
%
\begin{equation}
\label{eq:propLipPLinfd} \| P_{t, T}f \|_{a}\leq|\!|\!|P|\!|
\!|e^{|\!|\!|P|\!|\!|(T-t)} \| f \|_{a}
\end{equation}
for some constant $|\!|\!|P|\!|\!|$ (that only depends on $\Gamma$,
$\Psi$ and $c$).
\end{longlist}
\end{lemma}
\begin{pf}
Note that, by a usual density argument, one only needs to prove~\eqref
{eq:propLipPnn} and \eqref{eq:propLipPLinfd} for test functions $f$
that are $\cC^{2}$ w.r.t. $\tta$.
Fix $T>0$, $0<t<T$, $a\in [- \frac{1}{2}, \frac{1}{2} ]^{d}$ and
consider two different flows for \eqref{eq:odemeanfield} $\Phi
_{s}^{t}(\tta_{i}; \omega_{i}, x)$, for $i=1,2$, with different initial
condition and parameter but at the same site $x$, with the same
Brownian motion. For simplicity,\vadjust{\goodbreak} we write $\Phi_{s}^{t}(i)$ instead of
$\Phi_{s}^{t}(\tta_{i}; \omega_{i}, x)$. Then, using the one-sided
Lipschitz condition \eqref{eq:cgrowthcond} on $c$, we obtain
\begin{eqnarray*}
&&\bigl\| \Phi_{s}^{t}(2)-\Phi_{s}^{t}(1)
\bigr\|^{2}\\
&&\qquad= \| \tta_{2}-\tta _{1} \| ^{2}
+ 2 \int_{s}^{t} \bigl\langle
\Phi_{s}^{u}(2)-\Phi_{s}^{u}(1), c
\bigl(\Phi _{s}^{u}(2), \omega_{2}\bigr) - c
\bigl(\Phi_{s}^{u}(1), \omega_{1}\bigr)\bigr
\rangle \,\mathrm{d}u
\\
&&\qquad\quad{}+ 2 \int_{s}^{t}\bigl\langle
\Phi_{s}^{u}(2)-\Phi_{s}^{u}(1), v
\bigl(u, \Phi _{s}^{u}(2), \omega_{2}, x\bigr) -
v\bigl(u, \Phi_{s}^{u}(1), \omega _{1}, x\bigr)
\bigr\rangle\,\mathrm{d}u
\\
&&\qquad\leq\| \tta_{2}-\tta_{1} \|^{2} + 2L \int
_{s}^{t} \bigl(\bigl\| \Phi _{s}^{u}(2)-
\Phi_{s}^{u}(1)\bigr \|^{2} + \| \omega_{2}-
\omega_{1} \|^{2} \bigr)\,\mathrm{d}u
\\
&&\qquad\quad{}+ 2 \int_{s}^{t}\bigl\| \Phi_{s}^{u}(2)-
\Phi_{s}^{u}(1) \bigr\| \underbrace {\bigl\| v\bigl(u,
\Phi_{s}^{u}(2), \omega_{2}, x\bigr) - v\bigl(u,
\Phi_{s}^{u}(1), \omega_{1}, x\bigr)
\bigr\|}_{:=\delta v(u)}\,\mathrm{d}u,
\end{eqnarray*}
where the definition of $v(\cdot)$ is given in \eqref{eq:odemeanfield}.
The Lipschitz-continuity of $\Gamma$ implies
\begin{eqnarray*}
\delta v(u)&\leq&\int\bigl\| \Gamma\bigl(\Phi_{s}^{u}(2),
\omega_{2}, \bar \tta, \bar \omega\bigr) - \Gamma\bigl(
\Phi_{s}^{u}(1), \omega_{1}, \bar \tta, \bar
\omega\bigr) \bigr\|\Psi(x, \bar x) \nu_{u}^{\bar\omega, \bar x}(\mathrm{d}\bar
\tta) \mu(\mathrm {d}\bar \omega) \,\mathrm{d}\bar x
\\
&\leq&\| \Gamma \|_{\mathrm{Lip}}S(\Psi) \bigl(\bigl\| \Phi_{s}^{u}(2)-
\Phi_{s}^{u}(1) \bigr\| + \| \omega_{2}-
\omega_{1} \| \bigr),
\end{eqnarray*}
where $S(\Psi)$ has already been defined in \eqref{eq:SPsi}.
Putting things together we see that, for $C= 2L + 3 \| \Gamma \|
_{\mathrm{Lip}}S(\Psi)$,
%
%
\begin{eqnarray}
\bigl\| \Phi_{s}^{t}(2)-\Phi_{s}^{t}(1)
\bigr\|^{2}&\leq&\| \tta_{2}-\tta _{1} \|^{2}
\nonumber
\\[-8pt]
\\[-8pt]
\nonumber
&&{}+ C \int_{s}^{t} \bigl(\bigl\| \Phi_{s}^{u}(2)-
\Phi_{s}^{u}(1)\bigr \|^{2} + \| \omega
_{2}-\omega_{1} \|^{2} \bigr)\,\mathrm{d}u.
\end{eqnarray}
An application of Gronwall's lemma leads to
%
%
\begin{eqnarray}
&&\bigl\| \Phi_{s}^{t}(\tta_{2}, \omega_{2},
x)-\Phi_{s}^{t}(\tta_{1}, \omega_{1},
x) \bigr\|^{2} + \| \omega_{2}-\omega_{1}
\|^{2}
\nonumber
\\[-8pt]
\\[-8pt]
\nonumber
&&\qquad\leq e^{C(t-s)} \bigl(\| \tta_{2} -
\tta_{1} \|^{2} + \| \omega _{2}-
\omega_{1} \|^{2} \bigr).
\end{eqnarray}
Then, in the case where $\Psi$ satisfies hypothesis (H1),
we have $P_{t, T}f(\tta, \omega, x)= \chi_{R}(x-a) g(\Phi
_{t}^{T}(\tta; \omega
, x), \omega)$, when $f(\tta, \omega, x)= \chi_{R}(x-a) g(\tta,
\omega)$. But then,
\begin{eqnarray*}
&&\bigl\| g\bigl(\Phi_{t}^{T}(\tta_{2};
\omega_{2}, x), \omega_{2}\bigr) - g\bigl(\Phi
_{t}^{T}(\tta _{1}; \omega_{1}, x),
\omega_{1}\bigr) \bigr\|^{2}\\
&&\qquad \leq\| f \| _{R, a}^{2}
\bigl(\bigl\| \Phi _{t}^{T}(2) - \Phi_{t}^{T}(1)
\bigr\| + \bigl\| \omega _{2}-\omega_{1} \bigr\| \bigr)^{2}
\\
&&\qquad\leq2\| f \|_{R, a}^{2} \bigl(\bigl\| \Phi_{t}^{T}(2)
- \Phi_{t}^{T}(1) \bigr\|^{2} + \|
\omega_{2}-\omega_{1} \|^{2} \bigr)
\\
&&\qquad\leq2\| f \|_{R, a}^{2} e^{C(T-t)} \bigl(\|
\tta_{2} - \tta _{1} \| ^{2} + \|
\omega_{2}-\omega_{1} \|^{2} \bigr),
\end{eqnarray*}
so that
\begin{eqnarray*}
&&\bigl\| g\bigl(\Phi_{t}^{T}(\tta_{2};
\omega_{2}, x), \omega_{2}\bigr) - g\bigl(\Phi
_{t}^{T}(\tta _{1}; \omega_{1}, x),
\omega_{1}\bigr) \bigr\|\\
&&\qquad\leq\sqrt{2}\| f \|_{R, a}e^{
({C}/{2})(T-t)}
\bigl(\| \tta_{2} - \tta_{1} \| + \| \omega _{2}-
\omega _{1} \| \bigr),
\end{eqnarray*}
which is the desired estimate \eqref{eq:PnnLipschitz} and gives \eqref
{eq:propLipPnn}. The same kind of calculation in the case of
hypothesis (H2) leads to the estimate \eqref
{eq:fLipttaom} for $P_{t,T}f$.

Thus, it remains to prove estimates \eqref{eq:falphabounded} and
\eqref
{eq:falphaHolder} for $P_{t,T}f$ in the case of hypothesis (H2). The case of \eqref{eq:falphabounded} is
straightforward. As far as \eqref{eq:falphaHolder} is concerned, the
same kind of calculation with two different flows $\Phi
_{s}^{t}(x):=\Phi
_{s}^{t}(\tta; \omega, x)$ and $\Phi_{s}^{t}(y):=\Phi_{s}^{t}(\tta
; \omega,
y)$, with the same $\tta$ and $\omega$ but at different sites $x$
and $y$
leads to
\begin{eqnarray*}
&&\bigl\| \Phi_{s}^{t}(x)-\Phi_{s}^{t}(y)
\bigr\|^{2}\\
&&\qquad\leq 2L \int_{s}^{t}\bigl\| \Phi
_{s}^{u}(x)-\Phi_{s}^{u}(y)
\bigr\|^{2}\,\mathrm{d}u
\\
&&\qquad\quad{}+ 2 \int_{s}^{t}\bigl\| \Phi _{s}^{u}(x)-
\Phi_{s}^{u}(y) \bigr\| \underbrace{\bigl\| v\bigl(u,
\Phi_{s}^{u}(x), \omega, x\bigr) - v\bigl(u,
\Phi_{s}^{u}(y), \omega, y\bigr) \bigr\| }_{:=\delta v(u, x, y)}\,
\mathrm{d}u,
\end{eqnarray*}
with
\begin{eqnarray*}
&&\delta v(u, x, y)\\[-2pt ]
&&\qquad\leq\int\bigl\| \Gamma\bigl(\Phi_{s}^{u}(x),
\omega, \bar \tta, \bar\omega\bigr)\Psi(x, \bar x) \\[-2pt ]
&&\hspace*{47pt}{}- \Gamma\bigl(
\Phi_{s}^{u}(y), \omega, \bar\tta, \bar\omega \bigr)\Psi(y,
\bar x) \bigr\| \nu_{u}^{\bar\omega,
\bar x}(\mathrm{d}\bar\tta) \mu (\mathrm{d}
\bar\omega )\,\mathrm{d}\bar x
\\[-2pt ]
&&\qquad\leq\int\bigl\| \Gamma\bigl(\Phi_{s}^{u}(x), \omega, \bar\tta,
\bar \omega\bigr) - \Gamma \bigl(\Phi_{s}^{u}(y), \omega,
\bar\tta, \bar\omega \bigr) \bigr\|\Psi(x, \bar x) \nu_{u}^{\bar
\omega, \bar x}(
\mathrm{d}\bar\tta) \mu(\mathrm{d}\bar\omega )\,\mathrm{d}\bar x
\\[-2pt ]
&&\qquad\quad{}+\int\bigl\| \Gamma\bigl(\Phi_{s}^{u}(y), \omega, \bar\tta,
\bar\omega \bigr) \bigr\|\bigl\vert\Psi(x, \bar x) - \Psi(y, \bar x)\bigr\vert
\nu_{u}^{\bar\omega, \bar x}(\mathrm{d} \bar\tta) \mu(\mathrm{d}\bar\omega)\,
\mathrm{d}\bar x
\\[-2pt ]
&&\qquad\leq\| \Gamma \|_{\mathrm{Lip}} S(\Psi) \bigl(\bigl\| \Phi_{s}^{u}(x)-
\Phi_{s}^{u}(y) \bigr\| \bigr)
\\[-2pt ]
&&\qquad\quad{}+\| \Gamma \|_{\infty} \int_{[-1, 1]^{d}} \bigl\vert\Psi(x, \bar
x) - \Psi (y, \bar x)\bigr\vert\underbrace{\int_{\cX\times\cE}
\nu_{u}^{\bar
\omega,
\bar x}(\mathrm{d}\bar\tta) \mu(\mathrm{d}\bar\omega
)}_{=1}\,\mathrm{d}\bar x
\\[-2pt ]
&&\qquad\leq\| \Gamma \|_{\mathrm{Lip}}S(\Psi)\bigl\| \Phi_{s}^{u}(x)-
\Phi _{s}^{u}(y) \bigr\|+\cI _{2}(\Psi)\| \Gamma
\|_{\infty}\| x- y \|^{(d-\alpha)\wedge1},
\end{eqnarray*}
where $S(\Psi)$ is defined in \eqref{eq:SPsi} and where we used
assumption \eqref{eq:intpsi}. This gives, for $C=2L + 2\| \Gamma \|
_{\mathrm{Lip}}S(\Psi) + \cI_{2}(\Psi) \| \Gamma \|_{\infty}$,
\begin{eqnarray*}
&&\bigl\| \Phi_{s}^{t}(x)-\Phi_{s}^{t}(y)
\bigr\|^{2}\\
&&\qquad\leq C\int_{s}^{t}\bigl\| \Phi
_{s}^{u}(x)-\Phi_{s}^{u}(y)
\bigr\|^{2}\,\mathrm{d}u +\cI_{2}(\Psi)\| \Gamma
\|_{\infty} (t-s) \| x- y \|^{2((d-\alpha)\wedge1)}.
\end{eqnarray*}
Consequently, by Gronwall's lemma,
%
%
\begin{eqnarray}
\label{eq:gronwallPhixy}&& \bigl\| \Phi_{s}^{t}(\tta; \omega, x)-
\Phi_{s}^{t}(\tta; \omega, y)\bigr \| ^{2}
\nonumber
\\[-8pt]
\\[-8pt]
\nonumber
&&\qquad\leq\cI
_{2}(\Psi)\| \Gamma \|_{\infty} (t-s)e^{C(t-s)}\| x-y \|
^{2((d-\alpha)\wedge1)}.
\end{eqnarray}
Then, for any $0<t\leq T$, we have
%
%
\begin{eqnarray}
\| \delta P_{t,T}f \|^{2}&:=&\bigl\| \| x-a \|^{2\gamma}
P_{t,T}f(\tta, \omega, x) - \| y-a \|^{2\gamma}
P_{t,T}f(\tta, \omega, y) \bigr\| ^{2}
\nonumber
\\
&= &\bigl\| \| x-a \|^{2\gamma} f\bigl(\Phi_{t}^{T}(\tta;
\omega, x), \omega, x\bigr) \nonumber\\
&&\hspace*{4pt}{}- \| y-a \|^{2\gamma } f\bigl(
\Phi_{t}^{T}(\tta; \omega, y), \omega, y\bigr)
\bigr\|^{2}
\nonumber
\\
&\leq& \bigl(\| x-a \|^{2\gamma} \bigl\| f\bigl(\Phi_{t}^{T}(
\tta; \omega, x), \omega, x\bigr) - f\bigl(\Phi_{t}^{T}(
\tta; \omega, y), \omega, x\bigr) \bigr\|
\nonumber
\\
&&{}+\bigl\| \| x-a \|^{2\gamma } f\bigl(\Phi_{t}^{T}(\tta;
\omega, y), \omega, x\bigr)\nonumber\\
\label{eq:ineqfalpha}&&\hspace*{104pt}{} - \| y-a \|^{2\gamma } f\bigl(\Phi
_{t}^{T}(\tta; \omega, y), \omega, y\bigr) \bigr\|
\bigr)^{2}
\nonumber
\\
&\leq&\| f \|_{a}^{2} \bigl(\bigl\| \Phi_{t}^{T}(x)
- \Phi_{t}^{T}(y) \bigr\| + \| x-y \|^{(2\gamma-\alpha)\wedge1}
\bigr)^{2}
\\
\label{eq:ineqfalpha2}
&\leq&2\| f \|_{a}^{2} \bigl(\bigl\| \Phi_{t}^{T}(x)
- \Phi _{t}^{T}(y) \bigr\| ^{2} + \| x-y
\|^{2((2\gamma-\alpha)\wedge1)} \bigr)
\nonumber\\
&\leq&2\| f \|_{a}^{2} \bigl(\cI_{2}(\Psi)\|
\Gamma \|_{\infty} (T-t)\vee 1 \bigr)e^{{C(T-t)}}
\nonumber
\\[-8pt]
\\[-8pt]
\nonumber
&&{}\times \bigl( \| x-y
\|^{2((d-\alpha)\wedge1)} + \| x-y \| ^{2((2\gamma-\alpha)\wedge1)} \bigr),\nonumber
\end{eqnarray}
where we used assumptions \eqref{eq:falphabounded} and \eqref
{eq:falphaHolder} in \eqref{eq:ineqfalpha} and estimation \eqref
{eq:gronwallPhixy} in~\eqref{eq:ineqfalpha2}. Using the definition of
$\gamma$ [recall \eqref{eq:assgamma}], it is always true that
$d-\alpha
\geq2\gamma-\alpha$. Consequently,
\begin{eqnarray*}
&&\bigl\| \| x-a \|^{2\gamma} P_{t,T}f(\tta, \omega, x) - \| y-a \|
^{2\gamma} P_{t,T}f(\tta, \omega, y) \bigr\|
\\
&&\qquad\leq2 \bigl(T\cI _{2}(\Psi)\| \Gamma \|_{\infty}\vee 1
\bigr)^{ {1}/{2}}e^{{ ({C}/{2})(T-t)}}\| f \|_{a} \| x-y \|
^{(2\gamma-\alpha)\wedge1},
\end{eqnarray*}
which leads to \eqref{eq:falphaHolder}. Lemma~\ref{lem:propagatorLip}
is proved.
\end{pf}

%
\begin{remark}
One could wonder why we have not simply used in the calculation above
the global Lipschitz assumption about $c$ (recall Assumption~\ref
{ass:cGlobal}), instead of the more involved one-sided Lipschitz
inequality used here. The crucial reason for this is that in order to
be able to discard Assumption~\ref{ass:cGlobal} in Section~\ref
{sec:loclip} below, we need to ensure that the estimates of Lemma~\ref
{lem:propagatorLip} do not depend on the modulus of continuity of $c$,
but only on its one-sided Lipschitz constant $L$.
\end{remark}
Using \eqref{eq:propLipPnn} [resp., \eqref{eq:propLipPLinfd}] in
\eqref
{eq:varf}, we easily see that for every $a\in [- \frac{1}{2},
\frac
{1}{2} ]^{d}$, for any given $f\in\cC_{R, a}$ with $\| f \|_{R,
a}\leq
1$ (resp., $f\in\cC_{a}$ with $\| f \|_{a}\leq1$), we have
%
%
\begin{eqnarray}
\label{eq:estimfnut}&& \bigl\| \bigl\langle f, \nu_{T}^{(N)}\bigr\rangle
- \langle f, \nu_{T}\rangle\bigr \| \nonumber\\
&&\qquad\leq\bigl\| \bigl\langle P_{0, T}f
, \nu _{0}^{(N)}\bigr\rangle - \langle P_{0, T}f,
\nu_{0}\rangle \bigr\|\nonumber
\nonumber
\\[-8pt]
\\[-8pt]
\nonumber
&&\qquad\quad{}+ \biggl\| \frac{1}{ \llvert \gL _{N} \rrvert } \sum_{k}\int
_{0}^{T} \nabla_{\tta}
(P_{t, T}f ) \bigl(\tta_{k}(t), \omega_{k},
x_{k} \bigr) \cdot\bigl(\sig\, \mathrm {d} B_{k}(t)\bigr)
\biggr\|
\\
&&\qquad\quad{}+ \frac{1}{\llvert \gL_{N} \rrvert } \sum_{k}\int
_{0}^{T} \| \nabla_{\tta}P_{t, T}f
\| \bigl\| \bigl\langle\Gamma (\tta _{k}, \omega _{k}, \cdot,
\cdot )\Psi (x_{k}, \cdot ), \nu _{t}^{(N)} -
\nu_{t}\bigr\rangle \bigr\|\,\mathrm{d}t.\nonumber
\end{eqnarray}
Using \eqref{eq:nablafR} and \eqref{eq:propLipPnn} [resp., \eqref
{eq:nablafalpha} and \eqref{eq:propLipPLinfd}], the term $\| \nabla
_{\tta}P_{t, T}f \| (\tta_{k}(t),\break  \omega_{k}, x_{k} )$ in
the third
summand of \eqref{eq:estimfnut} can be bounded by $\sqrt{2}e^{|\!|\!|P
|\!|\!|(T-t)}\| \chi_{R} \|_{\infty}$ in case of hypothesis (H1)
and by $\| x_{k} - a \|^{-\alpha}|\!|\!|P |\!|\!|e^{|\!|\!|P |\!|\!|
(T-t)}$ in case of hypothesis (H2). In both cases, the
bound that we find can be written in the form\looseness=-1
%
%
\begin{equation}
\label{eq:rho} \| \nabla_{\tta}P_{t, T}f \| \bigl(
\tta_{k}(t), \omega_{k}, x_{k} \bigr)\leq
e^{|\!|\!|P |\!|\!|(T-t)} \rho(x_{k})
\end{equation}\looseness=0
($\rho$ is a constant in the first case and proportional to $\| x_{k}
- a \|^{-\alpha}$ in the second). In particular, it is uniform in $f$ and
$(\tta_{k}, \omega_{k})$. Let us now fix the integer $p$ equal to
$2$ in
the case of hypothesis (H1) or defined as in \eqref{eq:p}
in the case of hypothesis (H2). Elevating inequality
\eqref{eq:estimfnut} to the power $p$ and taking the expectation lead to
%
%
\begin{eqnarray}
\label{eq:normpK} &&\hspace*{-6pt}\frac{1}{3^{p-1}} \bE\bigl\| \bigl\langle f, \nu_{T}^{(N)}-
\nu_{T}\bigr\rangle \bigr\|^{p}\nonumber \\
&&\hspace*{-10pt}\qquad\leq \bE\bigl\| \bigl\langle
P_{0, T}f, \nu_{0}^{(N)}-\nu_{0}\bigr
\rangle \bigr\|^{p}
\nonumber
\\[-4pt]
\\[-10pt]
\nonumber
&&\hspace*{-10pt}\qquad\quad{}+ \bE\biggl\| \frac {1}{ \llvert \gL_{N} \rrvert } \sum_{k}\int
_{0}^{T} \nabla _{\tta}
(P_{t, T}f ) \bigl(\tta_{k}(t), \omega_{k},
x_{k} \bigr) \cdot \bigl(\sig\,\mathrm{d}B_{k}(t)\bigr)
\biggr\|^{p}
\\
&&\hspace*{-10pt}\qquad\quad{}+ \bE\biggl\llvert \frac{1}{ \llvert
\gL
_{N} \rrvert } \sum_{k}
\int_{0}^{T} e^{|\!|\!|P |\!|\!|(T-t)} \rho
(x_{k}) \bigl\| \bigl\langle\Gamma (\tta_{k},
\omega_{k}, \cdot, \cdot )\Psi (x_{k}, \cdot ),
\nu_{t}^{(N)} - \nu _{t}\bigr\rangle \bigr\|\,
\mathrm{d}t \biggr\rrvert ^{p}.
\nonumber
\end{eqnarray}
Let us concentrate on the third term of the last inequality that we
denote by $D_{N}$. By successive use of H\"older's inequality (recall
that $ \frac{1}{p}+ \frac{1}{q}=1$), one has
%
%
\begin{eqnarray}
\label{aux:DN} D_{N}&\leq& \biggl(\int_{0}^{T}
e^{q |\!|\!|P |\!|\!|(T-t)}\,\mathrm{d} t \biggr)^{ {p}/{q}}\nonumber\\
&&{}\times \bE\int
_{0}^{T} \biggl\llvert \frac{1}{ \llvert
\gL
_{N} \rrvert } \sum
_{k} \rho(x_{k})\bigl \| \bigl\langle\Gamma
(\tta_{k}, \omega_{k}, \cdot, \cdot )\Psi
(x_{k}, \cdot ), \nu _{t}^{(N)} -
\nu_{t}\bigr\rangle \bigr\|\biggr\rrvert ^{p}\,\mathrm{d}t
\nonumber
\\[-8pt]
\\[-8pt]
\nonumber
&\leq& \biggl(\frac{e^{q|\!|\!|P |\!|\!|T}-1}{q |\!|\!|P |\!|\!|
} \biggr)^{ {p}/{q}} \biggl(\frac{1}{ \llvert \gL_{N} \rrvert }
\sum_{k} \rho(x_{k})^{q}
\biggr)^{ {p}/{q}}\\
&&{}\times \int_{0}^{T}
\frac{1}{
\llvert \gL_{N} \rrvert } \sum_{k} \bE\bigl\| \bigl\langle
\Gamma (\tta_{k}, \omega_{k}, \cdot, \cdot )\Psi
(x_{k}, \cdot ), \nu _{t}^{(N)} -
\nu_{t}\bigr\rangle \bigr\|^{p}\,\mathrm{d}t.\nonumber
\end{eqnarray}

At this point, here are the main steps of proof that we will follow in
the remainder of this paper: we have built the spaces of test functions
(recall Definitions \ref{def:spacePnearest} and \ref
{def:spacepowerinfd}) in such a way that they precisely include the
functions $(\tta, \omega, x)\mapsto\Gamma (\tta_{k}, \omega
_{k}, \tta,
\omega )\Psi (x_{k}, x  )$ for all $k$ (in this
case, $a$
is equal to $x_{k}$). Since the distances between two random measures
introduced in Definitions \ref{def:distancePnn} and~\ref
{def:distancePLinfd} are exactly the suprema of evaluations over all
such test functions, we are thus able to bound the term within the
integral in \eqref{aux:DN} in terms of the distance between $\nu^{(N)}$
and $\nu$.

The second point of the proof is to obtain an estimate (uniform in $f$)
of the speed of convergence to $0$ of the two first terms in \eqref
{eq:normpK}. Taking the supremum over all test functions $f$ and
applying Gronwall's lemma lead to the conclusion.

Those steps are somehow easy to follow in the $P$-nearest-neighbor case
(see Section~\ref{sec:Pnn}) but are more technically demanding in the
power-law case; see Section~\ref{sec:PLinfd}.

\section{Law of large numbers in the $P$-nearest-neighbor case}
\label{sec:Pnn}
The purpose of this section is to prove Theorem~\ref{theo:LLNPnn}. Thus
throughout this section, we suppose that $\Psi$ satisfies
hypothesis (H1) for some $R\in(0, 1]$. In this case, the
integer $p$ introduced in \eqref{eq:normpK} is equal to $2$, and the
function $\rho$ in \eqref{eq:rho} is bounded (equal to $\sqrt {2}\| \chi_{R} \|_{\infty}$). In particular, the two terms in front
of the integral in
\eqref{aux:DN} are trivially bounded by a constant, equal to $\frac
{e^{2|\!|\!|P |\!|\!|T}-1}{2|\!|\!|P |\!|\!|} \| \chi_{R} \|_{\infty}^{2}$.

The following proposition proves the convergence to $0$ of the first
term in \eqref{eq:normpK} together with explicit rates:
%
%
\begin{proposition}[(Convergence of the initial condition)]
\label{prop:initialcondPnn}
There exists a numerical constant $C_{1}>0$ (independent of $R$) such
that for all $f\in\break  \bigcup_{a\in [- {1}/{2}, {1}/{2}
]^{d}}\cC_{R, a}$ with $\| f \|_{R, a}\leq1$ and $\| f \|_{\infty
}\leq1$,
%
%
\begin{equation}
\label{eq:initialcondPnn} \bE\bigl\| \bigl\langle P_{0, T}f, \nu_{0}^{(N)}
\bigr\rangle - \langle P_{0, T}f, \nu _{0}\rangle
\bigr\|^{2}\leq\frac{C_{1}}{N^{d\wedge2}}.
\end{equation}
\end{proposition}
\begin{pf}
Recall that the couples $(\tta_{i}(0), \omega_{i})_{1\leq i\leq N}$ are
supposed to be i.i.d. samples of the law $\zeta(\mathrm{d}\tta
)\otimes\mu
(\mathrm{d}\omega
)$ on $\cX\times\cE$. Let $f\in\cC_{R, a}$: by definition, $f(\tta
, \omega
, x)= g(\tta, \omega) \chi_{R}(x-a)$ so that $P_{0,T}f=\chi(x-a) P_{0,
T}g$. We write $\vphi:=P_{0,T}g$ for simplicity. Then
%
%
\begin{eqnarray}
\label{aux:Pnn} \delta_{N}(f)&:=& \bE\bigl\| \bigl\langle P_{0, T}f
, \nu_{0}^{(N)}\bigr\rangle - \langle P_{0, T}f,
\nu_{0}\rangle \bigr\|^{2}
\nonumber
\\[-2pt ]
&=&\bE\biggl\| \frac
{1}{\llvert \Lambda_{N} \rrvert }\sum_{j} \vphi(
\tta_{j}, \omega_{j})\chi _{R}(x_{j}-a)\nonumber
\\[-2pt ]
&&\hspace*{54pt}{}- \int \vphi(\tta, \omega)\chi _{R}(x-a) \zeta(\mathrm{d}\tta)\mu(
\mathrm{d}\omega )\,\mathrm{d} x\biggr \|^{2}
\nonumber
\\[-2pt ]
&\leq&2\bE\biggl\| \chi_{R}(x_{j}-a)\frac{1}{\llvert \Lambda_{N}
\rrvert }\sum
_{j} \biggl( \vphi(\tta_{j},
\omega_{j}) - \int \vphi (\tta, \omega )\zeta(\mathrm{d}\tta)\mu(
\mathrm{d}\omega ) \biggr) \biggr\|^{2}
\nonumber
\\[-2pt ]
&&{}+2\biggl\| \int\vphi(\tta, \omega)\zeta(\mathrm{d}\tta)\mu(\mathrm {d}\omega )
\biggl(\frac{1}{\llvert \Lambda_{N} \rrvert
}\sum_{j}
\chi_{R}(x_{j}-a) \nonumber\\[-2pt ]
\nonumber\\[-2pt ]
&&\hspace*{170pt}{}- \int \chi _{R}(x-a)\,
\mathrm{d}x \biggr) \biggr\|^{2}
\\[-2pt ]
&\leq&\frac{2}{(2R)^{2d}}\bE\biggl\| \frac{1}{\llvert \Lambda_{N}
\rrvert }\sum
_{j} \biggl(\vphi(\tta_{j},
\omega_{j}) - \int \vphi (\tta, \omega )\zeta(\mathrm{d}\tta)\mu(
\mathrm{d}\omega ) \biggr) \biggr\|^{2}
\nonumber
\\[-2pt ]
&&{}+2\| \vphi \|_{\infty}^{2}\biggl\llvert \frac{1}{\llvert \Lambda_{N}
\rrvert }
\sum_{j} \chi_{R}(x_{j}-a)
- \int\chi_{R}(x-a)\,\mathrm {d}x\biggr\rrvert ^{2}
\nonumber
\\[-2pt ]
&:= &A_{N} +B_{N}.\nonumber
\end{eqnarray}
Since the $(\tta_{i}, \omega_{i})$ are i.i.d. random variables (with law
$\zeta\otimes\mu$), a standard calculation shows
\[
A_{N}= \frac{2}{\llvert \Lambda_{N} \rrvert ^{2}(2R)^{2d}} \sum_{j}
\bE\biggl\| \vphi(\tta_{j}, \omega_{j}) - \int\vphi(\tta, \omega
)\zeta(\mathrm{d} \tta)\mu(\mathrm{d}\omega) \biggr\|^{2}\leq
\frac
{8\| f \|_{\infty}^{2}}{2^{d} N^{d}},
\]
since $\| \vphi \|_{\infty}=\| P_{0, T}g \|_{\infty}=(2R)^{d}\| f \|
_{\infty}$ and $ \llvert
\gL_{N} \rrvert = (2N+1)^{d}\geq(2N)^{d}$.

Let us now turn to the case of the term $B_{N}$ in \eqref{aux:Pnn}. We
place ourselves in the case of nonperiodic boundary condition; recall
Remark~\ref{rem:boundcond}. The periodic case is simpler and left to
the reader. Let $a=  (a_{1}, \ldots, a_{d} )$. One has
%
%
\begin{equation}
\label{eq:prodIa}\qquad \int_{  [-{1}/{2}, {1}/{2} ]^{d}} \chi _{R}(x-a)\,
\mathrm{d}x = \prod_{l=1}^{d} \biggl(
\frac{1}{2R} \int_{- {1}/{2}}^{ {1}/{2}}
\mathbh{1}_{\llvert  x-a_{l} \rrvert \leq R}\,\mathrm{d}x \biggr):= \prod
_{l=1}^{d} \cI(a_{l}).
\end{equation}
In the same way,
\[
\frac{1}{\llvert \Lambda_{N} \rrvert }\sum_{j} \chi_{R}(x_{j}-a)
= \prod_{l=1}^{d} \Biggl(
\frac{1}{2R(2N+1)} \sum_{j=-N}^{N} \mathbh
{1}_{\llvert  x_{j} - a_{l} \rrvert \leq R } \Biggr):= \prod_{l=1}^{d}
\cI_{N}(a_{l}).
\]
Then, from the obvious equality,
\begin{eqnarray*}
&&\prod_{l=1}^{d} \cI_{N}(a_{l})
- \prod_{l=1}^{d} \cI(a_{l}) \\
&&\qquad=
\sum_{k=1}^{d} \cI_{N}(a_{1})
\cdots\cI_{N}(a_{k-1}) \bigl(\cI _{N}(a_{k})
- \cI(a_{k}) \bigr) \cI(a_{k+1})\cdots\cI(a_{d})
\end{eqnarray*}
and a recursion argument, one only needs to consider the case $d=1$ in
order to prove \eqref{eq:initialcondPnn}.
An easy calculation shows the following: for all $a\in [-\frac
{1}{2}, \frac{1}{2} ]$, for all $R\in (0, 1 ]$,
\begin{eqnarray*}
\cI(a)&=& \frac{1}{2R}\int_{- {1}/{2}}^{ {1}/{2}} \mathbh
{1}_{\llvert  x-a \rrvert \leq R}\,\mathrm{d}x\\
& =& %
\cases{\displaystyle \frac{1}{2R}
\biggl(R+ \frac{1}{2} +a \biggr), &\quad $\mbox{if } \displaystyle-\! \frac
{1}{2} \leq a
\leq- \frac{1}{2} +R,$\vspace*{2pt}
\cr
1, &\quad $\mbox{if } \displaystyle- \!\frac{1}{2}
+R \leq a \leq\frac{1}{2} - R,$\vspace*{2pt}
\cr
\displaystyle\frac{1}{2R}
\biggl(R+ \frac{1}{2} -a \biggr), &\quad $\mbox{if } \displaystyle\frac{1}{2} -R \leq a
\leq\frac{1}{2}.$} %
\end{eqnarray*}
Thus, in the one-dimensional case, we need to distinguish three cases,
depending on the position of $a\in [-\frac{1}{2}, \frac
{1}{2}
]$ w.r.t. $R$; we only treat the case $- \frac{1}{2} \leq a \leq-
\frac
{1}{2} +R$, the two\vadjust{\goodbreak} others being similar and left to the reader. In
this case, one has successively
\begin{eqnarray*}
\bigl\llvert \cI_{N}(a) - \cI(a) \bigr\rrvert ^{2} &=&
\frac{1}{4R^{2}} \Biggl\llvert \frac{1}{2N+1} \sum
_{j=-N}^{N} \mathbh{1}_{\llvert  j - 2aN
\rrvert \leq2RN } - \biggl(R+
\frac{1}{2} +a \biggr) \Biggr\rrvert ^{2}
\\
&=& \frac{1}{4R^{2}} \biggl\llvert \frac{1}{2N+1} \bigl(\bigl\lfloor2N(R+a)
\bigr\rfloor+N \bigr)- \biggl(R+ \frac{1}{2} +a \biggr) \biggr\rrvert
^{2}
\\
&\leq&\frac{(R+a)^{2}}{4R^{2}(1+2N)^{2}} \leq\frac
{(2R-1/2)^{2}}{16R^{2}N^{2}}\leq\frac{1}{4N^{2}}.
\end{eqnarray*}

Proposition~\ref{prop:initialcondPnn} is proved.
\end{pf}
We are now in position to prove Theorem~\ref{theo:LLNPnn}:
\begin{pf*}{Proof of Theorem~\ref{theo:LLNPnn}}
\label{proof:theoPnn}
Fix some $a\in [- \frac{1}{2}, \frac{1}{2} ]^{d}$ and some
$f\in
\cC_{R,a}$ such that $\| f \|_{R,a}\leq1$ and $\| f \|_{\infty}\leq
1$. Let us
first give an estimate of the second term in \eqref{eq:estimfnut}.
Recall that $B_{k}$ is a Brownian motion in $\cX=\bR^{m}$ so that
$B_{k}$ may be written as $m$ i.i.d. Brownian motions $
(B_{k}^{(1)}, \ldots, B_{k}^{(m)} )$. Then, using \eqref
{eq:nablafR} (recall Remark~\ref{rem:nablafR}) in \eqref{aux:iniPnn1}
and using \eqref{eq:propLipPnn} (recall Lemma~\ref{lem:propagatorLip})
in \eqref{aux:iniPnn2}
%
%
\begin{eqnarray}
&&\bE\biggl\| \frac{1}{ \llvert \gL_{N} \rrvert } \sum_{k}\int
_{0}^{T} \nabla_{\tta}
(P_{t, T}f ) \bigl(\tta_{k}(t), \omega_{k},
x_{k} \bigr) \cdot\mathrm{d}B_{k}(t)\biggr \|^{2}\nonumber\\
&&\qquad=
\frac{1}{
\llvert \gL_{N}
\rrvert ^{2}} \sum_{k} \sum
_{l=1}^{m}\bE\int_{0}^{T}
\partial _{\tta
^{(l)}} (P_{t, T}f )^{2} \,\mathrm{d}t
\nonumber
\\
&&\qquad\leq\frac{m\| \chi_{R} \|_{\infty}^{2}}{ \llvert \gL_{N}
\rrvert
} \int_{0}^{T} \|
P_{t,T}f \|_{R,a}^{2}\,\mathrm{d}t\label{aux:iniPnn1}
\\
&&\qquad\leq\frac{m\| \chi_{R} \|_{\infty}^{2}}{ \llvert \gL_{N}
\rrvert
} 2\int_{0}^{T}
e^{2|\!|\!|P |\!|\!|(T-t)}\,\mathrm{d}t\label{aux:iniPnn2}
\\
&&\qquad=\frac{m (e^{2|\!|\!|P |\!|\!|T}-1)}{(2R)^{2d} \llvert \gL_{N}
\rrvert }\leq\frac{C_{2}}{N^{d}}\label{eq:martpartPnn},
\end{eqnarray}
where $C_{2}=\frac{m (e^{2|\!|\!|P |\!|\!|T}-1)}{8^{d}R^{2d}}$ and
where $|\!|\!|P |\!|\!|$ is defined by \eqref{eq:propLipPnn}.

Let us now give an estimate of the term $D_{N}$ in \eqref{aux:DN}: by
Definition~\ref{def:spacePnearest}, due to the assumptions made on
$\Gamma$, it is easy to see that for fixed $k$ the function
$f_{k}:=\Gamma (\tta_{k}, \omega_{k}, \cdot, \cdot
)\Psi
(x_{k}, \cdot )$ belongs to $\cC_{R, x_{k}}$ with norm $\|
f_{k} \|_{R, x_{k}}= \| \Gamma \|_{\mathrm{Lip}}$. Consequently, by
construction of
the distance $d_{R}$ (recall Definition~\ref{def:distancePnn}), one has
the following:
\[
\forall t>0\qquad \bE\bigl\| \bigl\langle\Gamma (\tta_{k},
\omega_{k}, \cdot, \cdot )\Psi (x_{k}, \cdot ), \nu
_{t}^{(N)} - \nu _{t}\bigr\rangle
\bigr\|^{2}\leq\| \Gamma \|_{\mathrm
{Lip}}^{2}d_{R}
\bigl(\nu^{(N)}_{t}, \nu_{t}\bigr)^{2}.\vadjust{\goodbreak}
\]
Putting together \eqref{eq:normpK}, \eqref{eq:initialcondPnn} and
\eqref
{eq:martpartPnn}, we obtain finally
\[
\bE\bigl\| \bigl\langle f, \nu_{T}^{(N)}- \nu_{T}
\bigr\rangle \bigr\|^{2} \leq3\frac
{C_{1}}{N^{2\wedge d}} + 3\frac{C_{2}}{N^{d}}+ 3
\frac{e^{2|\!|\!|P
|\!|\!|T}-1}{(2R)^{2d}|\!|\!|P |\!|\!|}\| \Gamma \|_{\mathrm
{Lip}}^{2}\int
_{0}^{T} d_{R}\bigl(
\nu^{(N)}_{t}, \nu_{t}\bigr)^{2} \,
\mathrm{d}t.
\]
Taking the supremum over all functions $f$ in $\bigcup_{a\in
[-1,1]^{d}}\cC_{R, a}$ and applying Gronwall's lemma leads to the
result. Theorem~\ref{theo:LLNPnn} is proved.
\end{pf*}

\section{Law of large numbers in the power-law case}
\label{sec:PLinfd}
We suppose in this section that the weight $\Psi$ satisfies
hypothesis (H2).

Let us begin with a technical lemma that will be of constant use
throughout this part:
%
%
\begin{lemma}
\label{lem:convsumRiem}
There exists a constant $C_{0}>0$ (that only depends on $\beta$), such
that for all $N,K\geq1$, for all $a\in\cD_{K}$:
\begin{longlist}[(1)]
\item[(1)] for all $0<\beta<d$, one has
%
%
\begin{equation}
\label{eq:Riemleq1} \sum_{j; j/N\neq a} \biggl\| \frac{j}{2N} -a
\biggr\|^{-\beta} \leq C_{0} %
\cases{ N^{d}K^{d},
&\quad $\mbox{if } a\notin\cD_{N},$\vspace*{2pt}
\cr
N^{d}, &\quad $\mbox{if } a\in\cD_{N}$;} %
\end{equation}
\item[(2)] for $\beta=d$, one has
%
%
\begin{equation}
\label{eq:Riemeg1} \sum_{j; j/N\neq a}\biggl \| \frac{j}{2N} -a
\biggr\|^{-d} \leq C_{0} %
\cases{ K^{d}N^{d}
\ln N,& \quad$\mbox{if } a\notin\cD_{N},$\vspace*{2pt}
\cr
N^{d}\ln N,
& \quad $\mbox{if } a\in\cD_{N}$;} %
\end{equation}
\item[(3)] for all $\beta>d$, one has
%
%
\begin{equation}
\label{eq:Riemgeq1} \sum_{j; j/N\neq a} \biggl\| \frac{j}{2N} -a
\biggr\|^{-\beta} \leq C_{0} %
\cases{ N^{\beta}K^{\beta},
& \quad $\mbox{if } a\notin\cD_{N},$\vspace*{2pt}
\cr
N^{\beta}, & \quad$\mbox{if } a\in\cD_{N}$.} %
\end{equation}
\end{longlist}
\end{lemma}
%
%
\begin{remark}
The estimates given in Lemma~\ref{lem:convsumRiem} in the case $a\in
\cD
_{N}$ are standard and optimal. The main technical problem of
Lemma~\ref
{lem:convsumRiem} lies in the case of $a\notin\cD_{N}$: in this case,
the point $a$ of the discretization $\cD_{K}$ can be arbitrarily close
to one point $ \frac{j}{2N}$ in the above sum. Those points belong to
the discretization $\cD_{N}$. The minimal distance between $a$ and the
discretization $\cD_{N}$ depends on $K$ (actually it depends on the
greatest common divisor of $K$ and $N$; see the proof of Lemma~\ref
{lem:convsumRiem}). This explains the dependence in $K$ of the
estimations of Lemma~\ref{lem:convsumRiem}.

The proof of Lemma~\ref{lem:convsumRiem} is postponed to the \hyperref
[app]{Appendix}.
Lemma~\ref{lem:convsumRiem} will be at the basis of most of the
estimations in this section.
\end{remark}

Theorem~\ref{theo:LLNpowerinfd} is a consequence of the two following
propositions:
%
%
\begin{proposition}
\label{prop:initialcond}
Let fix $\alpha\in[0, d)$, $\gamma$ and $p$ defined in \eqref
{eq:assgamma} and \eqref{eq:p}, respectively. There exists a constant
$C_{1}>0$ (that only depends on $p$ and $C_{0}$ defined in\vadjust{\goodbreak} Lemma~\ref
{lem:convsumRiem}) such that for all $K\geq1$, $N\geq1$, $a\in\cD_{K}$
and $f\in\cC_{a}$ with $\| f \|_{a}\leq1$,
%
%
\begin{eqnarray}
\label{eq:initialcondK} &&\bE\bigl\| \bigl\langle P_{0, T}f, \nu_{0}^{(N)}
\bigr\rangle - \langle P_{0, T}f, \nu _{0}\rangle
\bigr\|^{p}
\nonumber
\\[-8pt]
\\[-8pt]
\nonumber
&&\qquad\leq C_{1} %
\cases{\displaystyle \biggl(
\frac{K^{d}}{N^{\gamma\wedge1}} \biggr)^{p}, &\quad $\mbox{if } \displaystyle\alpha \in \biggl[0,
\frac{d}{2} \biggr),$\vspace*{2pt}
\cr
\displaystyle\biggl(\frac{K^{d}\ln N}{N^{ {d}/{2}\wedge1}}
\biggr)^{p}, &\quad $\mbox{if }\displaystyle \alpha=\frac{d}{2},$\vspace*{2pt}
\cr
\displaystyle\biggl(\frac{K^{ {3d}/{2}}\ln N}{N^{(d-\alpha)\wedge1}}
 \biggr)^{p}, &\quad
 $\mbox{if } \displaystyle\alpha\in \biggl(
\frac{d}{2}, d \biggr)$.} %
\end{eqnarray}
Moreover, in the case where $a\in\cD_{N}$, the previous estimates are
true for $K=1$.
\end{proposition}
%
%
\begin{proposition}
\label{prop:martingale}
Let fix $\alpha\in[0, d)$, $\gamma$ and $p$ defined in \eqref
{eq:assgamma} and \eqref{eq:p}, respectively. There exists a constant
$C_{2}>0$ such that for all $K\geq1$, for all $a\in\cD_{K}$, for all
$f\in\cC_{a}$ such that $\| f \|_{a}\leq1$
%
%
\begin{eqnarray}
&&\bE\biggl\| \frac{1}{ \llvert \gL_{N} \rrvert } \sum_{k}\int
_{0}^{T} \nabla_{\tta}
(P_{t, T}f ) \bigl(\tta_{k}(t), \omega_{k},
x_{k} \bigr) \cdot\mathrm{d}B_{k}(t)\biggr \|^{p}
\nonumber
\\[-8pt]
\\[-8pt]
\nonumber
&&\qquad\leq C_{2} %
\cases{ \displaystyle\biggl(\frac{K^{d}}{N^{d}}
\biggr)^{ {p}/{2}}, &\quad $\mbox{if } \displaystyle\alpha\in \biggl[0, \frac{d}{2} \biggr)$,
\vspace*{2pt}
\cr
\displaystyle\biggl(\frac{K^{d}\ln N}{N^{d}} \biggr)^{ {p}/{2}}, & \quad $\mbox{if }
\displaystyle\alpha=\frac{d}{2}$,\vspace*{2pt}
\cr
\displaystyle\biggl(\frac{K^{d}}{N^{d-\alpha}}
\biggr)^{p}, & \quad $\mbox{if } \displaystyle\alpha \in \biggl(\frac{d}{2}, d
\biggr)$. } %
\end{eqnarray}
Moreover, in the particular case where $a\in\cD_{N}$, the previous
estimates are true for $K=1$.
\end{proposition}
Let us admit for a moment Propositions \ref{prop:initialcond} and \ref
{prop:martingale}. Then the result of Theorem~\ref{theo:LLNpowerinfd}
is a straightforward consequence of the following proposition:

%
\begin{proposition}
\label{prop:distanceKN}
Under the assumptions made above, there exist constants $C_{3}$ and
$C_{4}$ such that for all $K,N\geq1$, one has
%
%
\begin{equation}
\label{eq:distanceKgeqN} \qquad\sup_{0\leq t\leq T} d_{K}\bigl(
\nu^{(N)}_{t}, \nu_{t}\bigr) \leq C_{3}
\cases{ \displaystyle\frac{1}{N^{\gamma\wedge1}} K^{d} e^{C_{4} K^{d}}, &\quad $
\mbox{if } \displaystyle\alpha \in \biggl[0, \frac{d}{2} \biggr)$,\vspace*{2pt}
\cr
\displaystyle\frac{\ln N}{N^{{d}/{2}\wedge1}}
K^{d}e^{C_{4} K^{2d}}, &\quad $\mbox{if } \displaystyle\alpha=\frac{d}{2}$,
\vspace*{2pt}
\cr
\displaystyle\frac{\ln N}{N^{(d-\alpha)\wedge1}} K^{ {3d}/{2}}e^{C_{4}
 K^{{dp}/{q}}}, &\quad $
\mbox{if } \displaystyle\alpha\in \biggl(\frac{d}{2}, d \biggr)$, } %
\end{equation}
where $q$ in \eqref{eq:distanceKgeqN} is the conjugate of $p$ and where
$C_{3}$ and $C_{4}$ are large enough constants that depend only on $p$,
$T$, $\Gamma$, $\Psi$, $c$ and on the constants $C_{1}$ and $C_{2}$
defined in Propositions \ref{prop:initialcond} and \ref{prop:martingale}.
\end{proposition}
\begin{pf}
Let us fix $K\geq1$, $a\in\cD_{K}$ and $f\in\cC_{a}$ with $\| f \|
_{a}\leq1$.
Let us recall the estimate obtained in \eqref{eq:normpK} and \eqref{aux:DN},
%
%
\begin{eqnarray}
\label{eq:normpinfd}&& \bE\bigl\| \bigl\langle f, \nu_{T}^{(N)}-
\nu_{T}\bigr\rangle \bigr\|^{p}\nonumber\\
&&\qquad \leq
3^{p-1}\bE\bigl\| \bigl\langle P_{0, T}f, \nu_{0}^{(N)}-
\nu_{0}\bigr\rangle \bigr\| ^{p} \nonumber\\
&&\qquad\quad{}+ 3^{p-1}\bE\biggl\|
\frac{1}{ \llvert \gL_{N} \rrvert } \sum_{k}\int
_{0}^{T} \nabla_{\tta}
(P_{t, T}f ) \bigl(\tta_{k}(t), \omega_{k},
x_{k} \bigr) \cdot\bigl(\sig\, \mathrm {d}B_{k}(t)\bigr)
\biggr\|^{p}
\\
&&\qquad\quad{}+ 3^{p-1} \biggl(\frac{e^{q|\!|\!|P |\!|\!|T}-1}{q
|\!|\!|P |\!|\!|} \biggr)^{ {p}/{q}} \biggl(
\frac{1}{ \llvert \gL
_{N} \rrvert } \sum_{k} \frac{1}{ \llvert  x_{k} -a \rrvert
^{q\alpha}}
\biggr)^{{p}/{q}}\nonumber\\
&&\quad\qquad{}\times
\int_{0}^{T} \frac{1}{
\llvert \gL_{N} \rrvert } \sum
_{k} \bE\bigl\| \bigl\langle\Gamma (\tta_{k},
\omega_{k}, \cdot, \cdot )\Psi (x_{k}, \cdot ), \nu
_{t}^{(N)} - \nu_{t}\bigr\rangle \bigr\|^{p}
\,\mathrm{d}t.\nonumber
\end{eqnarray}
We understand here the necessity of choosing $p$ (and its conjugate
$q$) different from $2$.
Indeed, the integer $q$ (recall Remark~\ref{rem:p}) has been precisely
chosen such that $q\alpha<d$ which ensures that the term\vspace*{1pt} $ (\frac
{1}{ \llvert \gL_{N} \rrvert } \sum_{k} \frac{1}{ \| x_{k}
-a \|^{q\alpha}}  )^{ {p}/{q}}$ is finite: more precisely, an
application of Lemma~\ref{lem:convsumRiem}, \eqref{eq:Riemleq1} shows
that this quantity is smaller than $K^{{dp}/{q}}$ whenever $a\in
\cD
_{K}$ and smaller than $1$ in the particular case where $a\in\cD_{N}$.

Let us now prove \eqref{eq:distanceKgeqN} in the case where $K>N$.
Notice first that, thanks to the assumptions made on $\Psi$ and
$\Gamma
$ in Section~\ref{subsec:assumptions}, for all $k$ the function
$f_{k}\dvtx(\tta, \omega, x)\mapsto\Gamma (\tta_{k}, \omega
_{k}, \tta
, \omega
)\Psi (x_{k}, x  )$ belongs to the space $\cC_{x_{k}}$ where
$x_{k}\in\cD_{N}$. Indeed [recall the definition of $\cI_{1}(\Psi)$
\eqref{eq:discont}], for all $k$ and $(\tta, \omega, \bar\tta,
\bar\omega, x)$,
\begin{eqnarray*}
&&\| x-x_{k} \|^{\alpha}\Psi (x_{k}, x )\bigl\| \Gamma(
\tta _{k}, \omega _{k}, \tta, \omega) - \Gamma(
\tta_{k}, \omega_{k}, \bar\tta, \bar\omega)\bigr \| \\
&&\qquad\leq
\cI_{1}(\Psi)\| \Gamma \|_{\mathrm{Lip}} \bigl(\| \bar\tta- \tta \| + \|
\bar\omega- \omega \| \bigr)
\end{eqnarray*}
and
\[
\| x-x_{k} \|^{\alpha}\Psi (x_{k}, x )\bigl\| \Gamma(
\tta _{k}, \omega _{k}, \tta, \omega) \bigr\| \leq
\cI_{1}(\Psi)\| \Gamma \| _{\infty}.
\]
As far as condition \eqref{eq:falphaHolder} is concerned, we have
[using \eqref{eq:holderpsi}]
\begin{eqnarray*}
&&\bigl\| \| x-x_{k} \|^{2\gamma} f_{k}(\tta, \omega, x) -
\| y-x_{k} \| ^{2\gamma} f_{k}(\tta, \omega, y)\bigr \|
\\
&&\qquad\leq\| \Gamma \|_{\infty} \bigl\llvert \| x-x_{k}
\|^{2\gamma} \Psi(x_{k}, x) - \| y-x_{k} \|
^{2\gamma} \Psi(x_{k}, y) \bigr\rrvert \\
&&\qquad\leq\cI_{3}(
\Psi)\| \Gamma \|_{\infty} \llvert x-y \rrvert ^{(2\gamma-\alpha)\wedge1}.
\end{eqnarray*}
Therefore, since $K>N$, by definition of the distance
$d^{(p)}_{K}(\cdot
, \cdot)$ (recall Definition~\ref{def:distancePLinfd}), for all $k$,
the following holds:
\[
\bE\bigl\| \bigl\langle\Gamma (\tta_{k}, \omega_{k}, \cdot,
\cdot )\Psi (x_{k}, \cdot ), \nu_{t}^{(N)} - \nu
_{t}\bigr\rangle \bigr\|^{p} \leq\eta _{1}d^{(p)}_{K}
\bigl(\nu_{t}^{(N)}, \nu_{t}\bigr)^{p}
\]
for the constant $\eta_{1}:=\max (\cI_{1}(\Psi)\| \Gamma \|
_{\mathrm{Lip}}, \cI_{1}(\Psi), \cI_{3}(\Psi)\| \Gamma \|_{\infty
} )^{p}$. Using this\break
estimate in \eqref{eq:normpinfd} and taking the supremum over all
functions $f$ in\break  $\bigcup_{a\in\cD_{L}, 1\leq L\leq K}\cC
_{a}$, one obtains
\begin{eqnarray*}
&& d^{(p)}_{K}\bigl(\nu^{(N)}_{T},
\nu_{T}\bigr)^{p}\\
&&\qquad \leq 3^{p-1}\sup
_{f}\bE\bigl\| \bigl\langle P_{0, T}f,
\nu_{0}^{(N)}\bigr\rangle - \langle P_{0, T}f, \nu
_{0}\rangle\bigr \|^{p}
\\
&&\qquad\quad{}+ 3^{p-1}\sup_{f}\bE\biggl\| \frac{1}{ \llvert \gL_{N} \rrvert } \sum
_{k}\int_{0}^{T}
\nabla_{\tta} (P_{t, T}f ) \bigl(\tta_{k}(t),
\omega_{k}, x_{k} \bigr) \cdot\bigl(\sig\,\mathrm
{d}B_{k}(t)\bigr) \biggr\|^{p}
\\
&&\qquad\quad{}+ 3^{p-1} \eta _{2} K^{ {dp}/{q}} \int
_{0}^{T} d^{(p)}_{K}\bigl(
\nu^{(N)}_{t}, \nu _{t}\bigr)^{p} \,
\mathrm{d}t
\end{eqnarray*}
for $\eta_{2}:= \eta_{1}  (\frac{e^{q|\!|\!|P |\!|\!|T}-1}{q
|\!|\!|P |\!|\!|} )^{ {p}/{q}}$. The results of
Propositions \ref{prop:initialcond} and \ref{prop:martingale} together
with an application of Gronwall's lemma leads to the estimate \eqref
{eq:distanceKgeqN} in the case where $K>N$. Note that one can choose in
this case the constants $C_{3}:= 3^{{(p-1)}/{p}}  (2\max
(C_{1}, C_{2} ) )^{{1}/{p}}$ (where $C_{1}$ and $C_{2}$
come from Propositions \ref{prop:initialcond} and~\ref
{prop:martingale}) and $C_{4}:= \frac{3^{p-1}}{p} T\eta_{2}$.

Let us now turn to the case where $K\leq N$. In this situation,
we\break
cannot use Gronwall's inequality in order to obtain an analogous
estimate\break  on $d^{(p)}_{K}(\nu^{(N)}, \nu)$, since the function $f_{k}$
($k\in\gL_{N}$) defined at the beginning\break  of this proof has not the
sufficient regularity ($f_{k}$ belongs to $\cC_{x_{k}}$ where\break
$x_{k}\in
\cD_{N}$ and hence may not belong to
$\bigcup_{a\in\cD_{K'},
1\leq K'\leq K}\cC_{a}$ for $K<N$). Non\-etheless, one can bound the
term $\frac{1}{\eta_{1}}\bE\| \langle\Gamma (\tta_{k},
\omega_{k}, \cdot, \cdot )\Psi (x_{k}, \cdot ),
\nu_{t}^{(N)} - \nu _{t}\rangle \|^{p}$ by\break  $\sup_{f} \bE\| \langle
f, \nu^{(N)}_{t}\rangle - \langle f, \nu _{t}\rangle \|^{p}$,
where the supremum is taken over functions $f$ in $\bigcup_{a\in\cD_{N}}\cC_{a}$ with $\| f \|_{a}\leq1$. Using this estimate in
\eqref{eq:normpinfd} and a calculation similar to the previous one
gives the following estimate:
%
%
\begin{eqnarray}
\label{eq:supfIN}&& \sup_{0\leq t\leq T}
\sup_{f\in\bigcup_{a\in\cD_{N}}\cC_{a}}
\bigl(\bE\bigl\| \bigl\langle f, \nu^{(N)}_{t}\bigr\rangle -
\langle f, \nu _{t}\rangle\bigr \|^{p} \bigr)
\nonumber
\\[-8pt]
\\[-8pt]
\nonumber
&&\qquad\leq
\bigl(C_{3}e^{C_{4}} \bigr)^{p} %
\cases{
\displaystyle\biggl(\frac{1}{N^{\gamma\wedge1}} \biggr)^{p},
&\quad $\mbox{if } \displaystyle\alpha\in \biggl[0,
\frac{d}{2} \biggr)$,\vspace*{2pt}
\cr
\displaystyle\biggl(\frac{\ln N}{N^{ {d}/{2}\wedge1}}
\biggr)^{p}, &\quad $\mbox {if }\displaystyle \alpha=\frac{d}{2}$,\vspace*{2pt}
\cr
\displaystyle\biggl(\frac{1}{N^{(d-\alpha)\wedge1}} \biggr)^{p}, &\quad
$\mbox{if }\displaystyle \alpha \in \biggl(
\frac{d}{2}, d \biggr)$. } %
\end{eqnarray}
But then, for instance in the case $\alpha\in [0, \frac
{d}{2} )$
(we let the two other cases to the reader), for all $K\leq N$, for all
$f\in\bigcup_{a\in\cD_{K'}}\cC_{a}$ for $K'\leq K$,\ inserting directly
\eqref{eq:supfIN} into \eqref{eq:normpinfd} and using again
Propositions \ref{prop:initialcond} and \ref{prop:martingale} leads to
\begin{eqnarray*}
\bE\bigl\| \bigl\langle f, \nu_{t}^{(N)}\bigr\rangle - \langle f
, \nu _{T}\rangle \bigr\|^{p} &\leq&3^{p-1}
C_{1} \biggl( \frac{K^{d}}{N^{\gamma\wedge1}} \biggr)^{p} +
3^{p-1} C_{2} \biggl( \frac{K^{d/2}}{N^{d/2}}
\biggr)^{p}
\\
&&{}+ 3^{p-1} \biggl(\frac
{e^{q|\!|\!|
P |\!|\!|T}-1}{q |\!|\!|P |\!|\!|} \biggr)^{{p}/{q}} T
\bigl(C_{3} e^{C_{4}} \bigr)^{p} \biggl(
\frac{K^{ {d}/{q}}}{N^{\gamma
\wedge1}} \biggr)^{p}.
\end{eqnarray*}
Up to a change in the constant $C_{3}$, this term is anyway smaller
than $  (\frac{C_{3}}{N^{\gamma\wedge1}} K^{d}\times e^{C_{4}
K^{d}}
)^{p}$. Taking the supremum over all $f$ in $\bigcup_{a\in
\cD
_{K'}, K'\leq K}\cC_{a}$, one obtains the result.
\end{pf}
The rest of this part is devoted to the proofs of Propositions \ref
{prop:initialcond} and \ref{prop:martingale}:
\begin{pf*}{Proof of Proposition~\ref{prop:initialcond}}
Recall that the couples $(\tta_{i}(0), \omega_{i})_{1\leq i\leq N}$ are
supposed to be chosen i.i.d. according to the law $\zeta(\mathrm
{d}\tta
)\otimes
\mu(\mathrm{d}\omega)$ on $\cX\times\cE$. Fix $a= \frac{l}{K}\in
\cD
_{K}$, $f\in\cC
_{a}$ with $\| f \|_{a}\leq1$ as well as $\alpha\in(0, d)$ and the
integer $p\geq2$ defined in \eqref{eq:p}. Write again $\vphi:=P_{0,T}f$
for simplicity. Then
\begin{eqnarray*}
\delta_{N}(f)&:=& \bE\bigl\| \bigl\langle P_{0, T}f,
\nu_{0}^{(N)}\bigr\rangle - \langle P_{0, T}f,
\nu_{0}\rangle\bigr \|^{p}
\\
&=&\bE\biggl\| \frac{1}{\llvert \Lambda_{N} \rrvert  }\sum_{j} \vphi(
\tta_{j}, \omega _{j}, x_{j}) - \int\vphi(
\tta, \omega, x) \zeta(\mathrm{d}\tta )\mu(\mathrm{d}\omega)\,\mathrm{d}x
\biggr\|^{p}
\\
&\leq&2^{p-1}\bE\biggl\| \frac{1}{\llvert \Lambda_{N} \rrvert
}\sum_{j}
\vphi(\tta_{j}, \omega_{j}, x_{j}) -
\frac{1}{\llvert \Lambda _{N} \rrvert }\sum_{j}\int\vphi(\tta, \omega,
x_{j}) \zeta(\mathrm{d}\tta )\mu(\mathrm{d}\omega ) \biggr\|^{p}
\\
&&{}+2^{p-1}\biggl\| \frac{1}{\llvert \Lambda_{N} \rrvert }\sum_{j}
\int \vphi(\tta, \omega, x_{j}) \zeta(\mathrm{d}\tta)\mu (\mathrm{d}
\omega)\\
&&\hspace*{82pt}{}- \int\vphi(\tta, \omega, x) \zeta(\mathrm {d}\tta)
\mu(\mathrm{d}\omega)\,
\mathrm{d}x\biggr \|^{p}
\\
&:=& A_{N}+ B_{N}.
\end{eqnarray*}
For simplicity, let us write $X_{j}:= \vphi(\tta_{j}, \omega_{j},
x_{j}) -
\int\vphi(\tta, \omega, x_{j})\zeta(\mathrm{d}\tta)\mu(\mathrm
{d}\omega
)$; note that $\bE
X_{j}=0$ for all $j$. Since the $(\tta_{i}, \omega_{i})$ are i.i.d. random
variables with law $\zeta\otimes\mu$, the first term $A_{N}$ becomes
%
%
\begin{eqnarray}
\label{aux:AN1} A_{N}&=& \frac{1}{\llvert \Lambda_{N} \rrvert ^{p}} \sum
_{l=1}^{\lfloor p/2\rfloor}\sum_{(k_{1}+\cdots+k_{l}=\lfloor
p/2\rfloor
)}
\sum_{j_{1}, \ldots, j_{l}} \bE \bigl(X_{j_{1}}^{2k_{1}}
\cdots X_{j_{l}}^{2k_{l}} \bigr)
\nonumber
\\[-8pt]
\\[-8pt]
\nonumber
&\leq&\frac{2^{2\lfloor p/2\rfloor}}{\llvert \Lambda_{N} \rrvert
^{p}} \sum_{l=1}^{\lfloor p/2\rfloor}
\sum_{(k_{1}+\cdots
+k_{l}=\lfloor
p/2\rfloor)}\sum_{j_{1}, \ldots, j_{l}}
\frac{1}{\| x_{j_{1}}-a \|
^{2\alpha k_{1}}}\cdots\frac{1}{\| x_{j_{l}}-a \|^{2\alpha
k_{l}}},
\end{eqnarray}
where we used $\| f \|_{a}\leq1$ and assumption \eqref{eq:falphabounded}
in \eqref{aux:AN1}. Let us concentrate on the contribution of $l=1$ to
the sum in \eqref{aux:AN1}, that we call $\tA_{N}$ (where $\tp=2
\lfloor p/2\rfloor$)
\[
\tA_{N}= \frac{2^{\tp}}{\llvert \Lambda_{N} \rrvert ^{p}} \sum_{j}
\frac{1}{\| x_{j}-a \|^{2\tp\alpha}}.
\]

Here, one has to distinguish two cases, depending on the value of
$\alpha\in[0,d)$:
\begin{longlist}[(1)]
\item[(1)] If $0\leq\alpha< \frac{d}{2}$, then by definition $p=2$ and
$p\alpha<d$ so that an application of Lemma~\ref{lem:convsumRiem},
\eqref{eq:Riemleq1} leads to
%
%
\begin{equation}
\label{eq:estimANleq12} \tA_{N}\leq\frac{1}{N^{2d}} C_{0}
\cdot K^{d} N^{d}= C_{0} \frac
{K^{d}}{N^{d}}.
\end{equation}
\item[(2)] If $\alpha\geq\frac{d}{2}$, then $p$ is chosen such that
$p> \frac{d}{d-\alpha}$ so that $p\alpha> d$. Then Lemma~\ref
{lem:convsumRiem}, \eqref{eq:Riemgeq1} leads to
%
%
\begin{equation}
\label{eq:estimANeq12} \tA_{N}\leq\frac{1}{N^{pd}} C_{0}
\cdot K^{p\alpha} N^{p\alpha} = C_{0}\frac{K^{p\alpha}}{N^{p(d-\alpha)}}.
\end{equation}
It is also easy to see that the other terms in \eqref{aux:AN1} are
negligible w.r.t. $\tA_{N}$ as $N\to\infty$.
\end{longlist}
Let us now turn to the second term $B_{N}$: $(B_{N})^{{1}/{p}}$ is
the difference between the Riemann sum of the function $\Phi:=x\mapsto
\int\vphi(\tta, \omega, x) \zeta(\mathrm{d}\tta)\mu(\mathrm
{d}\omega)$
and its integral,
so that it should be small with $N$. But one has to be careful since
$\vphi$ as a discontinuity ($\vphi$ belongs to some $\cC_{a}$ for some
$a$) and since we want to have a result uniformly in the function
$\vphi$,
%
%
\begin{eqnarray}
\label{eq:exprBN} \frac{1}{2^{p-1}}B_{N}&=&\biggl\| \frac{1}{ \llvert \gL_{N} \rrvert }\sum
_{j}\Phi(x_{j}) - \int\Phi(x)
\,\mathrm{d}x \biggr\|^{p}
\nonumber
\\[-8pt]
\\[-8pt]
\nonumber
&\leq& \biggl\llvert \sum_{j}
\int_{\Delta_{j}} \bigl\| \Phi(x_{j}) - \Phi(x) \bigr\|
\,\mathrm{d}x \biggr\rrvert ^{p},
\end{eqnarray}
where $\Delta_{j}:= \{z\in [- \frac{1}{2}, \frac{1}{2}
]^{d}; \forall k=1,\dots, d, j_{k}\leq z_{k}< j_{k}+ \frac{1}{2N}\}$ is
the infinitesimal subdomain of $\gL_{N}$ of size $ \frac{1}{2N}$ of
corner $j$.
Let us begin with the following straightforward inequality:
%
%
\begin{eqnarray}
\label{eq:inegtriangPhi}&&\bigl \| \Phi(x) - \Phi(y) \bigr\|\nonumber\\
&&\qquad\leq
\bigl\| \| x-a \|^{-\gamma} - \|
y-a \| ^{-\gamma} \bigr\| \bigl\| \| x-a \|^{\gamma} \Phi(x) + \| y-a
\|^{\gamma} \Phi(y) \bigr\|
\\
&&\qquad\quad{}+ \frac{1}{ \| x-a \|^{\gamma}\| y-a \|^{\gamma}} \bigl\| \Phi(x) \| x-a \|^{2\gamma} - \Phi(y) \| y-a
\|^{2\gamma} \bigr\|.\nonumber
\end{eqnarray}
Using the assumptions made on $f$, we deduce in particular from \eqref
{eq:falphabounded} and $\| f \|_{a}\leq1$ that $\| x-a \|^{\gamma}
\Phi(x)$
is bounded by $\| x-a \|^{\gamma-\alpha}$. Using also \eqref
{eq:falphaHolder}, it is then immediate to see that
%
%
\begin{eqnarray}
\label{eq:estimPhi}
 \bigl\| \Phi(x) - \Phi(y) \bigr\| &\leq&\frac{ \| x-y \|^{\gamma}}{ \| x-a \|
^{\gamma}
\| y-a \|^{\gamma}} \bigl( \| x-a
\|^{\gamma-\alpha}+ \| y-a \| ^{\gamma
-\alpha
} \bigr) \nonumber\\
&&{}+ \frac{\| x-y \|^{(2\gamma-\alpha)\wedge1}}{ \| x-a \|
^{\gamma}\| y-a \|^{\gamma}}
\nonumber
\\[-8pt]
\\[-8pt]
\nonumber
&=& \frac{ \| x-y \|^{\gamma}}{ \| x-a \|^{\alpha} \| y-a \|^{\gamma}} + \frac
{ \| x-y \|^{\gamma}}{ \| x-a \|^{\gamma} \| y-a \|^{\alpha}} \\
&&{}+ \frac
{\| x-y \|^{(2\gamma-\alpha)\wedge1}}{ \| x-a \|^{\gamma}\| y-a \|
^{\gamma}}.\nonumber
\end{eqnarray}
Using \eqref{eq:estimPhi} in \eqref{eq:exprBN}, one obtains that
%
%
\begin{eqnarray}
\label{eq:estimBN} B_{N}&\leq&2^{p-1}  \biggl(\sum
_{j} \int_{\Delta_{j}} \frac{\|
x-x_{j} \|^{\gamma}}{\| x-a \|^{\alpha}\| x_{j}-a \|^{\gamma
}}\,
\mathrm{d}x \nonumber\\
&&\hspace*{26pt}{}+ \sum_{j} \int_{\Delta_{j}}
\frac{\| x-x_{j} \|^{\gamma}}{\| x-a \|
^{\gamma}\| x_{j}-a \|^{\alpha}}\,\mathrm{d}x
\nonumber
\\[-8pt]
\\[-8pt]
\nonumber
&&\hspace*{26pt}{} + \sum_{j} \int_{\Delta_{j}}
\frac{\| x-x_{j} \|^{(2\gamma-\alpha)\wedge1}}{\| x-a \|^{\gamma
}\| x_{j}-a \|^{\gamma
}}\,\mathrm{d}x \biggr)^{p}\\
&:=&2^{p-1}
\bigl(S_{N}^{(1)}+ S_{N}^{(2)} +
S_{N}^{(3)} \bigr)^{p}.
\nonumber
\end{eqnarray}
The first of the three sums in \eqref{eq:estimBN} can be bounded by the
following quantity:
\begin{eqnarray*}
S_{N}^{(1)}&\leq&\sum_{j}
\frac{1}{{\min (\| x_{j-1}-a \|
^{\alpha}, \| x_{j}-a \|^{\alpha} )\| x_{j}-a \|^{\gamma
}}}\int_{\Delta_{j}} \| x-x_{j}
\|^{\gamma}\,\mathrm{d}x
\\
&=& \frac{1}{N^{d+\gamma}}\sum_{j}
\frac{1}{{\min (\|
x_{j-1}-a \|^{\alpha}, \| x_{j}-a \|^{\alpha} )\| x_{j}-a \|
^{\gamma}}}.
\end{eqnarray*}
Let us once again distinguish three cases, depending on the value of
$\alpha$:
\begin{longlist}[(1)]
\item[(1)]
if $\alpha\in [0, \frac{d}{2} )$, then $\alpha+\gamma<d$ [recall
\eqref{eq:assgamma}], so that an application of Lemma~\ref
{lem:convsumRiem}, \eqref{eq:Riemleq1} leads to
%
%
\begin{equation}
\label{eq:SN1leq12} S_{N}^{(1)}\leq C_{0}
\frac{K^{d}}{N^{\gamma}};
\end{equation}
\item[(2)]
if $\alpha= \frac{d}{2}$, then $\alpha+\gamma=d$ [recall \eqref
{eq:assgamma}], so that Lemma~\ref{lem:convsumRiem}, \eqref
{eq:Riemeg1} gives
%
%
\begin{equation}
\label{eq:SN1eq12} S_{N}^{(1)}\leq C_{0}
\frac{K^{d}\ln N}{N^{{d}/{2}}};
\end{equation}
\item[(3)]
if $\alpha\in (\frac{d}{2}, d )$, then $\alpha+\gamma>d$, so
that Lemma~\ref{lem:convsumRiem}, \eqref{eq:Riemgeq1} gives
%
%
\begin{equation}
\label{eq:SN1geq12} S_{N}^{(1)}\leq C_{0}
\frac{K^{\alpha+\gamma}}{N^{d-\alpha}}\leq C_{0}
\frac{K^{ {3d}/{2}}}{N^{d-\alpha}}.
\end{equation}
\end{longlist}
The same calculation leads to the same estimates for the second term
$S_{N}^{(2)}$ in \eqref{eq:estimBN}. A very similar calculation also
leads to the following estimate for the last term $S_{N}^{(3)}$:
%
%
\begin{equation}
\label{eq:SN3} S_{N}^{(3)}\leq C_{0}
\cases{\displaystyle\frac{K^{d}}{N^{(2\gamma-\alpha)\wedge1}}, &
 \quad$\mbox{if } \displaystyle\alpha\in \biggl[0,
\frac{d}{2} \biggr)$,\vspace*{2pt}
\cr
\displaystyle\frac{K^{d}\ln N}{N^{(d-\alpha)\wedge1}}, & \quad$\mbox{if }\displaystyle
 \alpha\in
\biggl[\frac{d}{2}, d \biggr)$. } %
\end{equation}
Combining estimations \eqref{eq:SN3} and \eqref{eq:estimANleq12} [resp.,
\eqref{eq:estimANeq12}] and \eqref{eq:SN1leq12} [resp., \eqref
{eq:SN1eq12} or~\eqref{eq:SN1geq12}] leads to the desired estimation
\eqref{eq:initialcondK}.
The proof of the case where $a\in\cD_{N}$ is analogous and uses the
estimates for $a\in\cD_{N}$ in Lemma~\ref{lem:convsumRiem}.
Proposition~\ref{prop:initialcond} is proved.
\end{pf*}
It remains to prove Proposition~\ref{prop:martingale}, whose purpose is
to control the martingale term in \eqref{eq:estimfnut}:
\begin{pf*}{Proof of Proposition~\ref{prop:martingale}}
Fix some $K\geq1$, $a\in\cD_{K}$ and $f\in\cC_{a}$ such that $\| f
\|_{a}\leq1$.
The martingale $M_{t}^{N}:=\frac{1}{ \llvert \gL_{N} \rrvert }
\sum_{k}\int_{0}^{T} \nabla_{\tta}  (P_{t, T}f )
(\tta
_{k}(t), \omega_{k}, x_{k} ) \cdot\mathrm{d}B_{k}(t)$ may be
written as
$M_{t}^{N}=\frac{1}{ \llvert \gL_{N} \rrvert } \sum_{k}\sum_{l=1}^{m}\int_{0}^{T} \partial_{\tta^{(l)}}  (P_{t, T}f
)
(\tta_{k}(t), \omega_{k},\break  x_{k} ) \,\mathrm{d}B_{k}^{(l)}(t)$, where
for all
$k$, $B_{k}=(B_{k}^{(1)}, \ldots, B_{k}^{(m)})$. Consequently, its
quad\-ratic variation process is given by
\[
\bigl\langle M^{N} \bigr\rangle_{t}= \frac{1}{\llvert \gL_{N} \rrvert ^{2}}
\sum_{k}\sum_{l=1}^{m}
\int_{0}^{T}\bigl \| \partial_{\tta^{(l)}}
P_{t,
T}f \bigl(\tta_{k}(t), \omega_{k},
x_{k} \bigr)\bigr \|^{2}\,\mathrm{d}t.
\]
Applying Remark~\ref{rem:nablafalpha} and Lemma~\ref
{lem:propagatorLip}, we have almost surely that
\[
\bigl\langle M^{N} \bigr\rangle_{t}\leq
\frac{m|\!|\!|P |\!|\!|^{2}}{\llvert
\gL_{N} \rrvert ^{2}}\sum_{k} \frac{1}{ \llvert  x_{k} -a
\rrvert ^{2\alpha}}
\int_{0}^{T} e^{2|\!|\!|P |\!|\!|(T-t)}\,\mathrm{d}t.
\]
An argument repeatedly used in this work shows that one can bound the
quadratic variation by $C \frac{K^{d}}{N^{d}}$ (resp., $C \frac
{K^{d}\ln
N}{N^{d}}$ and $ C \frac{K^{2\alpha}}{N^{2(d-\alpha)}}$) when
$\alpha<
\frac{d}{2}$ (resp., $\alpha= \frac{d}{2}$ and $\alpha> \frac{d}{2}$),
for some constant $C>0$. Then the Burkholder--Davis--Gundy inequality
$\bE (\| M^{N}_{t} \|^{p} )\leq C_{p}\bE (\langle M^{N}
\rangle_{t}^{ {p}/{2}} )$
gives the result. Proposition~\ref{prop:martingale} is proved.
\end{pf*}

\section{The case of a locally Lipschitz dynamics \texorpdfstring{$c(\cdot)$}{c(.)}}
\label{sec:loclip}
One of the key arguments of the proofs of Theorems \ref{theo:LLNPnn}
and \ref{theo:LLNpowerinfd} is the fact that one can derive a
Kolmogorov equation [recall \eqref{eq:backKolm}] for the propagator
$P_{s, t}f$ defined in \eqref{eq:propagator}. Under Assumption~\ref
{ass:Gammac} on the dynamics $c(\cdot)$ (one-sided Lipschitz condition
and absence of \emph{global} Lispchitz continuity), deriving such a
Kolmogorov equation appears to be problematic; see, in particular,
\cite
{MR1311478,1209.6035}. Even if such a result existed, we could not
find a proper reference in the literature.

One can bypass this technical difficulty and prove nevertheless
Theorems~\ref{theo:LLNPnn} and \ref{theo:LLNpowerinfd} by an
approximation argument. We will suppose throughout this section that
$c$ satisfies only Assumption~\ref{ass:Gammac}.

\subsection{Yosida approximation}
Let us denote for all $(\tta, \omega)$, $\tc(\tta, \omega):=
c(\tta, \omega) -
L\tta$, where we recall that $L$ is the constant appearing in the
one-sided Lipschitz continuity assumption \eqref{eq:cgrowthcond}. In
terms of $\tc$, \eqref{eq:cgrowthcond} reads
%
%
\begin{equation}
\label{eq:dissipative} \forall(\tta, \omega), (\bar\tta, \bar\omega)\qquad
 \bigl\langle\tta
-\bar\tta, \tc(\tta, \omega)-\tc(\bar\tta, \bar\omega )\bigr\rangle \leq0,
\end{equation}
and, for example, the mean field evolution \eqref{eq:odemeanfield} reads
%
%
\begin{eqnarray}
\label{eq:odemeanfieldtilde} \mathrm{d}\tta(t)&= \tc\bigl(\tta(t), \omega\bigr)\,\mathrm{d}t +
\tv\bigl(t, \tta(t), \omega, x\bigr)\,\mathrm{d}t+ \sig \cdot\mathrm{d}B(t),
\end{eqnarray}
where $\tv(t, \tta(t), \omega, x):= v(t, \tta(t), \omega, x) +
L \tta(t)$.

For all $\lambda>0$, consider $\tc_{\lambda}$ the Yosida approximation
of $\tc$ (see \cite{MR1840644}, Appendix A, for a review of the basic
properties of Yosida approximations),
%
%
\begin{equation}
\label{eq:yosida} \forall(\tta, \omega)\qquad\tc_{\lambda}(\tta, \omega):= \tc
\bigl(R_{\lambda} (\lambda\tta), \omega\bigr)
\end{equation}
for
%
%
\begin{equation}
\forall(\tta, \omega)\qquad R_{\lambda}(\tta, \omega):= \bigl(\lambda- \tc(
\cdot, \omega) \bigr)^{-1}(\tta).
\end{equation}
Consider now the solution $\tta_{\lambda}$ of the following SDE [with
the same initial condition and driven by the same Brownian motion $B$
as in \eqref{eq:odemeanfieldtilde}]:
%
%
\begin{eqnarray}
\label{eq:odemeanfieldlambda} \mathrm{d}\tta_{\lambda}(t)&= \tc_{\lambda}\bigl(
\tta_{\lambda}(t), \omega \bigr)\,\mathrm{d}t + \tv \bigl(t,
\tta_{\lambda}(t), \omega, x\bigr)\,\mathrm{d}t+ \sig\cdot\mathrm{d}B(t),
\end{eqnarray}
that is, the analog of \eqref{eq:odemeanfieldtilde} where $\tc$ has
been replaced by its Yosida approximation. Note that one can proceed
exactly in the same way for microscopic system \eqref{eq:odegene}.
From now on, whatever $X$ may be, the subscript notation $X_{\lambda}$
will refer to the analog of $X$ when the dynamics has been replaced by
its Yosida approximation. Note that we will, most of the time, drop the
dependencies of the functions in $\omega$, for simplicity of notation.

It is easy to see that $\tc$ and $\tc_{\lambda}$ have the same
regularity in
$\tta$; see, for example, \cite{MR1840644}, page 304. Moreover, $\tc
_{\lambda}$
has the supplementary property to be uniformly Lipschitz continuous. In
other words, $\tc_{\lambda}$ satisfies Assumption~\ref{ass:Gammac}
as well as
Assumption~\ref{ass:cGlobal}, so that everything that has been done
before is applicable: Theorems \ref{theo:LLNPnn} and \ref
{theo:LLNpowerinfd} are true in the case of an interaction ruled by
$\tc_{\lambda}$
%
%
\begin{equation}
\label{eq:convdlambda} \sup_{t\in[0, T]}d \bigl(\nu_{t, \lambda}^{(N)},
\nu_{t, \lambda
} \bigr) \leq C N^{-\beta}
\end{equation}
for $d$ either equal to $d_{R}(\cdot, \cdot)$ or $d^{(p)}_{\infty
}(\cdot
, \cdot)$ and $\beta$ one of the appropriate exponent appearing in the
formulation of Theorems \ref{theo:LLNPnn} and \ref{theo:LLNpowerinfd}.
Note that the constant $C$ in \eqref{eq:convdlambda} \emph{does not
depend on $\lambda$}. Indeed, the assumption made in Section~\ref
{sec:propagator} about the global Lipschitz continuity of $c$ was made
only to ensure the existence of the Kolmogorov equation. In particular,
the modulus of continuity of $c$ did not enter into the calculation
made in Section~\ref{sec:propagator}: the only dependence in the
dynamics $c$ was in its one-sided Lipschitz constant $L$ (recall
Lemma~\ref{lem:propagatorLip}), which is conserved by the Yosida
approximation. In other words, every constant estimates made upon
evolution \eqref{eq:odemeanfieldlambda} is independent on $\lambda$.

Now, Theorems \ref{theo:LLNPnn} and \ref{theo:LLNpowerinfd} in our
general framework are an easy consequence of the triangular inequality
and the following proposition:
%
%
\begin{proposition}
\label{prop:distancelambda}
For all $N\geq1$,
%
%
\begin{eqnarray}
\label{eq:distnulambN}\sup_{t\in[0, T]}d \bigl(\nu_{t, \lambda}^{(N)},
\nu_{t}^{(N)} \bigr) &\mathop{\longrightarrow}\limits^{\lambda\to\infty}& 0,
\\
\label{eq:distnulamb}\sup_{t\in[0, T]}d (\nu_{t, \lambda}, \nu
_{t} ) &\mathop{\longrightarrow}\limits^{\lambda\to\infty}& 0.
\end{eqnarray}
\end{proposition}
The rest of this section is devoted to prove Proposition~\ref
{prop:distancelambda}. Let us begin with some a priori estimates:
%
%
\begin{lemma}
\label{lem:apriori}
We have the following a priori estimates:
%
%
\begin{equation}
\label{eq:ttalmoment} \sup_{\lambda>0} \bE \Bigl(\sup
_{t\in[0, T]} \bigl\| \tta_{\lambda} (t) \bigr\|^{2} \Bigr)<
\infty
\end{equation}
and
%
%
\begin{equation}
\label{eq:ttalinteg} \bP \biggl(\sup_{\lambda>0} \int
_{0}^{T} \bigl\| \tc_{\lambda}\bigl(\tta
_{\lambda}(s)\bigr)\bigr \|^{2} \,\mathrm{d} s<\infty \biggr)=1.
\end{equation}
\end{lemma}
\begin{pf}
Let us first prove the first estimate \eqref{eq:ttalmoment}: applying
It\^o's formula,
\begin{eqnarray*}
\bigl\| \tta_{\lambda}(t) \bigr\|^{2}&=&\bigl\| \tta_{\lambda}(0)
\bigr\|^{2} + 2\int_{0}^{t} \bigl\langle
\tta_{\lambda} (s), \tc_{\lambda} \bigl(\tta_{\lambda
}(s)\bigr)
+ \tv\bigl(s, \tta_{\lambda}(s), \omega, x\bigr)\bigr\rangle \,\mathrm{d}s
\\
&&{}+ 2\int_{0}^{t}\bigl\langle
\tta_{\lambda}(s), \mathrm{d}B(s)\bigr\rangle + \operatorname{tr}\bigl(\sig
\sig^{T}\bigr)t
\\
&\leq&\bigl\| \tta_{\lambda}(0) \bigr\|^{2} + 2 \bigl(\bigl\| \tc(0)\bigr \|
 + L+ \|
\Gamma \|_{\infty}S(\Psi ) \bigr)\int_{0}^{t}
\bigl\| \tta_{\lambda}(s) \bigr\|^{2} \,\mathrm{d}s
\\
&&{}+ 2\int_{0}^{t}\bigl\langle
\tta_{\lambda} (s), \,\mathrm{d}B(s)\bigr\rangle +\operatorname{tr}\bigl(\sig
\sig^{T}\bigr)T.
\end{eqnarray*}
Taking expectations and using the Burkholder--Davis--Gundy inequality,
we obtain that for some constant $C>0$ (independent of $\lambda$),
\begin{eqnarray*}
\bE \Bigl(\sup_{s\leq t}\bigl\| \tta_{\lambda}(s)
\bigr\|^{2} \Bigr)&\leq& \bE \bigl(\bigl\| \tta (0) \bigr\|^{2} \bigr) + \operatorname{tr}
\bigl(\sig\sig^{T}\bigr)T +2C \int_{0}^{t}
\bE \Bigl(\sup_{u\leq s}\bigl\| \tta_{\lambda}(u)
\bigr\|^{2} \Bigr)\,\mathrm{d}s
\\
&&{} + 6 \operatorname{tr}\bigl(\sig\sig^{T}\bigr)^{1/2} \bE \biggl( \biggl(
\int_{0}^{t} \bigl\| \tta_{\lambda}(u)
\bigr\|^{2}\,\mathrm {d}u \biggr)^{
{1}/{2}} \biggr)
\\
&\leq&\bE \bigl(\bigl\| \tta(0) \bigr\|^{2} \bigr) + \operatorname{tr}\bigl(\sig
\sig^{T}\bigr)T +2C \int_{0}^{t} \bE
\Bigl(\sup_{u\leq s}\bigl\| \tta_{\lambda}(u) \bigr\|^{2}
\Bigr)\,\mathrm{d}s
\\
&&{} + 18 \operatorname{tr}\bigl(\sig\sig^{T}\bigr)T + \frac{1}{2}\bE \Bigl( \sup
_{u\leq t}\bigl\| \tta _{\lambda} (u)\bigr \|^{2} \Bigr),
\end{eqnarray*}
which implies
\[
\bE \Bigl(\sup_{s\leq t}\bigl\| \tta_{\lambda}(s)
\bigr\|^{2} \Bigr)\leq2 \bigl(\bE \bigl(\bigl\| \tta(0)\bigr \|^{2} \bigr)
+ 19 \operatorname{tr}\bigl(\sig\sig^{T}\bigr)T \bigr) +4C \int_{0}^{t}
\bE \Bigl(\sup_{u\leq s}\bigl\| \tta_{\lambda}(u)
\bigr\|^{2} \Bigr)\,\mathrm{d}s,
\]
and Gronwall's lemma leads to the result.

Let us now turn to the second estimate \eqref{eq:ttalinteg}: define
$Y_{\lambda}
(t):=\tta_{\lambda}(t) - \sig\cdot B(t)$. Then $Y_{\lambda}$ satisfies
%
%
\begin{equation}
\mathrm{d}Y_{\lambda}(t)= \bigl(\tc_{\lambda}\bigl(Y_{\lambda
}(t)+B(t),
\omega\bigr) + \tv\bigl(t, Y_{\lambda} (t)+B(t), \omega, x\bigr) \bigr)\,
\mathrm{d}t.
\end{equation}
Clearly,
\begin{eqnarray*}
&&\bigl\| Y_{\lambda}(t) \bigr\|^{2}\\
&&\quad= \bigl\| Y_{\lambda}(0)
\bigr\|^{2} + 2\int_{0}^{t} \bigl\langle
Y_{\lambda} (s), \tc_{\lambda}\bigl(Y_{\lambda} (s)+\sig\cdot
B(s)\bigr)\bigr\rangle \,\mathrm{d}s
\\
&&\qquad{}+ 2\int_{0}^{t} \bigl\langle
Y_{\lambda}(s), \tv\bigl(s, Y_{\lambda} (s)+\sig\cdot B(s)\bigr),
\omega, x\bigr\rangle \,\mathrm {d}s
\\
&&\quad\leq\bigl\| Y_{\lambda}(0) \bigr\|^{2} + 2 \bigl(\bigl\| \tc(0) \bigr\| + L+
\|
\Gamma \|_{\infty}S(\Psi ) \bigr)\int_{0}^{t}
\bigl\| Y_{\lambda}(s) \bigr\|^{2}\,\mathrm{d}s
\\
&&\qquad{}+ 2\int_{0}^{t} \bigl\langle Y_{\lambda}
(s), \tc_{\lambda}\bigl(\sig\cdot B(s)\bigr)\bigr\rangle \,\mathrm{d}s
\\
&&\quad\leq\bigl\| Y_{\lambda}(0) \bigr\|^{2} \\
&&\qquad{}+ 2 \biggl(\bigl\| \tc(0) \bigr\| + L+ \|
\Gamma \|_{\infty}S(\Psi) + \int_{0}^{t}
\bigl\| \tc_{\lambda}\bigl(\sig\cdot B(s)\bigr) \bigr\|^{2}\,\mathrm {d}s
\biggr)\int_{0}^{t} \bigl\| Y_{\lambda} (s)
\bigr\|^{2}\,\mathrm{d}s,
\end{eqnarray*}
taking the supremum in $\lambda$ and using $Y_{\lambda}(0)=\tta
_{\lambda}(0)=\tta
(0)$, we have
\begin{eqnarray*}
\sup_{\lambda}\bigl\| Y_{\lambda}(t) \bigr\|^{2}&\leq&\bigl\|
\tta(0) \bigr\|^{2} + 2 \biggl(C + \int_{0}^{t}
\bigl\| \tc_{\lambda}\bigl(\sig\cdot B(s)\bigr) \bigr\|^{2}\,\mathrm{d}s
\biggr)\int_{0}^{t} \sup_{\lambda
}
\bigl\| Y_{\lambda}(s) \bigr\|^{2}\,\mathrm{d}s
\\
&\leq&\bigl\| \tta(0) \bigr\|^{2} + 2 \biggl(C + \int_{0}^{t}
\bigl\| \tc\bigl(\sig \cdot B(s)\bigr) \bigr\|^{2}\,\mathrm{d}s \biggr)\int
_{0}^{t} \sup_{\lambda}\bigl\|
Y_{\lambda} (s) \bigr\|^{2}\,\mathrm{d}s,
\end{eqnarray*}
where we used the pointwise estimate $\| \tc_{\lambda}(\tta) \|\leq
\| \tc (\tta) \|$.
Gronwall's lemma gives
\[
\sup_{\lambda}\bigl\| Y_{\lambda}(t) \bigr\|^{2} \leq\bigl\|
\tta(0) \bigr\|^{2} \exp \biggl(2 \biggl(C + \int_{0}^{T}
\bigl\| \tc\bigl(\sig\cdot B(s)\bigr) \bigr\|^{2}\,\mathrm{d}s \biggr)T \biggr)
\]
that is almost surely finite, since $\tc$ is locally bounded, and the
trajectories of $B$ are almost surely bounded. Consequently,
\[
\sup_{\lambda}\sup_{t\leq T} \bigl\|
\tta_{\lambda}(t)\bigr \|^{2}\leq\sup_{\lambda
}\sup
_{t\leq T}\bigl\| Y_{\lambda}(t) \bigr\|^{2} + \sup
_{t\leq T}\bigl\| B(t) \bigr\| ^{2}<\infty\qquad \mbox{a.s.}
\]
Since $\tc$ is polynomially bounded, this implies now that
\[
\sup_{\lambda} \int_{0}^{T} \bigl\|
\tc_{\lambda}\bigl(\tta_{\lambda}(t)\bigr) \bigr\|^{2}\,
\mathrm{d}t<\infty \qquad\mbox{a.s.},
\]
which is the result.
\end{pf}
The key estimate of this section is the following:
%
%
\begin{proposition}
\label{prop:thetalambda}
Almost surely, the following holds:
%
%
\begin{equation}
\label{eq:supthetalambda} \limsup_{\lambda\to\infty} \sup_{t\in[0, T]}
\bigl\| \tta(t)- \tta _{\lambda} (t) \bigr\| =0.
\end{equation}
\end{proposition}
\begin{pf}
Let us fix $\lambda<\mu$. Since the Brownian motion is the same, one
has successively [for a constant $C= L + \| \Gamma \|_{\mathrm
{Lip}}S(\Psi)$]
\begin{eqnarray*}
&&\frac{\mathrm{d}}{\mathrm{d}t} e^{-2Ct} \bigl\| \tta_{\mu}(t)-\tta
_{\lambda}(t) \bigr\|^{2}\\
&&\qquad= -2C e^{-2Ct} \bigl\|
\tta_{\mu}(t)-\tta_{\lambda}(t) \bigr\|^{2}
\\
&&\qquad\quad{}+ 2e^{-2Ct} \bigl\langle\tta_{\mu}(t)-\tta_{\lambda}
(t), \tc_{\mu} \bigl(\tta_{\mu
}(t)\bigr)-
\tc_{\lambda}\bigl(\tta_{\lambda}(t)\bigr)\bigr\rangle
\\
&&\qquad\quad{}+ 2e^{-2Ct} \bigl\langle\tta_{\mu}(t)-\tta_{\lambda}(t)
, \tv\bigl(t, \tta _{\mu}(t), \omega, x\bigr)- \tv\bigl(t,
\tta_{\lambda}(t), \omega, x\bigr)\bigr\rangle
\\
&&\qquad\leq-2C e^{-2Ct}\bigl \| \tta_{\mu}(t)-\tta_{\lambda}(t)
\bigr\|^{2}
\\
&&\qquad\quad{}+ 2e^{-2Ct} \bigl\langle\tta_{\mu} (t)-\tta_{\lambda}(t)
, \tc_{\mu
}\bigl(\tta_{\mu}(t)\bigr)- \tc_{\lambda}
\bigl(\tta_{\lambda}(t)\bigr)\bigr\rangle
\\
&&\quad\qquad{}+ 2e^{-2Ct} \bigl(L + \| \Gamma \|_{\mathrm{Lip}}S(\Psi) \bigr) \bigl\|
\tta_{\mu} (t)-\tta_{\lambda} (t) \bigr\|^{2}
\\
&&\qquad\leq2e^{-2Ct} \bigl\langle\tta_{\mu}(t)-
\tta_{\lambda}(t), \tc_{\mu
}\bigl(\tta_{\mu}(t)\bigr)-
\tc_{\lambda} \bigl(\tta_{\lambda} (t)\bigr)\bigr\rangle
\\
&&\qquad= 2e^{-2Ct} \biggl\langle \biggl(R_{\mu}\bigl(\mu
\tta_{\mu}(t)\bigr) - \frac
{1}{\mu} \tc\bigl(R_{\mu}
\bigl(\mu\tta_{\mu}(t)\bigr)\bigr) \biggr)
\\
&&\hspace*{36pt}\qquad\quad{}- \biggl(R_{\lambda}\bigl(\lambda\tta _{\lambda}(t)\bigr) -
\frac
{1}{\lambda} \tc \bigl(R_{\lambda}\bigl(\lambda\tta_{\lambda}(t)
\bigr)\bigr) \biggr),\\
&&\hspace*{90pt}\tc\bigl(R_{\mu}\bigl(\mu \tta_{\mu}(t)
\bigr)\bigr)- \tc\bigl(R_{\lambda} \bigl(\lambda\tta_{\lambda} (t)
\bigr)\bigr) \biggr\rangle
\\
&&\qquad\leq-2e^{-2Ct} \biggl\langle\frac{1}{\mu} \tc_{\mu}
\bigl(\tta_{\mu
}(t)\bigr)-\frac {1}{\lambda} \tc_{\lambda} \bigl(
\tta_{\lambda}(t)\bigr), \tc _{\mu}\bigl(\tta_{\mu}(t)
\bigr)- \tc_{\lambda}\bigl(\tta_{\lambda}(t)\bigr)\biggr\rangle.
\end{eqnarray*}
Integrating this inequality gives (since the initial condition is the same)
\begin{eqnarray*}
&&\frac{1}{2} e^{-2CT} \bigl\| (\tta_{\mu}-
\tta_{\lambda}) (T) \bigr\|^{2}\\
&&\qquad\leq - \int_{0}^{T}e^{-2Ct}
\biggl\langle\frac{1}{\mu} \tc_{\mu}\bigl(\tta_{\mu
}(t)
\bigr)-\frac {1}{\lambda} \tc_{\lambda} \bigl(\tta_{\lambda}(t)\bigr)
, \tc _{\mu}\bigl(\tta_{\mu}(t)\bigr)- \tc_{\lambda}
\bigl(\tta_{\lambda}(t)\bigr)\biggr\rangle \,\mathrm{d}t.
\end{eqnarray*}
This gives in particular that
\[
\int_{0}^{T}e^{-2Ct} \biggl\langle
\frac{1}{\mu} \tc_{\mu}\bigl(\tta_{\mu
}(t)\bigr)-
\frac {1}{\lambda} \tc_{\lambda}\bigl(\tta_{\lambda}(t)\bigr), \tc
_{\mu}\bigl(\tta_{\mu}(t)\bigr)- \tc_{\lambda}\bigl(
\tta_{\lambda}(t)\bigr)\biggr\rangle \,\mathrm{d}t\leq0.
\]
Let us denote as $\| \cdot \|_{H}$ the Hilbert norm in $H:= L^{2}([0, T],
e^{-2Cs}\,\mathrm{d}s; \cX)$. Then, from the identity
\begin{eqnarray*}
&&2 \biggl\langle\tc_{\mu}(\tta_{\mu})- \tc_{\lambda}(
\tta_{\lambda}), \frac{1}{\mu} \tc_{\mu}(
\tta_{\mu} )-\frac{1}{\lambda } \tc _{\lambda}(
\tta_{\lambda})\biggr\rangle_{H}\\
&&\qquad= \biggl(\frac{1}{\mu} +
\frac{1}{\lambda
} \biggr)\bigl \| \tc_{\mu} (\tta_{\mu})-
\tc_{\lambda}(\tta_{\lambda
}) \bigr\|_{H}^{2}
\\
&&\qquad\quad{}+ \biggl(\frac{1}{\mu} - \frac
{1}{\lambda
} \biggr) \bigl(\bigl\|
\tc_{\mu}(\tta_{\mu}) \bigr\|_{H}^{2} - \bigl\|
\tc _{\lambda}(\tta_{\lambda} ) \bigr\|_{H}^{2}
\bigr),
\end{eqnarray*}
one obtains that
%
%
\begin{equation}
\label{est:7} \qquad\quad\biggl(\frac{1}{\mu} + \frac{1}{\lambda} \biggr) \bigl\|
\tc_{\mu
}(\tta_{\mu})- \tc_{\lambda}(
\tta_{\lambda} ) \bigr\|_{H}^{2}\leq \biggl(
\frac{1}{\lambda} - \frac{1}{\mu} \biggr) \bigl(\bigl\| \tc_{\mu} (
\tta_{\mu}) \bigr\|_{H}^{2} - \bigl\| \tc_{\lambda
}(
\tta_{\lambda}) \bigr\|_{H}^{2} \bigr),
\end{equation}
which gives in particular that $\lambda\mapsto\| \tc_{\lambda}(\tta
_{\lambda} ) \|_{H}^{2}$ is
increasing and by \eqref{eq:ttalinteg} is bounded and thus convergent.
The same inequality \eqref{est:7} shows also that $\| \tc_{\mu}(\tta
_{\mu})- \tc_{\lambda} (\tta_{\lambda}) \|_{H}^{2} \to_{\lambda,
\mu\to\infty} 0$, so that $(\tc_{\lambda}
(\tta_{\lambda})(t))$
converges in $H$ to some~$c_{\infty}(t)$.

Going back to the first inequality of the proof, one has
\begin{eqnarray*}
&&\frac{1}{2} \sup_{t\in[0, T]}e^{-2Ct} \bigl\|
\tta_{\mu}(t)-\tta _{\lambda} (t)\bigr \|^{2}\\
&&\qquad\leq \int
_{0}^{T} e^{-2Ct} \bigl\langle
\tta_{\mu}(t)-\tta_{\lambda}(t), \tc_{\mu}\bigl(
\tta_{\mu}(t)\bigr)- \tc_{\lambda}\bigl(\tta_{\lambda} (t)
\bigr)\bigr\rangle\,\mathrm{d}t
\\
&&\qquad\leq\frac{1}{4T} \int_{0}^{T}
e^{-2Ct} \bigl\| \tta_{\mu}(t)-\tta _{\lambda} (t)
\bigr\|^{2}\,\mathrm{d} t
\\
&&\qquad\quad{}+ T \int_{0}^{T} e^{-2Ct} \bigl\|
\tc_{\mu}\bigl(\tta_{\mu}(t)\bigr)- \tc _{\lambda}\bigl(
\tta_{\lambda} (t)\bigr) \bigr\|^{2}\,\mathrm{d}t
\\
&&\qquad\leq\frac{1}{4} \sup_{t\in[0, T]} e^{-2Ct} \bigl\|
\tta_{\mu}(t)-\tta _{\lambda} (t) \bigr\|^{2}
\\
&&\qquad\quad{}+ T \int_{0}^{T} e^{-2Ct} \bigl\|
\tc_{\mu
}\bigl(\tta_{\mu}(t)\bigr)- \tc_{\lambda} \bigl(
\tta_{\lambda} (t)\bigr) \bigr\| ^{2}\,\mathrm{d}t.
\end{eqnarray*}
Hence
\[
\sup_{t\in[0, T]}e^{-2Ct} \bigl\| \tta_{\mu}(t)-
\tta_{\lambda}(t) \bigr\| ^{2} \leq4T \int_{0}^{T}
e^{-2Ct} \bigl\| \tc_{\mu}\bigl(\tta_{\mu}(t)\bigr)-
\tc_{\lambda}\bigl(\tta _{\lambda}(t)\bigr)\bigr \|^{2}\,
\mathrm{d}t,
\]
which goes to $0$ as $\lambda, \mu\to\infty$. This implies that there
exists an adapted process $\bar\tta$ with continuous trajectories such
that $\lim_{\lambda\to\infty} \tta_{\lambda}= \bar\tta$,
uniformly and almost
surely. Clearly, for all $t$, the strong continuity $\lim_{\lambda\to
\infty
} R_{\lambda}(\lambda\bar\tta(t))= \bar\tta(t)$ of the resolvent
and the
uniform Lipschitz continuity $\| R_{\lambda}(\lambda\tta_{\lambda
}(t)) - R_{\lambda} (\lambda\bar\tta (t)) \|\leq\| \tta_{\lambda
}(t)-\tta(t) \|$ implies that $\lim_{\lambda}R_{\lambda}
(\lambda\tta_{\lambda}
(t))=\bar\tta(t)$. Finally, continuity of $\tc$ gives $\lim_{\lambda\to
\infty}\tc_{\lambda}(\tta_{\lambda}(t))= \tc(R_{\lambda}(\lambda
\tta_{\lambda}(t)))=\tc(\bar\tta(t))$.
Consequently, we have that, almost surely $\tc(\bar\tta
_{t})=c_{\infty
}(t)$, so that $\bar\tta$ solves equation \eqref{eq:odemeanfieldtilde},
so that by uniqueness $\bar\tta=\tta$ almost surely.
\end{pf}

We are now in position to prove Proposition~\ref{prop:distancelambda}:
\begin{pf*}{Proof of Proposition~\ref{prop:distancelambda}}
We only prove \eqref{eq:distnulamb}, the proof of \eqref
{eq:distnulambN} follows from analogous estimates with the microscopic
equation \eqref{eq:odegene}.
We only treat the (more complicated) case of the power-law interaction.
Fix any $f$ in $\cC_{a}$ for some $a$ with $\| f \|_{a}\leq1$. Then, by
Lispchitz continuity of $f$ in the variable $\tta$
\begin{eqnarray*}
\bigl\llvert \langle f, \nu_{t, \lambda}\rangle - \langle f, \nu
_{t}\rangle\bigr\rrvert &\leq S(\Psi) \bE_{B}\bigl \|
\tta_{\lambda}(t) - \tta(t) \bigr\|.
\end{eqnarray*}
Taking the supremum in $f$ and in $t$ leads to
\[
\sup_{t\in[0, T]} d (\nu_{t, \lambda}, \nu_{t} )\leq
S(\Psi) \bE _{B} \sup_{t\in[0, T]}\bigl\|
\tta_{\lambda}(t) - \tta(t) \bigr\|.
\]
By \eqref{eq:supthetalambda} we have the almost sure convergence to $0$
of $\sup_{t\in[0, T]}\| \tta_{\lambda}(t) - \tta(t) \|$ and \eqref
{eq:ttalmoment}
gives the boundedness in $L^{2}$ implying uniform integrability. The
result follows.
\end{pf*}

\begin{appendix}\label{app}
\section*{Appendix: Proof of a technical lemma}

\begin{pf*}{Proof of Lemma~\ref{lem:convsumRiem}}
Let us proceed by induction on the dimension $d$. Let us fix $d=1$:
\begin{itemize}
\item
Let us begin with the case where $a\notin\cD_{N}$: let $J$ be the
integer such that $ \frac{J}{2N}< a< \frac{J+1}{2N}$. Then an easy
comparison with integrals shows the following:
\begin{eqnarray*}
&&\sum_{j} \biggl\llvert \frac{j}{2N} -a
\biggr\rrvert ^{-\beta} \\[-2pt ]
&&\qquad\leq 2^{\beta
}N^{\beta} \biggl( \int
_{0}^{J} \llvert 2aN - t \rrvert ^{-\beta}
\,\mathrm{d}t + \llvert 2aN- J \rrvert ^{-\beta}\\[-2pt ]
&&\hspace*{30pt}\qquad\quad{} + \bigl\llvert 2aN- (J+1)
\bigr\rrvert ^{-\beta}+ \int_{J+1}^{N} \llvert t- 2aN \rrvert
^{-\beta} \,\mathrm{d}t \biggr)
\\[-2pt ]
&&\qquad=2^{\beta}N^{\beta} \int_{0}^{J}
\llvert 2aN - t \rrvert ^{-\beta
} \,\mathrm{d}t + 2^{\beta}N^{\beta}
\int_{J+1}^{N} \llvert t- aN \rrvert
^{-\beta} \,\mathrm{d}t
\\[-2pt ]
&&\qquad\quad{}+ \biggl\llvert a- \frac{J}{2N} \biggr\rrvert ^{-\beta} + \biggl
\llvert a- \frac{J+1}{2N} \biggr\rrvert ^{-\beta}.
\end{eqnarray*}
It is straightforward to see that the two first integral terms are
smaller than $ \frac{N}{d-\beta}$ whereas each of the two remaining
terms is smaller than $\rho(N, K)^{-\beta}$, where $\rho(N, K):=
\inf_{\vert j\vert\leq N , \vert l\vert\leq
K, j/N\neq l/K} \llvert \frac{j}{2N} - \frac{l}{2K}\rrvert = \frac{\gcd(K,
N)}{2KN}\geq\frac{1}{2KN}$. Consequently, since $K\geq1$ and $\beta<1$,
\[
\sum_{j} \biggl\llvert \frac{j}{N} -a
\biggr\rrvert ^{-\beta} \leq \frac
{2N}{d-\beta} + 2 K^{\beta}
N^{\beta}\leq C_{0}NK.
\]
\item
The case where $a\in\cD_{N}$ is easier: in this case, $a= \frac{k}{2N}$
for some $k$. Then, once again by comparison with integrals,
\begin{eqnarray*}
\sum_{j;j/N\neq a} \biggl\llvert \frac{j}{2N} -a
\biggr\rrvert ^{-\beta}&=& 2^{\beta}N^{\beta}\sum
_{j\neq k} \llvert j -k \rrvert ^{-\beta
}\\[-2pt ]
&\leq&
\frac{N^{\beta}}{1-\beta} \bigl((N+k)^{1-\beta} + (N-k)^{1-\beta
} \bigr)\\[-2pt ]
&\leq&
\frac{2^{2-\beta}N}{1-\beta}.
\end{eqnarray*}
\end{itemize}
The other cases ($\beta=1$ and $\beta>1$) are similar and left to the
reader. Lemma~\ref{lem:convsumRiem} is proved in the particular case
of $d=1$.

The case of higher dimension is nothing but a technical complication of
the previous case $d=1$. Let us fix $d>1$, $a=(a_{1}, \ldots,
a_{d})\in
\cD_{K}$ and denote by $j=(j_{1}, \ldots, j_{d})$ any element of $\bZ^{d}$.

Let us begin with the case where $a\notin\cD_{K}$. Let $(J_{1},
\ldots,
J_{d})$ the $d$ integers between $-N$ and $N$ such that for all $l=1,
\ldots, d$, $J_{l}\leq2a_{l}N \leq J_{l}+1$, with at least one
inequality that is strict. The coordinates $J_{l}$ and $J_{l}+1$ are by
construction the closest integers to $2a_{l}N$ in $-N, \ldots, N$. For
the rest of this proof, we will refer to them as \emph{critical}
coordinates. Then one can decompose the sum $\sum_{j} \| \frac{j}{2N}
-a \|^{-\beta}$ according to the number $p$ of critical coordinates among
$(j_{1}, \ldots, j_{d})=j$, where $j$ is a typical index,
%
%
\setcounter{equation}{0}
\begin{equation}
\label{eq:sumajp} \sum_{j} \biggl\| \frac{j}{2N} -a
\biggr\|^{-\beta}= \sum_{p=0}^{d} \sum
_{(i_{1},
\ldots, i_{p})}\sum_{j\in\cJ_{(i_{1}, \ldots, i_{p})}}\biggl\|
\frac
{j}{2N} -a \biggr\|^{-\beta},
\end{equation}
where the second sum is taken over all the vectors $(i_{1}, \ldots,
i_{p})$ with strictly increasing indices taken among $1, \ldots, d$ and
where $\cJ(i_{1}, \ldots, i_{p})$ is a notation for the set of vectors
$j=(j_{1}, \ldots, j_{d})$ such that $j_{i_{l}}$ is critical for every
$l=1, \ldots, p$.

In the sum \eqref{eq:sumajp}, let us treat the cases $p=0$ and $p>0$
separately. Let us first focus on the case $p=0$: it corresponds to
vectors $j$ without critical coordinates, which means that we restrict
ourselves to $j$ such that for every $k=1, \ldots, d$, either
$j_{k}<J_{k}$ (in such case $\llvert  j_{k} - 2a_{k}N \rrvert =
2a_{k}N - j_{k}$) or either $j_{k}>J_{k}+1$ (in such case $\llvert
j_{k} - 2a_{k}N \rrvert = j_{k}- 2a_{k}N$). In particular, this sum
can be divided into $2^{d}$ sums $\sum_{j\in D}\| \frac{j}{2N} -a \|
^{-\beta}$ where $D$ is a connected subdomain of $[-1/2, 1/2]^{d}$,
which is defined by this binary choice for each $j_{k}$. For
simplicity, we only treat the case of $D_{0}:= \{j=(j_{1}, \ldots,
j_{d}); \forall k=1, \ldots, d, j_{k}<J_{k}\}$. The case of the other
$2^{d}-1$ subdomains can be treated in a similar way.

We have successively,
%
%
\begin{eqnarray}
\sum_{j\in D_{0}}\biggl\| \frac{j}{2N} -a
\biggr\|^{-\beta}&=& 2^{\beta}N^{\beta
}\mathop{\sum
_{j_{k}<J_{k}-1}}\Biggl\llvert \sum_{l=1}^{d}
(2a_{l}N- j_{l})^{2} \Biggr\rrvert
^{-\beta/2}\hspace*{-35pt}
\\
&\leq&2^{\beta}N^{\beta}\int_{-N}^{J_{1}}
\cdots\int_{-N}^{J_{d}} \Biggl\llvert \sum
_{l=1}^{d} (2a_{l}N- t_{l})^{2}
\Biggr\rrvert ^{-\beta
/2}\,\mathrm{d} t_{1}\cdots
\mathrm{d}t_{d}\hspace*{-35pt}
\\
&=& 2^{\beta}N^{\beta}\int_{2a_{1}N - J_{1}}^{N+2a_{1}N}
\cdots\int_{2a_{d}N - J_{d}}^{N+2a_{d}N} \Biggl\llvert \sum
_{l=1}^{d} u_{l}^{2} \Biggr
\rrvert ^{-\beta/2}\,\mathrm{d}u_{1}\cdots\,\mathrm{d}u_{d}\hspace*{-35pt}
\\
&\leq& CN^{\beta}\int_{w_{N}}^{2N}
\frac{1}{r^{\beta}}r^{d-1}\,\mathrm{d}r,\hspace*{-35pt}
\end{eqnarray}
where $w_{N}>0$ is the distance to $0$ of the point of coordinates
$(2a_{1}N-J_{1},\ldots, 2a_{d}-J_{d})$. The estimates found in
Lemma~\ref{lem:convsumRiem} are then straightforward: for example, in
the case $\beta<d$, an upper bound for the last quantity is $C
N^{\beta
} N^{d-\beta}= CN^{d}$. The other cases are treated in the same manner
and lead to the same desired estimate.

As far as the case $0<p\leq d$ is concerned, the particular case $p=d$
is a bit special: it corresponds to vectors $j$ with only critical
coordinates. Since in this case, each $ \llvert \frac{j_{k}}{2N} -
a_{k}\rrvert $ is either equal to $\llvert \frac{J_{k}}{2N} -
a_{k} \rrvert $ or $\llvert \frac{J_{k} +1}{2N} - a_{k}
\rrvert $ and is anyway larger than $\rho_{N, K}\geq\frac{1}{2NK}$ (where
the quantity $\rho_{N, K}$ has been defined in the beginning of this
proof), the contribution of this case to the whole sum can be bounded
by $2^{d}\cdot\frac{1}{(d \rho_{N, K}^{2})^{\beta/2}}\leq\frac{2^{d}
2^{\beta}}{d^{\beta/2}}N^{\beta}K^{\beta}= C N^{\beta}K^{\beta}$.

Let us now concentrate on the case $0<p<d$:
Then for a fixed choice of indices $(i_{1}, \ldots, i_{p})$, we have
\begin{eqnarray*}
&&\sum_{j\in\cJ_{(i_{1}, \ldots, i_{p})}}\biggl\| \frac{j}{2N} -a \biggr\|
^{-\beta
}\\
&&\qquad=\sum_{j\in\cJ_{(i_{1}, \ldots, i_{p})}}\biggl\llvert \sum
_{i=i_{1},
\ldots, i_{p}} \biggl(\frac{j_{i}}{2N} -a_{i}
\biggr)^{2} + \sum_{i\neq
i_{1}, \ldots, i_{p}} \biggl(
\frac{j_{i}}{2N} -a_{i} \biggr)^{2} \biggr\rrvert
^{-\beta/2}
\\
&&\qquad\leq\sum_{j\in\cJ_{(i_{1}, \ldots, i_{p})}}\biggl\llvert \sum
_{i\neq
i_{1}, \ldots, i_{p}} \biggl(\frac{j_{i}}{2N} -a_{i}
\biggr)^{2} \biggr\rrvert ^{-\beta/2}.
\end{eqnarray*}
But this last sum is nothing else than $\sum_{\bar j}\| \frac{\bar
j}{2N} -\bar a \|^{-\beta}$, where $\bar a$ (resp., $\bar j$) is the
vector in $[-1, 1]^{d-p}$, built upon the vector $a$ (resp., $j$) with
all its coordinates of index in $\{i_{1}, \ldots, i_{p}\}$ removed.
Since $p>0$, we see that, by induction hypothesis, that the previous
sum can be bounded by
\[
\cases{ C N^{d-p} K^{d-p}\ln N, &\quad $\mbox{if }
\beta\leq d-p,$\vspace*{2pt}
\cr
C N^{\beta},& \quad$\mbox{if } \beta> d-p$. }
\]
In particular, if $\beta\geq d$, then the contribution to \eqref
{eq:sumajp} of the sum over $0<p<d$ can be bounded by $C N^{d-p}
K^{d-p}\ln N \leq\min (C K^{d} N^{d} \ln N, C N^{\beta} )$.
If $\beta<d$, it is also straightforward to see that this contribution
is also smaller than $CN^{d}K^{d}$. The proof of Lemma~\ref
{lem:convsumRiem} follows, by induction.
\end{pf*}
\end{appendix}
\section*{Acknowledgment}
We would like to thank the referee for useful comments and suggestions.
%
%



\printaddresses


\begin{thebibliography}{40}

\bibitem{Acebron2005}
%
\begin{barticle}[auto:STB|2014/02/12|12:18:25]
\bauthor{\bsnm{Acebr{\'o}n},~\bfnm{J.~A.}\binits{J.~A.}},
\bauthor{\bsnm{Bonilla},~\bfnm{L.~L.}\binits{L.~L.}},
\bauthor{\bsnm{P{\'e}rez Vicente},~\bfnm{C.~J.}\binits{C.~J.}},
\bauthor{\bsnm{Ritort},~\bfnm{F.}\binits{F.}} \AND
\bauthor{\bsnm{Spigler},~\bfnm{R.}\binits{R.}}
(\byear{2005}).
\btitle{The Kuramoto model: A simple paradigm for synchronization phenomena}.
\bjournal{Rev. Modern Phys.}
\bvolume{77}
\bpages{137--185}.
\end{barticle}
%
\bptok{imsref}%
\endbibitem

\bibitem{22657695}
%
\begin{barticle}[mr]
\bauthor{\bsnm{Baladron},~\bfnm{Javier}\binits{J.}},
\bauthor{\bsnm{Fasoli},~\bfnm{Diego}\binits{D.}},
\bauthor{\bsnm{Faugeras},~\bfnm{Olivier}\binits{O.}} \AND
\bauthor{\bsnm{Touboul},~\bfnm{Jonathan}\binits{J.}}
(\byear{2012}).
\btitle{Mean-field description and propagation of chaos in networks of
{H}odgkin--{H}uxley and {F}itz{H}ugh--{N}agumo neurons}.
\bjournal{J. Math. Neurosci.}
\bvolume{2}
\bpages{Art. 10, 50}.
\bid{doi={10.1186/2190-8567-2-10}, issn={2190-8567}, mr={2974499}}
\end{barticle}
%
\bptok{imsref}%
\endbibitem

\bibitem{1209.4537}
%
\begin{bmisc}[auto:STB|2014/02/12|12:18:25]
\bauthor{\bsnm{Bertini},~\bfnm{L.}\binits{L.}},
\bauthor{\bsnm{Giacomin},~\bfnm{G.}\binits{G.}} \AND
\bauthor{\bsnm{Poquet},~\bfnm{C.}\binits{C.}}
(\byear{2012}).
\bhowpublished{Synchronization and random long time
dynamics for mean-field plane rotators.
Available at \arxivurl{arXiv:1209.4537}.}
\end{bmisc}
%
\bptok{imsref}%
\endbibitem

\bibitem{MR2731396}
%
\begin{barticle}[mr]
\bauthor{\bsnm{Bolley},~\bfnm{Fran{\c{c}}ois}\binits{F.}},
\bauthor{\bsnm{Guillin},~\bfnm{Arnaud}\binits{A.}} \AND
\bauthor{\bsnm{Malrieu},~\bfnm{Florent}\binits{F.}}
(\byear{2010}).
\btitle{Trend to equilibrium and particle approximation
for a weakly selfconsistent {V}lasov--{F}okker--{P}lanck equation}.
\bjournal{M2AN Math. Model. Numer. Anal.}
\bvolume{44}
\bpages{867--884}.
\bid{doi={10.1051/m2an/2010045}, issn={0764-583X}, mr={2731396}}
\end{barticle}
%
\bptok{imsref}%
\endbibitem

\bibitem{MR2280433}
%
\begin{barticle}[mr]
\bauthor{\bsnm{Bolley},~\bfnm{Fran{\c{c}}ois}\binits{F.}},
\bauthor{\bsnm{Guillin},~\bfnm{Arnaud}\binits{A.}} \AND
\bauthor{\bsnm{Villani},~\bfnm{C{\'e}dric}\binits{C.}}
(\byear{2007}).
\btitle{Quantitative concentration inequalities for empirical measures
on non-compact spaces}.
\bjournal{Probab. Theory Related Fields}
\bvolume{137}
\bpages{541--593}.
\bid{doi={10.1007/s00440-006-0004-7}, issn={0178-8051}, mr={2280433}}
\end{barticle}
%
\bptok{imsref}%
\endbibitem

\bibitem{MR2834721}
%
\begin{barticle}[mr]
\bauthor{\bsnm{Bossy},~\bfnm{Mireille}\binits{M.}},
\bauthor{\bsnm{Jabir},~\bfnm{Jean-Fran{\c{c}}ois}\binits{J.-F.}}
\AND
\bauthor{\bsnm{Talay},~\bfnm{Denis}\binits{D.}}
(\byear{2011}).
\btitle{On conditional {M}c{K}ean {L}agrangian stochastic models}.
\bjournal{Probab. Theory Related Fields}
\bvolume{151}
\bpages{319--351}.
\bid{doi={10.1007/s00440-010-0301-z}, issn={0178-8051}, mr={2834721}}
\end{barticle}
%
\bptok{imsref}%
\endbibitem

\bibitem{MR1410117}
%
\begin{barticle}[mr]
\bauthor{\bsnm{Bossy},~\bfnm{Mireille}\binits{M.}} \AND
\bauthor{\bsnm{Talay},~\bfnm{Denis}\binits{D.}}
(\byear{1996}).
\btitle{Convergence rate for the approximation of the limit law of
weakly interacting particles: Application to the {B}urgers equation}.
\bjournal{Ann. Appl. Probab.}
\bvolume{6}
\bpages{818--861}.
\bid{doi={10.1214/aoap/1034968229}, issn={1050-5164}, mr={1410117}}
\end{barticle}
%
\bptok{imsref}%
\endbibitem

\bibitem{MR1840644}
%
\begin{bbook}[mr]
\bauthor{\bsnm{Cerrai},~\bfnm{Sandra}\binits{S.}}
(\byear{2001}).
\btitle{Second Order {PDE}'s in Finite and Infinite Dimension: A
Probabilistic Approach}.
\bseries{Lecture Notes in Math.}
\bvolume{1762}.
\bpublisher{Springer},
\blocation{Berlin}.
\bid{doi={10.1007/b80743}, mr={1840644}}
\end{bbook}
%
\bptok{imsref}%
\endbibitem

\bibitem{PhysRevE.82.016205}
%
\begin{barticle}[auto:STB|2014/02/12|12:18:25]
\bauthor{\bsnm{Chowdhury},~\bfnm{D.}\binits{D.}} \AND
\bauthor{\bsnm{Cross},~\bfnm{M.~C.}\binits{M.~C.}}
(\byear{2010}).
\btitle{Synchronization of oscillators with long-range power law interactions}.
\bjournal{Phys. Rev. E (3)}
\bvolume{82}
\bpages{016205}.
\end{barticle}
%
\bptok{imsref}%
\endbibitem

\bibitem{daiPra96}
%
\begin{barticle}[mr]
\bauthor{\bsnm{Dai Pra},~\bfnm{Paolo}\binits{P.}} \AND
\bauthor{\bsnm{den Hollander},~\bfnm{Frank}\binits{F.}}
(\byear{1996}).
\btitle{Mc{K}ean--{V}lasov limit for interacting random processes in
random media}.
\bjournal{J. Stat. Phys.}
\bvolume{84}
\bpages{735--772}.
\bid{doi={10.1007/BF02179656}, issn={0022-4715}, mr={1400186}}
\end{barticle}
%
\bptok{imsref}%
\endbibitem

\bibitem{doi10.108007362999808809576}
%
\begin{barticle}[mr]
\bauthor{\bsnm{Da Prato},~\bfnm{Giuseppe}\binits{G.}} \AND
\bauthor{\bsnm{Tubaro},~\bfnm{Luciano}\binits{L.}}
(\byear{1998}).
\btitle{Some remarks about backward {I}t\^o formula and applications}.
\bjournal{Stoch. Anal. Appl.}
\bvolume{16}
\bpages{993--1003}.
\bid{doi={10.1080/07362999808809576}, issn={0736-2994}, mr={1650287}}
\end{barticle}
%
\bptok{imsref}%
\endbibitem

\bibitem{1211.0299}
%
\begin{bmisc}[auto:STB|2014/02/12|12:18:25]
\bauthor{\bsnm{Delarue},~\bfnm{F.}\binits{F.}},
\bauthor{\bsnm{Inglis},~\bfnm{J.}\binits{J.}},
\bauthor{\bsnm{Rubenthaler},~\bfnm{S.}\binits{S.}} \AND
\bauthor{\bsnm{Tanr{\'e}},~\bfnm{E.}\binits{E.}}
(\byear{2012}).
\bhowpublished{Global solvability of a networked integrate-and-fire
model of McKean--Vlasov type.
Available at \arxivurl{arXiv:1211.0299}.}
\end{bmisc}
%
\bptok{imsref}%
\endbibitem

\bibitem{MR1741805}
%
\begin{barticle}[mr]
\bauthor{\bsnm{Del Moral},~\bfnm{P.}\binits{P.}} \AND
\bauthor{\bsnm{Miclo},~\bfnm{L.}\binits{L.}}
(\byear{2000}).
\btitle{A {M}oran particle system approximation of {F}eynman--{K}ac formulae}.
\bjournal{Stochastic Process. Appl.}
\bvolume{86}
\bpages{193--216}.
\bid{doi={10.1016/S0304-4149(99)00094-0}, issn={0304-4149}, mr={1741805}}
\end{barticle}
%
\bptok{imsref}%
\endbibitem

\bibitem{MR1932358}
%
\begin{bbook}[mr]
\bauthor{\bsnm{Dudley},~\bfnm{R.~M.}\binits{R.~M.}}
(\byear{2002}).
\btitle{Real Analysis and Probability}.
\bseries{Cambridge Studies in Advanced Mathematics}
\bvolume{74}.
\bpublisher{Cambridge Univ. Press},
\blocation{Cambridge}.
\bid{doi={10.1017/CBO9780511755347}, mr={1932358}}
\end{bbook}
%
\bptok{imsref}%
\endbibitem

\bibitem{MR2674516}
%
\begin{bbook}[mr]
\bauthor{\bsnm{Ermentrout},~\bfnm{G.~Bard}\binits{G.~B.}} \AND
\bauthor{\bsnm{Terman},~\bfnm{David~H.}\binits{D.~H.}}
(\byear{2010}).
\btitle{Mathematical Foundations of Neuroscience}.
\bseries{Interdisciplinary Applied Mathematics}
\bvolume{35}.
\bpublisher{Springer},
\blocation{New York}.
\bid{doi={10.1007/978-0-387-87708-2}, mr={2674516}}
\end{bbook}
%
\bptok{imsref}%
\endbibitem

\bibitem{Gartner}
%
\begin{barticle}[mr]
\bauthor{\bsnm{G{\"a}rtner},~\bfnm{J{\"u}rgen}\binits{J.}}
(\byear{1988}).
\btitle{On the {M}c{K}ean--{V}lasov limit for interacting diffusions}.
\bjournal{Math. Nachr.}
\bvolume{137}
\bpages{197--248}.
\bid{doi={10.1002/mana.19881370116}, issn={0025-584X}, mr={0968996}}
\end{barticle}
%
\bptok{imsref}%
\endbibitem

\bibitem{Gelfand1964}
%
\begin{bbook}[mr]
\bauthor{\bsnm{Gel'fand},~\bfnm{I.~M.}\binits{I.~M.}} \AND
\bauthor{\bsnm{Vilenkin},~\bfnm{N.~Y.}\binits{N.~Y.}}
(\byear{1964}).
\btitle{Generalized Functions. {V}ol. 4: Applications of Harmonic Analysis}.
\bpublisher{Academic Press},
\blocation{New York}.
\bid{mr={0435834}}
\bptnote{check year}%
\end{bbook}
%
\bptok{imsref}%
\endbibitem

\bibitem{GLP2011}
%
\begin{bmisc}[auto:STB|2014/02/12|12:18:25]
\bauthor{\bsnm{Giacomin},~\bfnm{G.}\binits{G.}},
\bauthor{\bsnm{Lu{\c{c}}on},~\bfnm{E.}\binits{E.}} \AND
\bauthor{\bsnm{Poquet},~\bfnm{C.}\binits{C.}}
(\byear{2011}).
\bhowpublished{Coherence stability and effect of random natural
frequencies in population of coupled oscillators.
Available at \arxivurl{arXiv:1111.3581}}.
\end{bmisc}
%
\bptok{imsref}%
\endbibitem

\bibitem{GPP2012}
%
\begin{barticle}[mr]
\bauthor{\bsnm{Giacomin},~\bfnm{Giambattista}\binits{G.}},
\bauthor{\bsnm{Pakdaman},~\bfnm{Khashayar}\binits{K.}} \AND
\bauthor{\bsnm{Pellegrin},~\bfnm{Xavier}\binits{X.}}
(\byear{2012}).
\btitle{Global attractor and asymptotic dynamics in the {K}uramoto
model for coupled noisy phase oscillators}.
\bjournal{Nonlinearity}
\bvolume{25}
\bpages{1247--1273}.
\bid{doi={10.1088/0951-7715/25/5/1247}, issn={0951-7715}, mr={2914138}}
\end{barticle}
%
\bptok{imsref}%
\endbibitem

\bibitem{PhysRevE.85.066201}
%
\begin{barticle}[auto:STB|2014/02/12|12:18:25]
\bauthor{\bsnm{Gupta},~\bfnm{S.}\binits{S.}},
\bauthor{\bsnm{Potters},~\bfnm{M.}\binits{M.}} \AND
\bauthor{\bsnm{Ruffo},~\bfnm{S.}\binits{S.}}
(\byear{2012}).
\btitle{One-dimensional lattice of oscillators coupled through
power-law interactions: Continuum limit and dynamics of spatial Fourier modes}.
\bjournal{Phys. Rev. E (3)}
\bvolume{85}
\bpages{066201}.
\end{barticle}
%
\bptok{imsref}%
\endbibitem

\bibitem{1209.6035}
%
\begin{bmisc}[auto:STB|2014/02/12|12:18:25]
\bauthor{\bsnm{Hairer},~\bfnm{M.}\binits{M.}},
\bauthor{\bsnm{Hutzenthaler},~\bfnm{M.}\binits{M.}} \AND
\bauthor{\bsnm{Jentzen},~\bfnm{A.}\binits{A.}}
(\byear{2012}).
\bhowpublished{Loss of regularity for Kolmogorov equations.
Available at \arxivurl{arXiv:1209.6035}}.
\end{bmisc}
%
\bptok{imsref}%
\endbibitem

\bibitem{Jourdain1998}
%
\begin{barticle}[mr]
\bauthor{\bsnm{Jourdain},~\bfnm{B.}\binits{B.}} \AND
\bauthor{\bsnm{M{\'e}l{\'e}ard},~\bfnm{S.}\binits{S.}}
(\byear{1998}).
\btitle{Propagation of chaos and fluctuations for a moderate model
with smooth initial data}.
\bjournal{Ann. Inst. Henri Poincar\'e Probab. Stat.}
\bvolume{34}
\bpages{727--766}.
\bid{doi={10.1016/S0246-0203(99)80002-8}, issn={0246-0203}, mr={1653393}}
\end{barticle}
%
\bptok{imsref}%
\endbibitem

\bibitem{MR1311478}
%
\begin{bbook}[mr]
\bauthor{\bsnm{Krylov},~\bfnm{N.~V.}\binits{N.~V.}}
(\byear{1995}).
\btitle{Introduction to the Theory of Diffusion Processes}.
\bseries{Translations of Mathematical Monographs}
\bvolume{142}.
\bpublisher{Amer. Math. Soc.},
\blocation{Providence, RI}.
\bid{mr={1311478}}
\end{bbook}
%
\bptok{imsref}%
\endbibitem

\bibitem{Lucon2011}
%
\begin{barticle}[mr]
\bauthor{\bsnm{Lu{\c{c}}on},~\bfnm{Eric}\binits{E.}}
(\byear{2011}).
\btitle{Quenched limits and fluctuations of the empirical measure for
plane rotators in random media}.
\bjournal{Electron. J. Probab.}
\bvolume{16}
\bpages{792--829}.
\bid{doi={10.1214/EJP.v16-874}, issn={1083-6489}, mr={2793244}}
\end{barticle}
%
\bptok{imsref}%
\endbibitem

\bibitem{Malrieu2003}
%
\begin{barticle}[mr]
\bauthor{\bsnm{Malrieu},~\bfnm{Florent}\binits{F.}}
(\byear{2003}).
\btitle{Convergence to equilibrium for granular media equations and
their {E}uler schemes}.
\bjournal{Ann. Appl. Probab.}
\bvolume{13}
\bpages{540--560}.
\bid{doi={10.1214/aoap/1050689593}, issn={1050-5164}, mr={1970276}}
\end{barticle}
%
\bptok{imsref}%
\endbibitem

\bibitem{PhysRevE.66.011109}
%
\begin{barticle}[auto:STB|2014/02/12|12:18:25]
\bauthor{\bsnm{Mar{\'o}di},~\bfnm{M.}\binits{M.}},
\bauthor{\bsnm{d'Ovidio},~\bfnm{F.}\binits{F.}} \AND
\bauthor{\bsnm{Vicsek},~\bfnm{T.}\binits{T.}}
(\byear{2002}).
\btitle{Synchronization of oscillators with long range interaction:
Phase transition and anomalous finite size effects}.
\bjournal{Phys. Rev. E (3)}
\bvolume{66}
\bpages{011109}.
\end{barticle}
%
\bptok{imsref}%
\endbibitem

\bibitem{McKean1967}
%
\begin{bincollection}[mr]
\bauthor{\bsnm{McKean},~\bfnm{H.~P.}\binits{H.~P.} \bsuffix{Jr.}}
(\byear{1967}).
\btitle{Propagation of chaos for a class of non-linear parabolic equations.}
In \bbooktitle{Stochastic {D}ifferential {E}quations ({L}ecture
{S}eries in {D}ifferential {E}quations, {S}ession 7, {C}atholic
{U}niv., 1967)}
\bpages{41--57}.
\bpublisher{Air Force Office Sci. Res.},
\blocation{Arlington, VA}.
\bid{mr={0233437}}
\end{bincollection}
%
\bptok{imsref}%
\endbibitem

\bibitem{Meleard1987}
%
\begin{barticle}[mr]
\bauthor{\bsnm{M{\'e}l{\'e}ard},~\bfnm{Sylvie}\binits{S.}} \AND
\bauthor{\bsnm{Roelly-Coppoletta},~\bfnm{Sylvie}\binits{S.}}
(\byear{1987}).
\btitle{A propagation of chaos result for a system of particles with
moderate interaction}.
\bjournal{Stochastic Process. Appl.}
\bvolume{26}
\bpages{317--332}.
\bid{doi={10.1016/0304-4149(87)90184-0}, issn={0304-4149}, mr={0923112}}
\end{barticle}
%
\bptok{imsref}%
\endbibitem

\bibitem{Oelsch1984}
%
\begin{barticle}[mr]
\bauthor{\bsnm{Oelschl{\"a}ger},~\bfnm{Karl}\binits{K.}}
(\byear{1984}).
\btitle{A martingale approach to the law of large numbers for weakly
interacting stochastic processes}.
\bjournal{Ann. Probab.}
\bvolume{12}
\bpages{458--479}.
\bid{issn={0091-1798}, mr={0735849}}
\end{barticle}
%
\bptok{imsref}%
\endbibitem

\bibitem{MR779460}
%
\begin{barticle}[mr]
\bauthor{\bsnm{Oelschl{\"a}ger},~\bfnm{Karl}\binits{K.}}
(\byear{1985}).
\btitle{A law of large numbers for moderately interacting diffusion processes}.
\bjournal{Z. Wahrsch. Verw. Gebiete}
\bvolume{69}
\bpages{279--322}.
\bid{doi={10.1007/BF02450284}, issn={0044-3719}, mr={0779460}}
\end{barticle}
%
\bptok{imsref}%
\endbibitem

\bibitem{PhysRevLett.106.234102}
%
\begin{barticle}[auto:STB|2014/02/12|12:18:25]
\bauthor{\bsnm{Omelchenko},~\bfnm{I.}\binits{I.}},
\bauthor{\bsnm{Maistrenko},~\bfnm{Y.}\binits{Y.}},
\bauthor{\bsnm{H{\"o}vel},~\bfnm{P.}\binits{P.}} \AND
\bauthor{\bsnm{Sch{\"o}ll},~\bfnm{E.}\binits{E.}}
(\byear{2011}).
\btitle{Loss of coherence in dynamical networks:
Spatial chaos and chimera states}.
\bjournal{Phys. Rev. Lett.}
\bvolume{106}
\bpages{234102}.
\end{barticle}
%
\bptok{imsref}%
\endbibitem

\bibitem{PhysRevE.85.026212}
%
\begin{barticle}[auto:STB|2014/02/12|12:18:25]
\bauthor{\bsnm{Omelchenko},~\bfnm{I.}\binits{I.}},
\bauthor{\bsnm{Riemenschneider},~\bfnm{B.}\binits{B.}},
\bauthor{\bsnm{H{\"o}vel},~\bfnm{P.}\binits{P.}},
\bauthor{\bsnm{Maistrenko},~\bfnm{Y.}\binits{Y.}} \AND
\bauthor{\bsnm{Sch{\"o}ll},~\bfnm{E.}\binits{E.}}
(\byear{2012}).
\btitle{Transition from spatial coherence to incoherence in
coupled chaotic systems}.
\bjournal{Phys. Rev. E (3)}
\bvolume{85}
\bpages{026212}.
\end{barticle}
%
\bptok{imsref}%
\endbibitem

\bibitem{PhysRevE.54.R2193}
%
\begin{barticle}[auto:STB|2014/02/12|12:18:25]
\bauthor{\bsnm{Rogers},~\bfnm{J.~L.}\binits{J.~L.}} \AND
\bauthor{\bsnm{Wille},~\bfnm{L.~T.}\binits{L.~T.}}
(\byear{1996}).
\btitle{Phase transitions in nonlinear oscillator chains}.
\bjournal{Phys. Rev. E (3)}
\bvolume{54}
\bpages{R2193--R2196}.
\end{barticle}
%
\bptok{imsref}%
\endbibitem

\bibitem{Strogatz1991}
%
\begin{barticle}[mr]
\bauthor{\bsnm{Strogatz},~\bfnm{Steven~H.}\binits{S.~H.}} \AND
\bauthor{\bsnm{Mirollo},~\bfnm{Renato~E.}\binits{R.~E.}}
(\byear{1991}).
\btitle{Stability of incoherence in a population of coupled oscillators}.
\bjournal{J. Stat. Phys.}
\bvolume{63}
\bpages{613--635}.
\bid{doi={10.1007/BF01029202}, issn={0022-4715}, mr={1115806}}
\end{barticle}
%
\bptok{imsref}%
\endbibitem

\bibitem{Sznit84}
%
\begin{barticle}[mr]
\bauthor{\bsnm{Sznitman},~\bfnm{Alain-Sol}\binits{A.-S.}}
(\byear{1984}).
\btitle{Nonlinear reflecting diffusion process, and the propagation of
chaos and fluctuations associated}.
\bjournal{J. Funct. Anal.}
\bvolume{56}
\bpages{311--336}.
\bid{doi={10.1016/0022-1236(84)90080-6}, issn={0022-1236}, mr={0743844}}
\end{barticle}
%
\bptok{imsref}%
\endbibitem

\bibitem{SznitSflour}
%
\begin{bincollection}[mr]
\bauthor{\bsnm{Sznitman},~\bfnm{Alain-Sol}\binits{A.-S.}}
(\byear{1991}).
\btitle{Topics in propagation of chaos}.
In \bbooktitle{\'{E}cole D'\'{E}t\'e de {P}robabilit\'es de
{S}aint-{F}lour {XIX}---1989}.
\bseries{Lecture Notes in Math.}
\bvolume{1464}
\bpages{165--251}.
\bpublisher{Springer},
\blocation{Berlin}.
\bid{doi={10.1007/BFb0085169}, mr={1108185}}
\end{bincollection}
%
\bptok{imsref}%
\endbibitem

\bibitem{1108.2414}
%
\begin{bmisc}[auto:STB|2014/02/12|12:18:25]
\bauthor{\bsnm{Touboul},~\bfnm{J.}\binits{J.}}
(\byear{2011}).
\bhowpublished{Propagation of chaos in neural fields.
Available at \arxivurl{arXiv:1108.2414}}.
\end{bmisc}
%
\bptok{imsref}%
\endbibitem

\bibitem{MR2998591}
%
\begin{barticle}[mr]
\bauthor{\bsnm{Touboul},~\bfnm{Jonathan}\binits{J.}}
(\byear{2012}).
\btitle{Limits and dynamics of stochastic neuronal networks with
random heterogeneous delays}.
\bjournal{J. Stat. Phys.}
\bvolume{149}
\bpages{569--597}.
\bid{doi={10.1007/s10955-012-0607-6}, issn={0022-4715}, mr={2998591}}
\end{barticle}
%
\bptok{imsref}%
\endbibitem

\bibitem{MR2459454}
%
\begin{bbook}[mr]
\bauthor{\bsnm{Villani},~\bfnm{C{\'e}dric}\binits{C.}}
(\byear{2009}).
\btitle{Optimal Transport: Old and New}.
\bseries{Grundlehren der Mathematischen Wissenschaften}
\bvolume{338}.
\bpublisher{Springer},
\blocation{Berlin}.
\bid{doi={10.1007/978-3-540-71050-9}, mr={2459454}}
\end{bbook}
%
\bptok{imsref}%
\endbibitem

\bibitem{PhysRevLett.110.118101}
%
\begin{barticle}[auto:STB|2014/02/12|12:18:25]
\bauthor{\bsnm{Wainrib},~\bfnm{G.}\binits{G.}} \AND
\bauthor{\bsnm{Touboul},~\bfnm{J.}\binits{J.}}
(\byear{2013}).
\btitle{Topological and dynamical complexity of random neural networks}.
\bjournal{Phys. Rev. Lett.}
\bvolume{110}
\bpages{118101}.
\end{barticle}
%
\bptok{imsref}%
\endbibitem

\end{thebibliography}
\end{document}